\documentclass{amsart}

\usepackage{mathrsfs}
\usepackage{amscd}
\usepackage{amsmath}
\usepackage{amssymb}
\usepackage{amsthm}
\usepackage{epsf}
\usepackage{latexsym}
\usepackage{verbatim}
\usepackage[all, cmtip]{xy}
\usepackage{tikz}
\usetikzlibrary{calc}
\usetikzlibrary{positioning}
\usetikzlibrary{matrix}
\usepackage{float}
\usepackage[hidelinks]{hyperref}
\usepackage{comment}
\usepackage{enumitem}
\usepackage{subcaption}
\usepackage{tikz-cd}
\usepackage{cleveref}
\usepackage{makecell}

\tikzstyle{bsq}=[rectangle, draw, thick, minimum width=.5cm, minimum height=.5cm]
\tikzstyle{bver}=[rectangle, draw, thick, minimum width=1cm, minimum height=2cm]
\tikzstyle{bhor}=[rectangle, draw, thick, minimum width=2cm, minimum height=1cm]

\tikzstyle{divisor}=[circle,very thick,draw,scale=0.3,fill=white]
\tikzstyle{vertex}=[circle,draw,scale=0.3,fill=black]
\usetikzlibrary{patterns}

\usepackage[left=3.2cm, right=3.2cm]{geometry}

\newtheorem{theorem}{Theorem}[section]
\newtheorem{lemma}[theorem]{Lemma}

\newtheorem{corollary}[theorem]{Corollary}
\newtheorem{proposition}[theorem]{Proposition}
\newtheorem{question}[theorem]{Question}
\newtheorem{problem}[theorem]{Problem}
\newtheorem{varexample}[theorem]{Example}

\newtheorem*{theorem*}{Theorem}

\theoremstyle{definition}
\newtheorem{remark}[theorem]{Remark}
\newtheorem{notation}[theorem]{Notation}
\newtheorem{definition}[theorem]{Definition}

\newcommand{\Spec}{\mathrm{Spec}\,}
\def\AA{\mathbb{A}}

\def\ZZ{{\mathbb Z}}

\def\GG{{\mathbb G}} %

\def\PP{{\mathbb P}}

\def\VV{{\mathbb V}}

\def\cV{\mathcal{V}}

\def\fX{\mathfrak{X}}

\def\supp{\mathrm{supp}}

\newcommand{\RR}{\mathbb{R}}
\newcommand{\Rbar}{\overline{\mathbb{R}}}

\newcommand{\cL}{\mathcal{L}}
\newcommand{\cO}{\mathcal{O}}

\newcommand{\cB}{\mathcal{B}}
\newcommand{\cF}{\mathcal{F}}

\newcommand{\ord}{\operatorname{ord}}
\newcommand{\Trop}{\operatorname{Trop}}
\newcommand{\trop}{\operatorname{trop}}
\newcommand{\ddiv}{\operatorname{div}}
\newcommand{\Div}{\operatorname{Div}}

\newcommand{\PL}{\operatorname{PL}}

\newcommand{\val}{\operatorname{val}}
\newcommand{\Pic}{\operatorname{Pic}}

\newcommand{\Sym}{\operatorname{Sym}}

\newcommand{\outdeg}{\mathrm{outdeg}}

\newcommand{\relint}{\operatorname{relint}}
\newcommand{\Star}{\operatorname{Star}}

\newcommand{\BN}{\mathrm{BN}}
\newcommand{\Real}{\operatorname{Real}}

\newcommand{\ind}{\mathrm{ind}}

\newcommand{\td}[1]{\widetilde{#1}}

\newcommand{\an}{\mathrm{an}}

\newenvironment{example}{\begin{varexample}
\begin{normalfont}}{\end{normalfont}
\end{varexample}}

\begin{document}

\title{Tropical linear series and matroids}

\author[C.-W. Chang]{Chih-Wei Chang}\address{Chih-Wei Chang, Department of Mathematics, National Taiwan University, Taipei, 106, Taiwan}\email{cwchang0219@ntu.edu.tw}
\author[M. Dupraz]{Matthew Dupraz} \address{Matthew Dupraz, Institute of Mathematics, Free University of Berlin} \email{matthew.dupraz@fu-berlin.de}
\author[H. Iriarte]{Hernan Iriarte}\address{Hernan Iriarte, Department of Mathematics, University of Texas at Austin}\email{iriarte@utexas.edu}
\author[D. Jensen]{David Jensen}\address{David Jensen, Department of Mathematics, University of Kentucky
}\email{dhje223@uky.edu}
\author[D. Karp]{Dagan Karp}\address{Dagan Karp, Department of Mathematics, Harvey Mudd College
}\email{dkarp@hmc.edu}
\author[S. Payne]{Sam Payne}\address{Sam Payne, Department of Mathematics, University of Michigan}\email{sdpayne@umich.edu}
\author[J. Wang]{Jidong Wang}\address{Jidong Wang, Department of Mathematics, University of Texas at Austin}\email{jidongw@utexas.edu}

\bibliographystyle{alpha}

\begin{abstract}
We study a notion of tropical linear series on metric graphs that combines two essential properties of tropicalizations of linear series on algebraic curves: the Baker--Norine rank and the independence rank. Our main results relate the local and global geometry of these tropical linear series to the combinatorial geometry of matroids and valuated matroids, respectively. As an application, we characterize exactly when the tropicalization of the canonical linear series on a single curve is equal to the locus of realizable tropical canonical divisors determined by M\"oller, Ulirsch, and Werner. We also illustrate our results with a wealth of examples; in particular, we show that the Bergman fan of every matroid appears as the local fan of a tropical linear series on a metric graph. The paper concludes with a list of ten open questions for future investigation.
\end{abstract}

\maketitle

\setcounter{tocdepth}{1}
\tableofcontents

\section{Introduction}
\label{Sec:Intro}

\subsection{Baker--Norine rank and tropical independence} Nearly two decades ago, Baker and Norine introduced the rank of a divisor on a graph, in close analogy with the rank of a divisor on an algebraic curve, and proved the surprising and beautiful fact that it satisfies a precise analog of the Riemann--Roch Theorem \cite{BakerNorine07}.  This breakthrough inspired several new directions of research.  Baker proved the Specialization Lemma \cite{Baker08} relating ranks of divisors on curves to those on graphs via semistable degenerations, and outlined a program for relating divisor theory on graphs to the celebrated results of Brill-Noether theory on algebraic curves.  All of these results extend naturally to tropical curves, meaning metric graphs \cite{GathmannKerber08, MikhalkinZharkov08}, with specialization given by retraction to the skeleton of the Berkovich analytifications of curves over valued fields. Subsequent progress in Baker's program includes tropical proofs of the Brill--Noether and Gieseker--Petri Theorems \cite{tropicalBN, tropicalGP}. The latter introduced \emph{tropical independence} as a technique to bound the ranks of multiplication maps on algebraic linear series.

\medskip

Further applications of tropical geometry to linear series on algebraic curves and the Kodaira dimensions of moduli spaces \cite{M13, M23} use not only the Baker--Norine rank and tropical independence, but also more intricate combinatorial properties of tropicalizations of linear series rooted in linear incidence geometry. The ad hoc notion of  an abstract tropical linear series given by \cite[Definition~6.5]{M23} captures exactly the properties used for these specific applications, including two properties from linear incidence geometry that make the definition recursive. The recursive part of the definition is exceedingly difficult to verify for anything that does not arise as the tropicalization of an algebraic linear series.

\medskip

Here, we study a larger class of \emph{tropical linear series} characterized only in terms of the Baker--Norine rank and tropical independence; see \Cref{def:tropical-linear-series}. We now refer to the subclass of tropical linear series that satisfy the additional recursive properties in \cite[Definition~6.5]{M23} as \emph{strongly recursive tropical linear series}. The broader combinatorial theory of tropical linear series studied here is rich with examples associated to valuated matroids. %
It also offers insights into the geometry of tropicalizations of algebraic linear series. For instance, it had been an open problem whether strongly recursive tropical linear series are equidimensional \cite[Question~6.3(2)]{JensenPayne22}, until equidimensionality was proved more generally for all tropical linear series in the Master's thesis of the second author \cite{DuprazMasters}. In \Cref{sec:localdim}, we present a brief and streamlined version of Dupraz's original proof. %

\subsection{Tropical linear series} Let $\Gamma$ be a metric graph. A divisor $D$ on $\Gamma$ is a finite formal sum
\[
D = a_1 p_1 + \cdots + a_s p_s
\]
with $a_i \in \ZZ$ and $p_1, \ldots, p_s$ distinct points in $\Gamma$.  The \emph{multiplicity} of $D$ at $p_i$ is $a_i$ and the degree of $D$ is $\deg(D) := a_1 + \cdots + a_s$. The divisor $D$ is \emph{effective}, denoted $D \geq 0$, if its multiplicity is non-negative at every point, i.e., if $a_i \geq 0$ for all $i$. Let $\PL(\Gamma)$ denote the set of piecewise-linear functions on $\Gamma$ with integer slopes; it is an additive group with respect to addition and a tropical module with respect to scalar addition and pointwise minimum. There is an additive group homomorphism from $\PL(\Gamma)$ to the group of divisors on $\Gamma$ taking a piecewise-linear function $\varphi$ to the degree zero divisor $\ddiv(\varphi)$ whose multiplicity at a point $p$ is the sum of the incoming slopes of $\varphi$ at $p$.  Then
\[
R(D) := \{ \varphi \in \PL(\Gamma) : D + \ddiv(\varphi) \geq 0 \}
\]
is a finitely generated tropical submodule
by \cite[Corollary 9]{HMY12}.
The rank $r(D)$ is the largest integer $r$ such that, for every effective divisor $D'$ of degree $r$, there is some $\varphi \in R(D)$ such that $D - D' +  \ddiv(\varphi) \geq 0$ \cite{BakerNorine07}. This notion generalizes to tropical submodules of $R(D)$ as follows.

\begin{definition}
\label{def:submodule-rank}
Let $\Sigma \subseteq R(D)$ be a tropical submodule. The \emph{Baker--Norine rank} $r_{\BN}(\Sigma)$ is the largest integer $r$ such that, for every effective divisor $D'$ of degree $r$ on  $\Gamma$, there is some $\varphi \in \Sigma$ such that $D -D' + \ddiv(\varphi) \geq 0$.
\end{definition}

\noindent We also consider a second notion of rank for tropical submodules $\Sigma \subseteq R(D)$, based on tropical independence, as in \cite{tropicalGP, JensenPayne16}; we recall the definition of tropical independence in  \Cref{sec:independence}.

\begin{definition} \label{def:indep-rank}
Let $\Sigma \subseteq R(D)$ be a tropical submodule. The \emph{independence rank} $r_{\ind}(\Sigma)$ is the size of the largest tropically independent subset of $\Sigma$.
\end{definition}

Definitions~\ref{def:submodule-rank} and \ref{def:indep-rank} are motivated by tropicalizations of linear series on algebraic curves. Recall that a linear series on a projective algebraic curve $X$ is a pair $(D_X, V)$, where $D_X$ is a divisor and $V \subseteq H^0(X, \cO(D_X))$ is a linear subspace. We follow the common convention that the \emph{dimension} of an algebraic linear series refers to the dimension of the associated projective space $\PP(V)$, i.e., if $(D_X,V)$ is a linear series of dimension $r$ then $V$ is a vector space of dimension $r + 1$. Note that $r_{\BN}$ is an analog of $\dim \PP(V)$, whereas $r_{\ind}$ is an analog of $\dim V$.  More precisely, if $(D, \Sigma)$ is the tropicalization of a linear series $(D_X, V)$ of dimension $r$ then
\begin{equation} \label{eq:ranks}
    r_{\BN}(\Sigma) =  r \quad  \mbox{ and } \quad r_{\ind}(\Sigma) = r + 1.
\end{equation} 

\begin{remark}
To see why \eqref{eq:ranks} holds, first note that $r_{\BN}(\Sigma) \geq r$, by specialization  \cite{Baker08}, and $r_\ind(\Sigma) \leq r + 1$, because any lift of a tropically independent subset of $\Sigma$ is linearly independent in $V$. Next, observe that $|\Sigma| := \{ D + \ddiv(\varphi) : \varphi \in \Sigma \}$, 
is a polyhedral subset of $\Sym^d(\Gamma)$ of dimension at least $r_{\BN}(\Sigma)$. Finally, for any tropical submodule $\Sigma$ with polyehdral tropicalization, we have $r_\ind(\Sigma) = \dim|\Sigma| + 1$ (\Cref{cor:tropical-rank-properties-3}). A similar argument shows that $r_{\ind}(\Sigma) \geq r_{\BN}(\Sigma) + 1$ for any tropical submodule of $R(D)$ with polyhedral projectivization.
\end{remark}

\begin{definition}\label{def:tropical-linear-series}
    A \emph{tropical linear series} on a metric graph $\Gamma$ is a pair $(D,\Sigma)$ where $D$ is a divisor and $\Sigma \subseteq R(D)$ is a finitely generated submodule such that $r_{\ind}(\Sigma) = r_{\BN}(\Sigma) + 1$.
\end{definition}

\noindent As explained in \Cref{sec:proj-trop-modules}, the projectivization of a finitely generated tropical submodule $\Sigma \subseteq R(D)$ 
is naturally identified with
\[
|\Sigma| := \{ D + \ddiv(\varphi) : \varphi \in \Sigma \},
\]
which is a polyhedral subset of $\Sym^d(\Gamma)$, for $d = \deg(D)$.  The global dimension $\dim|\Sigma|$ is equal to $r_{\ind}(\Sigma) - 1$ (\Cref{prop:tropical-rank-properties-2}) and greater than or equal to $r_{\BN}(\Sigma)$. Thus, \Cref{def:tropical-linear-series} says that $\Sigma$ is a tropical linear series if its dimension is no larger than necessary, given its Baker--Norine rank.

Note that the ranks $r_{\BN}(\Sigma)$ and $r_{\ind}(\Sigma)$ are independent of the choice of divisor $D$ such that $\Sigma \subseteq \PL(\Gamma)$ is contained in $R(D)$.
When no confusion seems possible, we omit the divisor and refer to $\Sigma$ as a tropical linear series. 

\begin{theorem} \label{thm:puredim}
Let $\Sigma$ be a tropical linear series.  Then $|\Sigma|$ has pure dimension equal to $r_{\BN}(\Sigma)$.
\end{theorem}

\noindent We give a streamlined proof following the original ideas from the Master's thesis of Dupraz \cite{DuprazMasters}. An alternate approach was developed later in \cite{AGG}. Here, we also state and prove a much stronger theorem relating the local structure of $|\Sigma|$ to Bergman fans of matroids (\Cref{thm:local-structure}). 

Since $|\Sigma|$ has pure  dimension $r = r_{\BN}(\Sigma)$, and we follow the convention that the dimension of a linear series is its projective dimension, we say that $\Sigma$ is a \emph{tropical linear series of dimension $r$}. Thus, the tropicalization of a linear series of dimension $r$ is a tropical linear series of dimension $r$.

\begin{remark} \label{rem:BN-bound}
Any metric graph $\Gamma$ of first Betti number $g$ is the skeleton of a smooth projective curve of genus $g$. See, e.g., \cite[Appendix~B]{Baker08}. Therefore, the Brill--Noether theorem tells us that there is a divisor $D$ of degree $d$ and a tropical linear series $\Sigma \subseteq R(D)$ of dimension $r$ whenever $(r+1)(g-d+r) \geq g$. This bound is sharp, by \cite{tropicalBN}.
\end{remark}

\subsection{Matroidal linear series}

This work focuses on the relations between tropical linear series and matroids. If $K$ is a valued field and $V \subseteq \AA^n_K$ is a linear subvariety, then $\Trop(V) \subseteq \Rbar^n$ is a tropical submodule.  Abstracting the combinatorial properties of such submodules gives the cryptomorphic axiomatization of valuated matroids in terms of valuated covectors. Since we use this axiomatization of valuated matroids throughout, we simply refer to a submodule of $\Rbar^n$ satisfying the valuated covector axioms as a valuated matroid.  See \Cref{sec:matroids} for details and references.

\begin{theorem} \label{thm:realizablematroidal}
If $\Sigma$ is the tropicalization of a linear series of dimension $r$ then there is a realizable valuated matroid $\cV$ of rank $r + 1$ and a surjective homomorphism of tropical modules $\cV \twoheadrightarrow \Sigma$.
\end{theorem}

\noindent \Cref{thm:realizablematroidal} first appeared implicitly in the proof that the tropicalization of any linear series is finitely generated as a tropical module \cite[Proposition~6.4]{M23}. For the reader's convenience, we give a self-contained proof in \Cref{sec:realizablematroidal}.

A subset $\Sigma \subseteq\Rbar^n$ is a valuated matroid if and only if it contains $(\infty, \ldots, \infty)$ and the projectivization $\big(\Sigma \smallsetminus (\infty, \ldots, \infty)\big)/\RR$ is a tropical linear space. In particular, tropical linear spaces are another cryptomorphic incarnation of valuated matroids \cite{SpeyerLinear,BrandtEurZhang21}.  Here, we focus on valuated matroids, because we rely heavily on the tropical module structure, which is lost through projectivization.

\begin{definition} \label{def:MLS}
We say that a tropical linear series $(D, \Sigma)$ of dimension $r$ is a \emph{matroidal linear series} if $\Sigma$ is the image of a tropical module homomorphism from a valuated matroid of rank $r+1$.
\end{definition}

\begin{definition} \label{def:realizable}
 A tropical linear series is \emph{realizable} if it is the tropicalization of a linear series.
 \end{definition}

\noindent If $\Sigma$ is the image of a valuated matroid of rank $r + 1$, then $r_{\ind}(\Sigma) \leq r + 1$. Thus, a tropical submodule $\Sigma \subseteq R(D)$ is a matroidal linear series of dimension $r$ if and only it is the image of a valuated matroid of rank $r+1$ under a homomorphism of tropical modules and $r_{\BN}(\Sigma) = r$. By \Cref{thm:realizablematroidal}, every realizable tropical linear series is matroidal.

\begin{remark}
The surjection $\cV \twoheadrightarrow \Sigma$ is not unique and is not part of the data of a matroidal linear series. We refer to such a surjection as a \emph{parametrization}. \Cref{thm:realizablematroidal} says that the tropicalization of any linear series has a  parametrization by a realizable valuated matroid.  In Sections~\ref{sec:local-matroidal} and \ref{sec:Cartwright}, %
we give examples of matroidal linear series that are not parametrized by any realizable valuated matroid and hence are not tropicalizations of linear series on algebraic curves.
\end{remark}

\begin{remark}
It is immediate from the definitions that the restriction of a tropical or matroidal linear series to a subgraph is a tropical or matroidal linear series, respectively. %
We can therefore study tropical linear series on a complicated graph by considering their restrictions to smaller and simpler subgraphs, such as edges and loops. Such arguments are essential for applications, as in \cite{M13, M23}, which involve studying  tropicalizations of linear series on a chain of loops with bridges by systematically considering the restrictions to each loop and bridge.
\end{remark}

\begin{question} \label{q:tropical-implies-matroidal}
Is every tropical linear series a matroidal linear series?
\end{question}

\noindent We do not know the answer to this basic question. However, we show that the structure of every tropical linear series is closely related to a matroid locally near every point in an open dense subset of nondegenerate divisors.

\subsection{Nondegenerate divisors and local matroids}

Let $\Sigma \subseteq R(D)$ be a tropical linear series of degree $d$.  
We use two semi-continuous integer-valued functions on $|\Sigma|$ to characterize nondegenerate divisors. Let $\supp(D)$ be the finite set of points in $\Gamma$ that appear with nonzero multiplicity in $D$. The first function we consider is $\#\supp$, which is lower semi-continuous on $\Sym^d(\Gamma)$. The second is the \emph{valence 1 degree}, which is upper semi-continuous and takes a divisor to the sum of its multiplicities at the valence 1 points of $\Gamma$.

\begin{definition} \label{def:nondegenerate}
Let $\Sigma$ be a tropical linear series. A divisor $D \in |\Sigma|$ is \emph{nondegenerate} if $\# \supp$ has a local maximum in $|\Sigma|$ at $D$ and the valence 1 degree has a local minimum in $|\Sigma|$ at $D$.
\end{definition} %

\noindent A nondegenerate divisor $D$ has an open neighborhood $U$ in $|\Sigma|$ such that the restriction of $\#\supp$  to $U$ achieves its maximum at $D$, and the restriction of the valence $1$ degree to $U$ achieves its minimum at $D$. The locus of nondegenerate divisors is open and dense in $|\Sigma|$.

We study the local structure of $|\Sigma|$ near a nondegenerate divisor $D$, as follows. Roughly speaking, the star of a point $x$ in a polyhedral space $\Delta$, denoted $\Star(x)$, is the cone-shaped space that one sees looking out into $\Delta$ from $x$.  More precisely, $\Star(x)$ is the space of germs of linear maps $[0,\epsilon] \to \Delta$ that take $0$ to $x$. A choice of triangulation of $\Delta$ induces a simplicial fan structure on $\Star(x)$, but the support of $\Star(x)$ is independent of such choices. Our next theorem identifies $\Star(D)$ with the Bergman fan of a matroid, when $D \in |\Sigma|$ is nondegenerate.

For $\varphi \in \Sigma$, we define $\varphi_{\min} \subseteq \Gamma$ to be the \emph{minimizer} of $\varphi$, i.e.,
\[
\varphi_{\min} := \{ p \in \Gamma : \varphi(p) = \min_{x \in \Gamma} \varphi(x) \}.
\]
If $\Sigma \subseteq R(D)$ and $D \in |\Sigma|$ is nondegenerate, then $\varphi_{\min}$ is a union of closures of connected components of $\Gamma \smallsetminus \supp(D)$ (\Cref{Prop:NoIsolatedPoints}). 
Let $E$ be the set of connected components of $\Gamma \smallsetminus \supp(D)$, and 
\[
F_\varphi := \{ S \in E : S \not \subseteq \varphi_{\min} \}.
\]
If $D$ is nondegenerate, evaluation at one chosen point in each connected component of $\Gamma \smallsetminus \supp(D)$ gives local coordinates in a neighborhood of $D$ and induces an embedding $\Star(D) \subseteq \RR^{E}/(1, ..., 1)$.

\begin{theorem} \label{thm:local-structure}
Let $\Sigma \subseteq R(D)$ be a tropical linear series of dimension $r$ such that $D$ is nondegenerate in $|\Sigma|$. Then
$
\{ F_\varphi : \varphi \in \Sigma \} \cup \{E\}
$
is the set of flats of a matroid $M_{\Sigma}$ of rank $r + 1$ on $E$, and the image of the embedding $\Star(D) \subseteq \RR^{E}/(1, ..., 1)$ is the support of the Bergman fan of $M_{\Sigma}$.
\end{theorem}

\noindent We say that $M_{\Sigma}$ is the \emph{local matroid} of $\Sigma$. It is an invariant of the submodule $\Sigma \subseteq\PL(\Gamma)$ or, equivalently, an invariant of the pair $(|\Sigma|, D)$.  \Cref{thm:local-structure} identifies the local geometry of $|\Sigma|$ at a nondegenerate divisor $D$ with the Bergman fan of $M_\Sigma$. \Cref{thm:puredim} is an immediate corollary, because the nondegenerate divisors are dense in $|\Sigma|$. Note that we do not choose a preferred fan structure on the support of the Bergman fan. Throughout, whenever we refer to the Bergman fan of a matroid, we are primarily interested in its support.

When one varies the nondegenerate divisor in $|\Sigma|$, the local matroid varies as well. If $\Sigma \subseteq R(D)$ and $D' = D + \ddiv(\varphi') \in |\Sigma|$ then $\Sigma' := \{ \varphi' : \varphi' + \varphi \in \Sigma \}$ is a submodule of $R(D')$ and the map $\varphi' \mapsto \varphi + \varphi'$ induces the identity $|\Sigma'| = |\Sigma|$ as subsets of $\Sym^d(\Gamma)$. The local matroids $M_{\Sigma'}$ and $M_\Sigma$ have the same rank, and they have a common restriction to the set of components of $\Gamma \smallsetminus \supp(D)$ on which $\varphi'$ is minimized. %

The local matroids associated to $(|\Sigma|, D)$ for fixed $\Sigma$ and varying $D$ are closely analogous to the initial matroids $\cV^w$ of a fixed valuated matroid $\cV \subset \Rbar^n$ with respect to a weight vector $w \in \RR^n / (1,...,1)$.  Indeed, if $w \in |\cV|$ then $\Star(w)$ is the Bergman fan of $\cV^w$.  Also, we define the local matroid not only when $D \in |\Sigma|$ is nondegenerate but, more generally, whenever $\Sigma$ has \emph{big minimizers}, meaning that every minimizer $\varphi_{\min}$, for $\varphi \in \Sigma$, contains a connected component of $\Gamma \smallsetminus \supp(D)$. \Cref{thm:local_matroid} shows that the local matroid associated to $(|\Sigma|, D)$ is loopless if and only if $D \in |\Sigma|$, just as the initial matroid $\cV^w$ is loopless if and only if $w \in |\cV|$.

\subsection{Tropicalizations of canonical linear series}

As an application, we use our theory of local matroids to study tropicalizations of canonical linear series. For simplicity,  suppose $\Gamma$ is a trivalent graph with first Betti number $g$ and no bridges. Let $K_\Gamma$ be the sum of its vertices. Specialization from the canonical linear series of a curve with skeleton $\Gamma$ shows that $R(K_\Gamma)$ contains a tropical linear series of dimension $g-1$. If $\Sigma$ is such a tropical linear series then $|\Sigma|$ contains $K_\Gamma$ as a nondegenerate divisor, and our \Cref{thm:local_matroid} implies that the local matroid at $K_\Gamma$ is the cographic matroid. See   Example~\ref{ex:canonical}. This casts new light on an observation of Haase, Musiker, and Yu \cite[Theorem~25]{HMY12}.

M\"oller, Ulirsch, and Werner used the space of multiscale differentials \cite{BCGGM17} to determine the realizable locus $\Real(|K_\Gamma|) \subset |K_\Gamma|$, the space of divisors that can be realized as the tropicalization of an effective canonical divisor on a curve over a valued field with skeleton $\Gamma$ in equicharacteristic zero \cite{MUW21}. However, they left open the problem of understanding the possibilities for the tropicalization of the canonical linear series on a single curve. This is nontrivial; we give an example where $\Real(|K_\Gamma|)$ has dimension greater than $g-1$, and hence cannot be the tropicalization of any single canonical linear series. We use our theory of local matroids to show that the dimension is the only obstruction to realizing every divisor in $\Real(|K_\Gamma|)$ on every curve with skeleton $\Gamma$.

\begin{theorem} \label{thm:realizable}
Let $X$ be a curve over a nonarchimedean field of equicharacteristic zero with skeleton $\Gamma$. Then $\Trop(|K_X|) = \mathrm{Real}(|K_\Gamma|)$ if and only if $\mathrm{Real}(|K_\Gamma|)$ has dimension $g-1$.
\end{theorem}

\subsection{Local matroids and realizability}

For matroidal linear series, we prove a precise technical result relating the local matroids at nondegenerate divisors to the initial matroids of the parametrizing matroid \Cref{thm:local-is-initial}.  This makes the local matroids at nondegenerate divisors a key obstruction to realizability.

\begin{theorem} \label{thm:local-realizable}
If $\Sigma \subseteq R(D)$ is the tropicalization of a linear series and $D \in |\Sigma|$ is nondegenerate then the local matroid $M_{\Sigma}$ is realizable.
\end{theorem}

\begin{remark} Since $|\Sigma|$ is a polyhedral complex of pure dimension  $r$, there is an open dense set of points where $|\Sigma|$ is locally isomorphic to $\RR^r$ and hence locally looks like the Bergman fan of the Boolean matroid of rank $r+1$. \Cref{thm:local-structure} gives much stronger and more precise information. 
The local matroids that are not Boolean are essential for our applications.
\end{remark}

\begin{theorem} \label{thm:every-matroid}
Every loopless matroid is isomorphic to the local matroid of a tropical linear series at a nondegenerate divisor.
\end{theorem}

\noindent   More precisely, every loopless matroid is the local matroid of a matroidal linear series at a nondegenerate divisors on an interval and also on a loop. See \Cref{sec:local-matroidal}.  Together with \Cref{thm:local-realizable}, this provides a wealth of examples of non-realizable matroidal linear series.

\medskip

In \Cref{sec:Cartwright}, we revisit a class of rank 2 divisors studied by Cartwright. Let $M$ be a simple rank 3 matroid, let $\Gamma_M$ be the bipartite incidence graph of its flats of rank 1 and 2, and $D_M$ the sum of the vertices corresponding to rank 1 flats.  Then $r_{\BN}(D_M) = 2$ and $D_M$ is the tropicalization of a rank 2 divisor if and only if $M$ is realizable  \cite{Cartwright15}. We show that if $\Sigma \subseteq R(D_M)$ is a tropical linear series of dimension 2 then $D_M$ is contained in $|\Sigma|$ and nondegenerate, and the local matroids $M_{\Sigma}$ that occur in this way are exactly the adjoints of $M$ (Theorems~\ref{thm:local-matroid-adjoint} and \ref{thm:adjoint-and-linear-series}). As a consequence, we obtain a new proof that $D_M$ is not the tropicalization of a rank 2 divisor when $M$ is not realizable (\Cref{cor:Cartwright-not-realizable}).

\subsection{Linear incidence geometry and recursive structures} \label{sec:inductiveprop}

We conclude the introduction with a brief discussion of how strongly recursive tropical linear series arise from properties of linear incidence geometry. We also explain our choice to study the broader class of tropical linear series.

Let $(D_X, V)$ be a linear series of dimension $r$.  For $k < r$, any size $k + 1$ subset of $V$ is contained in a linear subseries $V' \subseteq V$ of dimension $k$. Likewise, if $V'$ and $V''$ are subseries of codimension $c_1$ and $c_2$ in $V$, with $c_1 + c_2 \leq r$, their intersection contains a subseries of codimension $c_1 + c_2$. Thus, if $(D, \Sigma)$ is the tropicalization of a linear series of dimension $r$, any subset of $\Sigma$ of size $k +1$ is contained in the tropicalization of a linear subseries of dimension $k$. Moreover, any two subsets of size $r$ are contained in tropicalizations of linear series of dimension $r -1$ whose intersection contains the tropicalization of a linear series of dimension $r -2$. The following is a restatement of \cite[Definition~6.5]{M23} in the terminology of the present paper.

\begin{definition}
A strongly recursive tropical linear series $\Sigma \subseteq R(D)$ of dimension $r$ is a  tropical linear series of dimension $r$ such that
\begin{enumerate}
    \item Any size $r$ subset $S \subseteq\Sigma$ is contained in a strongly recursive tropical linear subseries $\Sigma_S \subseteq\Sigma$ of dimension $r-1$, and
    \item Given any two size $r$ subsets $S$ and $S'$ of $\Sigma$, the dimension $r-1$ strongly recursive tropical linear subseries $\Sigma_S$ and $\Sigma_{S'}$ containing them can be chosen so that $\Sigma_S \cap \Sigma_{S'}$ contains a strongly recursive tropical linear subseries of dimension $r-2$. 
\end{enumerate}
\end{definition}
\noindent This definition builds in exactly the basic properties from linear incidence geometry that were needed for applications to the Kodaira dimensions of $M_{22}$ and $M_{23}$ in \cite{M23}. A straightforward specialization argument shows that any tropicalization of a linear series of dimension $r$ is a strongly recursive tropical linear series of dimension $r$. However, the recursive nature of these incidence properties makes them exceedingly difficult to verify for any combinatorial examples that do not arise as tropicalizations of algebraic linear series. Other basic properties of linear incidence geometry were omitted from that definition, even though they hold in all vector spaces and hence in all tropicalizations of linear series, simply because they were not needed for applications to the Kodaira dimensions of $M_{22}$ and $M_{23}$.  See \cite{FominPylyavskyy23} for a modern perspective on the complexities of linear incidence geometry. 

Recent work in valuated matroid theory shows that tropical analogs of such basic incidence geometry statements can fail in tropical linear spaces. In particular, the rank 4 Vam\'os matroid does not have the Levi intersection property and hence its associated tropical linear space contains 3 points that are not contained in any tropical linear subspace of codimension 1 \cite[Theorem~E]{Wang24}. Using this, we give a tropical linear series of dimension 3 on the interval and a subset of size 3 that is not contained in any tropical linear subseries of dimension 2. See \Cref{ex:vamos-series}.

Likewise, there is a relaxation of the V\'amos matroid whose Bergman fan contains two tropical linear spaces of codimension 1 whose intersection does not contain any tropical linear space of codimension 2 \cite[Example~6.14]{Wang24}. Using this, we give an example of a tropical linear series with two codimension 1 tropical linear subseries whose intersection does not contain any tropical linear subseries of codimension 2. See \Cref{ex:vamos-relaxiation-series}.

\Cref{ex:vamos-series} shows that non-realizable tropical linear series are not strongly recursive in general.  By relaxing our definition of tropical linear series to include such objects, we get a cleaner and simpler combinatorial theory, worthy of investigation in its own right. For applications to algebraic geometry, nothing is lost in this way; using specialization, one can appeal to the properties that follow from linear incidence geometry whenever one restricts attention to the realizable case.

\begin{remark}
 
This paper supersedes the preprint \cite{JensenPayne22} and incorporates ideas and results from Dupraz's Master's thesis \cite{DuprazMasters}, which will not be published separately.
\end{remark}

\medskip

\noindent {\textbf{Acknowledgments.}} The work of DJ was supported in part by NSF grant DMS--2054135. The work of SP was supported in part by NSF grant DMS--2542134 and a visit to the Institute for Advanced study supported by the Charles Simonyi Endowment.  

We are grateful to M. Baker, D. Cartwright, C. Haase, Y. Luo, D. Maclagan, M.~Mayo~Garc\'ia, and J. Yu for helpful comments on an earlier version of this paper.

\section{Preliminaries} \label{Sec:Prelim}

\subsection{Metric graphs} Roughly speaking, a metric graph is a finite graph $G$, possibly with loops and multiple edges, with a positive real length $\ell(e)$ assigned to each edge $e$ in $G$. Subdividing an edge does not change the resulting metric space, and we consider the metric graphs associated to two such data $(G, \ell)$ and $(G', \ell')$ to be isomorphic whenever the resulting metric spaces are isometric. 

To make this more precise, let $S(n, \epsilon)$ denote the star-shaped set of valence $n$ and radius $\epsilon >0$. It is the pointed length metric space obtained from $n$ copies of the interval $[0,\epsilon]$ by identifying all $n$ copies of the point $0$, and labeling the resulting distinguished point $*$.

\begin{figure}[ht]
    \centering
    \begin{tikzpicture}[scale=1.5]
        \draw (0,0)--(0:1cm)
            (0,0)--(72:1cm)
            (0,0)--(72*2:1cm)
            (0,0)--(72*3:1cm)
            (0,0)--(72*4:1cm);
    \end{tikzpicture}
    \caption{A star-shaped set of valence 5}
\end{figure}
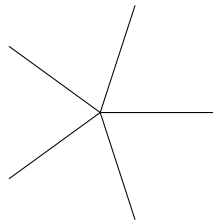

\begin{definition}\label{def:tropical-curve}
    A \emph{metric graph} is a compact connected length metric space in which each point $p$ has a neighborhood $U_p$ with a pointed isometry to a star-shaped set $S(n_p, \epsilon_p)$.
\end{definition}

\noindent By a pointed isometry, we mean an isometry $U_p \xrightarrow{\sim} S(n_p, \epsilon_p)$ that takes $p$ to $*$. Note that the positive integer $n_p$ is independent of the choice of star-shaped neighborhood; it is called the \emph{valence} of $p$. Since a metric graph is compact, it has only finitely many points $p$ of valence $n_p \neq 2$.

\subsection{Models of metric graphs}
Every metric graph can be constructed from a finite graph with a length function on edges, as follows. Let $G$ be a finite graph, possibly with loops and multiple edges, and let $\ell \colon E(G) \to \mathbb{R}_{> 0}$ be a function that assigns a positive real length to each edge.  This data gives rise to a metric graph $\Gamma = \Gamma(G, \ell)$, which is constructed from the disjoint union of the vertex set $V(G)$ and $\bigsqcup_{e \in E(G)} [0,\ell(e)]$ by choosing an orientation of each edge and identifying the endpoints of $[0,\ell(e)]$ with the vertices incident to $e$, accordingly. The data of $(G, \ell)$, plus an isometry $\Gamma(G,\ell) \to \Gamma$, is a \emph{model} of $\Gamma$. Note that every metric graph $\Gamma$ has a model, but these models are never unique, because any edge can be subdivided. In particular, the vertex set $V(G)$ contains the set of points $p \in \Gamma(G, \ell)$ of valence $n_p \neq 2$, but it may be strictly larger. 

Starting from an arbitrary model of $\Gamma$ and subdividing at $p$ if needed, we can construct a model $(G,\ell)$ of $\Gamma$ in which $p$ is a vertex of $G$.  Then $n_p$ is the number of half-edges of $G$ incident to $p$.

Every metric graph has a model $(G, \ell)$ in which the underlying graph $G$ is \emph{simple}, without loops or multiple edges. One can construct such a model from an arbitrary model by subdividing every edge. Also, any two models $(G, \ell)$ and $(G', \ell')$ of $\Gamma$ have a common refinement, obtained by subdividing at all points of $\Gamma$ that are vertices of $G'$ but not vertices of $G$, or vice versa.

\subsection{Tangent directions in metric graphs} 

Following \cite[Section~1.4]{BPR13}, we say that the \emph{tangent directions} of $\Gamma$ at $p$ are the germs of isometric embeddings $\alpha \colon [0,\epsilon] \to \Gamma$ with $\alpha(0) = p$.  Note that the number of tangent directions at $p$ is equal to the valence $n_p$. %

\subsection{Divisors on metric graphs}

The divisor group $\Div(\Gamma)$ is the free abelian group on the points of $\Gamma$. Its elements are \emph{divisors}. The degree of a divisor $D = a_1 p_1 + \cdots + a_r p_r$ is the sum of its coefficients $\deg(D) := a_1 + \cdots + a_r$. A divisor $D$ is \emph{effective}, denoted $D \geq 0$, if $a_i \geq 0$ for all $i$.

Let $\PL(\Gamma)$ denote the group of continuous piecewise-linear functions with integer slopes on $\Gamma$. In other words, a function $\varphi \colon \Gamma \to \RR$ is in $\PL(\Gamma)$ 
if it is continuous and there is some model $(G, \ell)$ of $\Gamma$ such that the restriction of $\varphi$ to each edge is affine linear with integer slope. 

\begin{definition}
Let $\zeta = [\alpha]$ be a tangent direction at $p$.  The \emph{outgoing slope} of $\varphi \in \PL(\Gamma)$ along $\zeta$, denoted $s_\zeta(\varphi)$, is the derivative of $\varphi \circ \alpha$ at $0$. The \emph{incoming slope} of $\varphi$ at $p$ is $-s_\zeta(\varphi)$.
\end{definition}

Let $T_p(\Gamma)$ denote the set of tangent directions of $\Gamma$ at $p$.  The order of $\varphi \in \PL(\Gamma)$ at $p$ is the sum of the incoming slopes
\[
\ord_p(\varphi) := - \sum_{\zeta \in T_p(\Gamma)} s_\zeta(\varphi). 
\]
Note that $\ord_p(\varphi) = 0$ for all but finitely many $p$ and hence
\[
\ddiv(\varphi) := \sum_p \ord_p(\varphi) \cdot  p
\]
is a divisor. The divisors of the form $\ddiv(\varphi)$ for some $\varphi \in \PL (\Gamma)$ are called \emph{principal}, and $$\ddiv(\varphi + \varphi') = \ddiv(\varphi) + \ddiv(\varphi'),$$  so the principal divisors are a subgroup of $\Div(\Gamma)$. Two divisors $D$ and $D'$ in $\Div(\Gamma)$ are equivalent, denoted $D \sim D'$,  if $D - D'$ is principal. Note that the degree of any principal divisor is 0, and hence any two equivalent divisors have the same degree.

\subsection{Complete linear systems} \label{sec:metric} The complete linear system of $D \in \Div(\Gamma)$ is the set of effective divisors equivalent to $D$
\[
|D| := \{ D' \sim D : D' \geq 0 \}.
\] 
We recall, following \cite[Section 4]{HMY12}, that $|D|$ carries a canonical piecewise-linear structure, as follows.

The space of effective divisors of degree $D$ on $\Gamma$ is naturally identified with 
\[
\Sym^d(\Gamma) := \Gamma^d / S_d.
\]
Here, $S_d$ is the symmetric group acting by permuting the factors of $\Gamma^d$.  The topology on $\Sym^d(\Gamma)$, as a quotient of $\Gamma^d$, is compatible with a natural piecewise-linear structure, defined as follows. 

Choose a model $(G,\ell)$ of $\Gamma$ without loops or multiple edges, and choose an orientation of each edge. The relative interior of a face of $\Sym^d(\Gamma)$ corresponds to the following data:
\begin{itemize}
\item For each vertex $v$, a  nonnegative integer $m_v$;
\item For each edge $e$, a finite sequence of positive integers $(m_{e,1}, \ldots, m_{e,k_e})$;
\end{itemize}
subject to the condition
\[
\sum_v m_v + \sum_{e,i} m_{e,i} = d.
\]
We allow the possibility that $k_e = 0$, i.e., the sequence of positive integers associated to an edge $e$ may be empty. A divisor $D$ is in the relative interior of the given face if and only if it contains each vertex $v$ with multiplicity $m_v$ and exactly $k_e$ distinct points in the interior of $e$ with multiplicities $(m_{e,1}, \ldots, m_{e,k_e})$, respectively, where the points are ordered with respect to the chosen orientation on $e$. The relative interior of a face of $\Sym^d(\Gamma)$ is a \emph{combinatorial type} of divisor of degree $d$, and the subspace of divisors of fixed combinatorial type is the interior of a product of simplices.

For any divisor $D$ of degree $d$, there is a natural polyhedral structure on $|D|$ in which the relative interior of a face is composed of divisors $D' = \ddiv(\varphi) + D$ of a fixed combinatorial type and such that $\varphi$ has fixed slopes at the starting point of each half-edge (see \cite[Section~4]{HMY12}).

Note that $|D| \subseteq \Sym^d(\Gamma)$ is a compact, connected Hausdorff space. Its topology has an alternate metric description in terms of $\PL(\Gamma)$, as follows. Let $\|\cdot\|_\infty$ be the function on $\PL(\Gamma)$ given by
\[\|\varphi\|_\infty = \sup(\varphi) - \inf(\varphi).\]
For any $D_1, D_2 \in |D|$, there is some $\varphi \in \PL(\Gamma)$, well-defined up to an additive scalar, such that
$D_1 - D_2 = \ddiv(\varphi)$. The function $\| \cdot \|_\infty$ is invariant under scalar addition, and we define a metric by %
\[
    d_\infty(D_1, D_2) = \|\varphi\|_\infty.
\]
The metric topology agrees with the subspace topology on $|D| \subseteq\Sym^d(\Gamma)$ \cite[Proposition B.1]{Luo18}.

\subsection{Tropical modules}

Let $\Rbar := \RR \cup \{\infty\}$ be the tropical semifield, with addition and multiplication given respectively by
\[
a \oplus b := \min\{a, b\} \quad \mbox{ and } \quad
        a \odot b := a + b.
        \]
Note that $\infty$ is the zero element, i.e., the identity for addition and the absorbing element for multiplication. A \emph{tropical module} $\Sigma$ is a semimodule over $\Rbar$.
A tropical module is \emph{trivial} if it is equal to $\{ \infty \}$.

Let $\Sigma$ be a tropical module. The tropical submodule generated by a subset $S \subseteq \Sigma$, denoted $\langle S \rangle$, is the set of finite tropical linear combinations
\[
a_1 \odot \varphi_1 \oplus \cdots \oplus a_n \odot \varphi_n
\]
with $a_i \in \Rbar$ and $\varphi_i \in S$. It is the smallest tropical submodule of $\Sigma$ that contains $S$. A tropical module $\Sigma$ is \emph{finitely generated} if it is equal to $\langle S \rangle$ for some finite subset $S \subseteq \Sigma$. For any divisor $D$, the tropical module $R(D)$ is finitely generated, and the minimal generating set is essentially unique up to tropical scaling \cite[Corollary~9]{HMY12}. 

\begin{example}
For any set $E$, the real-valued functions $\RR^E$, together with an additional zero element denoted $\infty$, %
is a tropical module with operations defined pointwise:
\[
(\varphi \oplus \varphi') (x) = \min\{ \varphi(x), \varphi'(x) \} \quad \mbox{ and } \quad (a \odot \varphi) (x) = \varphi(x) + a
\]
for $x \in E$ and $a \in \RR$.

Note that $\PL(\Gamma) \subseteq \RR^\Gamma$ is preserved by these operations, and hence $\PL(\Gamma) \cup \{ \infty \}$ is a tropical module.  Similarly, for any divisor $D \in \Div(\Gamma)$, define 
\[
R(D) := \{ \varphi \in \PL(\Gamma) : D + \ddiv(\varphi) \geq 0\}.
\] 
Then $\big(R(D) \cup \{\infty\}\big) \subseteq \big(\PL(\Gamma) \cup \{\infty\}\big)$ is a tropical submodule \cite[Lemma~4]{HMY12}.
\end{example}

\begin{notation}Throughout, we write the tropical module operations on $\PL(\Gamma) \cup \{\infty \}$ and its tropical submodules classically rather than tropically, i.e., we write $\min\{\varphi, \varphi' \}$ for the tropical sum and $\varphi + a$ for the product with a nonzero scalar $a$. When no confusion seems possible, we omit the zero elements (denoted $\infty$), and refer to $R(D)$ and $\PL(\Gamma)$ as tropical modules.
\end{notation}

\subsection{Tropical independence} \label{sec:independence} 
A tropical linear combination of $\{\varphi_0, \ldots, \varphi_r\} \subseteq \PL(\Gamma)$ is an expression of the form
\begin{equation}\label{eq:vartheta}
\theta = \min \{ \varphi_0 + a_0, \ldots, \varphi_r + a_r \},
\end{equation}
where $a_0, \ldots, a_r$ are real numbers.  Note that we consider the tuple $(a_0, \ldots, a_r)$ to be part of the data of the tropical linear combination, even though different tuples may give rise to the same pointwise minimum $\theta$.  

\begin{definition}
The set $\{\varphi_0, \ldots, \varphi_r\} \subseteq \PL(\Gamma)$ is \emph{tropically dependent} if there is a tropical linear combination $\theta$ such that the minimum in \eqref{eq:vartheta} is achieved at least twice at every point in $\Gamma$.
\end{definition}

Such a tropical linear combination is a \emph{tropical dependence}.  In other words, $\theta$ is a tropical dependence if and only if
\begin{equation}\label{eq:tropicaldepend}
\theta = \min_{j \neq i} \{\varphi_j + a_j \} \mbox{ \ for each \ } 0 \leq i \leq r.
\end{equation}
If no such tropical dependence exists, then $\{ \varphi_0, \ldots, \varphi_r\}$ is \emph{tropically independent}.

\begin{example}
Consider the three functions on the interval pictured in \Cref{fig:trop-dep}. Suppose that $\varphi_i$ all have infinum $0$, then $\min\{\varphi_1, \varphi_2, \varphi_3\}$ is a tropical dependence. Note however that we cannot write any of the three functions as a tropical linear combination of the other two. 

    \begin{figure}[ht]
        \centering
        \begin{tikzpicture}
            \draw (-0.5, 2) node {{$\varphi_1$}};
            \draw (-0.5, 1) node {{$\varphi_2$}};
            \draw (-0.5, 0) node {{$\varphi_3$}};
            \draw[shift={(0, 2)}] (0, 0) -- (0.5, 0.5) -- (1, 0) -- (3, 0);
            \draw[shift={(0, 1)}] (0, 0) -- (1, 0) -- (1.5, 0.5) -- (2, 0) -- (3, 0);
            \draw (0, 0) -- (2, 0) -- (2.5, 0.5) -- (3, 0);
        \end{tikzpicture}
        \caption{Three tropically dependent functions on an interval.}
        \label{fig:trop-dep}
    \end{figure}
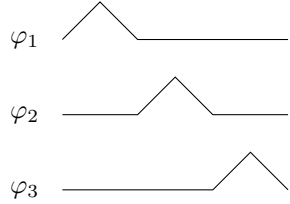
\end{example}

\begin{definition}
    A \emph{certificate of independence} for a set of functions $\{\varphi_0, \dots, \varphi_r\} \subseteq \PL(\Gamma)$ is a tropical linear combination 
    \[\theta = \min\{\varphi_0 + a_0, \ldots, \varphi_r + a_r\}\] such that for each $i$ there is some point at which the minimum is attained uniquely by $\varphi_i + a_i$.
\end{definition}

\noindent We will repeatedly use the fact that a  subset $S \subseteq \PL(\Gamma)$ is tropically independent if and only if there is a certificate of independence for $S$ \cite[Theorem~1.6]{M23}.

\subsection{Matroids and valuated matroids as tropical modules}
\label{sec:matroids}

Matroids are a combinatorial abstraction and generalization of the properties of linear independence for finite subsets of a vector space. The theory is vast; we recall only the minimal background needed for the purposes of this paper and refer the reader to \cite{oxley2006matroid} for further details. 

Some matroids come from vector spaces, as follows. Given a field $K$ and a vector space $V \subseteq K^n$, one obtains a matroid $M_V$ on the set $\{1, \ldots, n \}$ by saying that $I \subseteq \{1, \ldots, n\}$ is \emph{independent} if the corresponding set of coordinate linear functions $\{ e_i^* : i \in I \}$ is linearly independent in the dual vector space $V^*$.  In this situation, the elements $\{1, \ldots, n \}$ are tacitly identified with the coordinate linear functions and are referred to as the \emph{vectors} of $M_V$.  The independent sets of vectors satisfy the following properties:
\begin{itemize}
\item The empty set is independent.
\item Any subset of an independent set is independent.
\item If $I$ and $J$ are independent sets and $|I| < |J|$ then $I \cup \{j \}$ is independent for some $j \in J \smallsetminus I$.
\end{itemize}
By definition, any collection of subsets of a finite set $E$ that satisfy these three properties is the collection of independent sets in a matroid on $E$.
The maximal independent sets are called \emph{bases}, and all bases have the same size, which is the \emph{rank} of the matroid. Matroids that come from a linear subspace $V \subseteq K^n$ as above are called \emph{realizable}, and the rank of $M_V$ is $\dim_K V$.

Matroids have many cryptomorphic axiomatizations. For instance, a matroid $M$ on $E$ is characterized equivalently by its independent sets, its bases, the \emph{rank function} taking a subset of $E$ to the size of the largest independent set that it contains, or by the \emph{flats} which are the subsets that are maximal of a given rank. We give most of our attention to the characterization of matroids in terms of \emph{covectors}. 

Let us return to the case of the realizable matroid $M_V$ associated to a vector space $V \subseteq K^n$. We now assume that the field $K$ is infinite. %
The support of a point $v = (v_1, \ldots, v_n)$ in  $V$ is
\[
\supp(v) := \{ i \in \{1, \ldots, n\} : v_i \neq 0 \}.
\]
The set of \emph{covectors} of $M_V$ is $\{ \supp(v) : v \in V \}$. It has the following properties.
\begin{itemize}
    \item The empty set is a covector.
    \item Any union of covectors is a covector.
    \item If $I$ and $J$ are covectors with $i \in I \cap J$, then there is a covector contained in $I \cup J \smallsetminus \{i \}$ that contains the symmetric difference $I \triangle J$.
\end{itemize}
Conversely, any set of subsets that satisfies these three properties is the set of covectors of a matroid. 

To each matroid $M$ on a finite set $E$, one associates $\Trop(M) \subseteq \Rbar^E$. It is the tropical submodule generated by the tropical indicator functions of the covectors and is a tropical linear space of dimension equal to the rank of $M$. Here, the tropical indicator function of a subset $I \subseteq E$ takes the value $0$ on elements of $I$ and $\infty$ on elements of $E \smallsetminus I$. Let us assume that $\Trop(M) \cap \RR^E$ is nonempty, which is the case if and only if $M$ is \emph{loopless}, i.e., every single element subset of $E$ is independent. Then, $\Trop(M) \cap \RR^E$ is invariant under translation by $(1, \ldots, 1)$ and the image under projection to $\RR^E / (1, \ldots, 1)$ is the support of the Bergman fan $\cB(M)$. (We do not choose a preferred fan structure  and refer to this space simply as the Bergman fan when no confusion seems possible.) The preimage of the Bergman fan is dense in $\Trop(M)$ and the set of covectors of $M$ is exactly the subsets of $E$ whose indicator functions are contained in $\Trop(M)$. Thus, $M$ carries exactly the same information as the Bergman fan $\cB(M) \subseteq \RR^E / (1, \ldots, 1) $ and the tropical submodule $\Trop(M) \subseteq \Rbar^E$.

\bigskip

We are interested not only in matroids but also \emph{valuated matroids}, which similarly abstract and generalize the properties of linear independence for finite subsets of a vector space over a valued field, where one records not only which subsets are linearly dependent, but the valuations of the coefficients in each linear relation. We give particular attention to the tropical modules that arise in this context, i.e., the valuated analogs of the Bergman fan $\cB(M) \subseteq \RR^E / (1,\ldots, 1)$ and of $\Trop(M)$. Our presentation of valuated matroids is very brief. See \cite{MurotaTamura01} for a more detailed foundational treatment, and \cite{BrandtEurZhang21}, and the references therein, for further details on their cryptomorphic axiomatizations and relations to the geometry of tropical linear spaces.

We return to the situation where $V \subseteq K^n$ is a linear subspace, and now we suppose that $K$ carries a nontrivial valuation $\val \colon K \to \Rbar$. To each $v = (v_1, \ldots, v_n)$ in $V$ we associate 
\[
\trop(v) = (\val(v_1), \ldots, \val(v_n)) 
\]
in $\Rbar^n$. The set of \emph{valuated covectors} of the associated valuated matroid $\cV$ is the tropical submodule of $\Rbar^n$ generated by $\{ \trop(v) : v \in V \}$. Assuming the residue field of $K$ is infinite, the valuated covectors satisfy the \emph{valuated exchange axiom}:
\begin{itemize}
    \item If $w = (w_1, \ldots, w_n)$ and $(w'_1, \ldots, w'_n)$ are valuated covectors with $w_i = w'_i$ then there is a valuated covector $w'' = (w''_1, \ldots, w''_n)$ such that $w''_i = \infty$ and $w''_j \geq \min \{w_j, w'_j \}$ for $j \neq i$, with equality when $w_j \neq w'_j$.
\end{itemize}
The valuated matroids that come from subspaces of $K^n$ in this way are called \emph{realizable}. 

Valuated matroids have many different cryptomorphic axiomatizations. Most importantly for our purposes, a valuated matroid is determined by its tropical module of valuated covectors, and any tropical submodule of $\Rbar^n$ that satisfies the valuated exchange axiom occurs in this way. When no confusion seems possible, we identify a valuated matroid $\cV$ with its tropical module of valuated covectors. Hence, we say that a matroidal linear series $\Sigma$ (\Cref{def:MLS}) is the image of a valuated matroid under a homomorphism of tropical modules $\cV \twoheadrightarrow \Sigma$. We note that every valuated matroid is finitely generated as a tropical module, with one generator for each covector of minimal support. Here, the support of $w \in \cV$ is $\supp(w) := \{ i : w_i \neq \infty \}$.

If $\cV \subseteq \Rbar^n$ is a valuated matroid, and $v$ is contained in $\cV \cap \RR^n$, then the image of $\Star(v)$ in $\RR^n/(1, \ldots, 1)$ is the Bergman fan of a matroid, called the \emph{initial matroid} of $\cV$ at $v$. If a valuated matroid $\cV$ is realizable, then all of its initial matroids are realizable. Also, the support sets of elements of $\cV$ are the covectors of a matroid $M$, called the \emph{underlying matroid} of $\cV$. The underlying matroid and initial matroids all have the same rank, and this is the rank of $\cV$.

Finally, let us note that if $M$ is an ordinary matroid, then $\Trop(M)$ is the set of covectors of a valuated matroid $\cV_M$. The underlying matroid of $\cV_M$ is $M$, as is the initial matroid of $\cV_M$ at $0$.  Moreover, $M$ is realizable as a matroid if and only if $\cV_M$ is realizable as a valuated matroid.

\section{Projectivization and dimension for tropical modules}\label{sec:proj-and-dim-trop-module}

In this section, we discuss the relationship between independence rank and dimension for finitely generated tropical modules.

\subsection{Independence rank and generating sets} 

Let $E$ be a possibly infinite set. We are equally interested in the cases where $E$ is finite and where $E$ is a metric graph $\Gamma$. The notions of tropical dependence and independence discussed in \Cref{sec:independence} for subsets of $\PL(\Gamma)$ generalize essentially without change to subsets of tropical submodules of $\Rbar^E$.

\begin{definition}
   The \emph{independence rank} of a tropical submodule $\Sigma \subseteq \Rbar^E$, denoted $r_{\ind}(\Sigma)$,  is the supremum of the sizes of its tropically independent finite subsets.
\end{definition}

We start with a lemma relating  independence rank to independent subsets  of a generating set.

\begin{lemma}
\label{Lem:GeneratorDependence}

Let $S \subseteq \Rbar^E$.  Then the independence rank of the tropical submodule $\langle S \rangle$ is the supremum of the sizes of the tropically independent subsets of $S$.

\end{lemma}

\begin{proof} 
Let $\{y_1, \dots, y_n\} \subseteq \langle S \rangle$. Write each $y_i$ as a finite tropical linear combination $$y_i = \min_{x\in S}\{x + a_i(x)\},$$ where $a_i(x) = \infty$ for all but finitely many $x\in S$. 

Suppose $\{ y_1 , \ldots , y_n \}$ is tropically independent.  We will show that there is a tropically independent subset of $S$ of size $n$. By \cite[Theorem~1.6]{M23} there is a certificate of independence 
\begin{equation} \label{eq:y-indep}
    \theta = \min \{ y_1 + b_1, \ldots, y_n + b_n \}.
\end{equation} 
For each $i \in \{1, \ldots, n\}$, there is a point $p_i$ where $y_i + b_i$ achieves the minimum uniquely in \eqref{eq:y-indep}.  Choose $x_i\in S$ so that $x_i + a_i(x_i)$ achieves the minimum at $p_i$ in the expression $y_i = \min_{x\in S} \{ x + a_i(x) \}$.  We now show that the elements $x_1, \ldots, x_n$ are distinct and $\{x_1, \ldots, x_n \}$ is tropically independent.

Since $y_i + b_i$ achieves the minimum uniquely at $p_i$ in \eqref{eq:y-indep}, we have $$a_i(x_i) + b_i <  a_j(x_i) + b_j \mbox{  for all } j \neq i.$$  
 
It follows that $x_i \neq x_j$ for $i \neq j$, and $$\min \{x_1 + a_1(x_1) + b_1,
\ldots, x_n + a_n(x_n) + b_n  \}$$ is a certificate of independence for $\{ x_1, \ldots, x_n \} \subseteq S$. 
\end{proof}

\noindent By \Cref{Lem:GeneratorDependence}, the independence rank of any finitely generated tropical submodule $\Sigma \subseteq \Rbar^E$ is finite and bounded above by the size of any generating set. 

\begin{lemma} \label{Lem:Submodules}
Let $S \subseteq \Rbar^E$.  For any $x \in \langle S \rangle$, there is a tropically independent finite subset $S' \subseteq S$ such that $x$ is contained in  $\langle S' \rangle$.
\end{lemma}

\begin{proof}
Let $x \in \langle S \rangle$, and let $S' \subseteq S$ be minimal such that $x \in \langle S' \rangle$.  Write
\begin{equation} \label{eq:minimal-exp}
x = \min_{y \in S'} \{ y + a(y) \}.
\end{equation}
Since $S'$ is minimal, for each $y \in S'$ there is some $e \in E$ such that $y + a(y)$ achieves the minimum uniquely in \eqref{eq:minimal-exp} at $e$.  Hence \eqref{eq:minimal-exp} is a certificate of independence for $S'$.
\end{proof}

\subsection{Projectivizations of tropical modules} \label{sec:proj-trop-modules} 

If $\Sigma \subseteq \RR^n$ is a finitely generated tropical submodule, then its projectivization $|\Sigma|$ is the tropical convex hull of the images of its generators in $\RR^n/(1,\ldots, 1)$.  In particular, it is a polyhedral set. We now discuss the analogous statements for finitely generated submodules of $\PL(\Gamma)$.  Note that any such submodule is contained in $R(D)$ for some divisor $D$ on $\Gamma$.

For any two functions $\varphi$, $\varphi' \in \PL(\Gamma)$, we have $\ddiv(\varphi) = \ddiv(\varphi')$ if and only if $\varphi' = \varphi + a$ for some real number $a$. Hence the tropical complete linear series $|D|$ is the \emph{projectivization} of $R(D)$, i.e.
\[
|D| = R(D) / \RR.
\]
We likewise denote the projectivization of any tropical submodule $\Sigma \subseteq R(D)$ by 
\[|\Sigma| := \Sigma / \RR.\]
Note that $|\Sigma|$ is a subspace of $|D|.$ 
    If $D \sim D'$, then $D' = D + \ddiv(\varphi)$ for some $\varphi \in \PL(\Gamma)$ and the map $\varphi' \mapsto \varphi' + \varphi$ induces an isomorphism of tropical modules $R(D) \xrightarrow{\sim} R(D')$. This descends to a bijection $|D| \cong |D'|$ preserving the relevant polyhedral structures.

\begin{proposition}
    \label{prop:fin-gen-definable}
    If $\Sigma \subseteq R(D)$ is finitely generated, then $|\Sigma|$ is a polyhedral subset of $\Sym^d\Gamma$.
\end{proposition}

\begin{proof} 
    Let  $\{\varphi_0, \dots, \varphi_r\} \subseteq \Sigma$ be a generating set. For $a = (a_0, \ldots, a_r)$ in $\RR^{r+1}$, consider
    \begin{equation} \label{eq:phia}
    \varphi_a = \min\{ \varphi_0 + a_0, \dots, \varphi_r + a_r \}
    \end{equation}
    Consider the map $\Phi \colon \RR^{r+1} \to |\Sigma|$ given by $a \mapsto D + \ddiv(\varphi_a)$. 
    Let $(G,\ell)$ be a model of $\Gamma$ such that each generator  $\varphi_i$ is affine on every edge. Let $I = (I_e)_{e \in E(G)}$ be a collection of subsets of $\{0, \ldots, r \}$ indexed by the edges of $G$. 
    Consider the subset $Q_I \subseteq \RR^{r+1}$ defined by the condition that $\varphi_i + a_i$ achieves the minimum in \eqref{eq:phia} on a nonempty open subset of $e$ if and only if $i \in I_e$.

    The set $Q_I$ is cut out by a finite set of linear inequalities, so it is a polyhedron in $\RR^{r + 1}$, and the nonempty sets $Q_I$ form a polyhedral subdivision of $\RR^{r+1}$. From the description of the polyhedral structure on $\Sym^d(\Gamma)$, it follows that $\Phi(Q_I)$ is contained in a face $\sigma$ of $\Sym^d(\Gamma)$.  The restriction of $\Phi$ maps $Q_I$ affine linearly into $\sigma$, and it follows that $|\Sigma|$ is a polyhedral subset of $\Sym^d(\Gamma)$.
\end{proof}

\subsection{Independence rank and dimension} We now explain how the independence rank controls the dimension of projectivizations of tropical submodules of $\Rbar^n$ and $R(D)$. We use relations to the tropical rank of matrices, as presented, e.g., in  \cite[Chapter 5]{MaclaganSturmfels}. A square matrix with entries from $\Rbar$ is \emph{tropically nonsingular} if the minimum in the tropical determinant is achieved only once. The tropical rank of a matrix with entries from $\Rbar$ is the size of the largest nonsingular square submatrix. 

\begin{proposition}\label{prop:tropical-rank-properties-1}
    Let $\Sigma \subseteq \Rbar^m$ be a tropical submodule generated by $\{v_0,...,v_n\}$. Let $A$ be the tropical matrix with row vectors $v_0,...,v_n$. Then the following are equivalent.
    \begin{enumerate}
        \item The independence rank $r_{\ind}(\Sigma)$ is $r+1$;
        \item The tropical rank of $A$ is $r+1$;
        \item The dimension of $|\Sigma| = \Sigma /\RR$ is $r$.
    \end{enumerate}
\end{proposition}

\begin{proof}
The equivalence of (2) and (3) is \cite[Theorem 5.3.23]{MaclaganSturmfels}. We prove the equivalence of (1) and (2). Assume $r_{\ind}(\Sigma)=r+1$. By \Cref{Lem:GeneratorDependence}, we may assume that $v_0,...,v_r$ are tropically independent. Then there is a certificate of independence, i.e., there are constants $a_0,...,a_r$ and indices $i_0,...,i_r$, such that
    \[\min\{v_0+a_0,...,v_r+a_r\}\]
is uniquely achieved by $v_j + a_j$ in the coordinate labeled by $i_j$. The permutation $j \mapsto i_j$ then corresponds to the unique leading term in the tropical determinant of this $(r+1) \times (r+1)$ submatrix, and hence the tropical rank of $A$ is at least $r+1$.

Conversely, suppose $A$ has a nonvanishing minor $A'$ of size $r+1$. Then by (3), the tropical row span $\Sigma'$ of $A'$ is full-dimensional in $\Rbar^{r+1}$. By translation, we may assume that $\Sigma'$ contains an open ball centered at the origin. In particular, it contains the points 
    \[(0,\epsilon,\epsilon,..., \epsilon),(\epsilon,0,\epsilon,..., \epsilon),..., (\epsilon,\epsilon,..., \epsilon,0)\] for sufficiently small $\epsilon > 0$. This set is tropically independent, since the coordinatewise minimum is a certificate of independence, and hence $r_{\ind}(\Sigma')=r + 1$. Projection to a subset of the coordinates does not increase the independence rank of a tropical module, so it follows that $r_{\ind}(\Sigma)\geq r+1$.
\end{proof}

For finitely generated tropical submodules in $\PL(\Gamma)$, we have a similar statement.

\begin{proposition}\label{prop:tropical-rank-properties-2}
    Let $\Sigma \subseteq \PL(\Gamma)$ be the tropical submodule generated by $\{ \varphi_0,...,\varphi_n \}$. For each finite subset $S = \{ v_1, \ldots, v_s \}$ of $\Gamma$, let $A_S$ be the matrix with $(A_S)_{ij} = \varphi_i(v_j)$. Then the following are equivalent.
    \begin{enumerate}
        \item The independence rank $r_{\ind}(\Sigma)$ is $r+1$;
        \item The maximum of the tropical rank of $A_S$, over all finite subsets $S \subseteq \Gamma$, is $r + 1$;
       
        \item The dimension of $|\Sigma| = \Sigma/\RR$ is $r$.
    \end{enumerate}
\end{proposition}

\begin{proof}
The proof that (1) is equivalent to (2) is essentially identical to the proof of \Cref{prop:tropical-rank-properties-1}. We now show that (1) is equivalent to (3).

Suppose $r_\ind(\Sigma) = r + 1$.  By \Cref{Lem:GeneratorDependence}, after renumbering, we may assume $\{ \varphi_0, \ldots, \varphi_r \}$ is tropically independent. By \cite[Theorem~1.6]{M23}, there is a certificate of independence $$ \theta = \min\{ \varphi_0 + a_0, \ldots, \varphi_r + a_r \}.$$ The map $(b_0, \ldots, b_r) \mapsto \min\{ \varphi_0 +b_0, \ldots, \varphi_r + b_r \}$ embeds an open neighborhood of $(a_0, \ldots, a_r)$ into  $\Sigma$ and  projects to an $r$-dimensional subset of $|\Sigma|$. Thus, $\dim|\Sigma| \geq r$.  It remains to show that $\dim|\Sigma| \leq r$. 
By \Cref{Lem:Submodules}, $|\Sigma|$ is the union of the subsets $|\langle S \rangle |$, where $S$ ranges over subsets of $\{ \varphi_0, \ldots, \varphi_n \}$ of size $r + 1$.  Each such subset $|\langle S \rangle |$ has dimension at most $r$, as required. 
\end{proof}

In \Cref{sec:local_structures,Sec:Canonical}, we work with a broader class of tropical submodules of $\PL(\Gamma)$, those whose projectivization is polyhedral. %

\begin{definition}
A submodule $\Sigma \subseteq R(D)$ is \emph{polyhedral} if $|\Sigma|$ is a closed polyhedral subset of $|D|$.
\end{definition}

\begin{lemma}\label{lem:polyhedral-submodule-dimension}
    Let $\Sigma\subseteq R(D)$ be a polyhedral submodule. Then the dimension of $|\Sigma|$ is the maximum of the dimensions of its finitely generated submodules.
\end{lemma}

\begin{proof}
    It is enough to find one finite subset $S\subseteq \Sigma$ such that $\dim |\langle S \rangle| = \dim |\Sigma|$. Let $\tau$ be a face of maximal dimension in $|\Sigma|$. By re-centering, we may assume that $D$ is in the relative interior of $\tau$. The combinatorial type of divisors is constant in the relative interior of $\tau$, so $D$ is nondegenerate. Let $E=\{C_1,...,C_s\}$ be the connected components of $\Gamma\smallsetminus\supp(D)$. We claim that 
    \[\cL:=\{F_\varphi:\varphi\in \Sigma\}\cup \{E\}\]
    is the lattice of flats of the Boolean matroid on some partition of $E$. 
    First, note that $\cL$ is closed under intersection. Indeed, for $F_1$ and $F_2$ in $\cL$ there are functions $\varphi_1$ and $\varphi_2$ in $\Sigma$ with $F_{\varphi_i} = F_i$ and $\min \varphi_i = 0$. Then $F_{\min \{ \varphi_1, \varphi_2 \}} = F_1 \cap F_2$.
    Next, we assert that $\cL$ is closed under taking complements. To see this, start with any $\varphi$ in $\Sigma$. For simplicity, assume $\min \varphi = 0$. Then 
    \[
    D_\epsilon := \ddiv (\min \{ \varphi, \epsilon \}
    \]
    is contained in $\tau$ for sufficiently small $\epsilon > 0$. This divisor may be thought of as a minor variant of ``chip-firing along $\varphi_{\min}$ for time $\epsilon$" in the sense of \cite[Remark~2.9]{GST22}; the variation is that we allow slope equal to $D(x)$, rather than $1$, along a tangent direction at a point $x$ in the boundary of $\varphi_{\min}$. Since $D$ is in the relative interior of a maximal face $\tau$, the opposite of any such chip-firing move is well-defined in $\Sigma$, for sufficiently small $\epsilon$, i.e., we can fire $\overline{\Gamma \smallsetminus \varphi_{\min}}$, with the opposite slopes, for a sufficiently small time $\epsilon'$. In particular, $\cL \cup E$ is closed under complements, as claimed. Since $\cL \cup E$ is closed under complements and intersections, it is a Boolean lattice.

    Let $A_0,...,A_r$ be the minimal nonempty elements of $\cL$. Choose $\psi_0,...,\psi_r\in \Sigma$ such that $F_{\psi_i}=A_i$ and $\min\psi_i=0$ for all $i$. Let $D_i=D+\ddiv(\psi_i)$ and let  $\eta_i$ be the tangent vector at $D$ in the direction of the tropical line segment connecting $D'$ and $D$. Then $\eta_0,...,\eta_r$ positively span the tangent space of $\tau$ at $D$, so $\dim|\Sigma|\leq r$. On the other hand, $\min\{\psi_0,...,\psi_r\}$ is a certificate of tropical independence. By \Cref{prop:tropical-rank-properties-2}, 
    $\dim|\langle \psi_0,...,\psi_r\rangle|=r$, and $\dim|\Sigma|=\dim |\langle \psi_0,...,\psi_r\rangle|$, as required.
\end{proof}

\begin{corollary}\label{cor:tropical-rank-properties-3}
    Let $\Sigma \subseteq R(D)$ be a polyhedral submodule. For $S = \{\varphi_0, \ldots, \varphi_s\}$ a finite subset of $\Sigma$ and $\{v_1, \ldots, v_n \}$ a finite subset of $\Gamma$, let $A_{ST}$ be the matrix with $(A_{ST})_{ij} = \varphi_i(v_j)$.  Then the following are equivalent.
    \begin{enumerate}
        \item The independence rank $r_{\ind}(\Sigma)$ is $r+1$;
        \item The maximum of the tropical rank of the matrices $A_{ST}$, over all choices of finite subsets $S \subset \Sigma$ and $T \subset \Gamma$, is $r + 1$;
        \item The dimension of $|\Sigma| = \Sigma/\RR$ is $r$.
    \end{enumerate}
\end{corollary}

\subsection{Pure dimensionality of tropical linear series} \label{sec:localdim}

\Cref{prop:tropical-rank-properties-2} shows that the independence rank determines the global dimension of a finitely generated submodule of $\PL(\Gamma)$ and, more genearlly, any polyhedral submodule of $R(D) \subset \PL(\Gamma)$. We now show that the Baker--Norine rank gives a lower bound for its local dimension. This completes the proof of \Cref{thm:puredim}, showing that if $\Sigma$ is a tropical linear series of Baker-Norine rank $r$ then $|\Sigma|$ has pure dimension $r$. This theorem first appeared in the Master's thesis of the second named author \cite{DuprazMasters}. Similar statements with somewhat different proofs appeared later in \cite{AGG}. Dupraz's Master's thesis will not be published otherwise, and we include a brief but complete presentation based on his original proof.

\medskip

\begin{definition}
    Given an an effective divisor $D'$, we define
    \begin{equation*}
        \Sigma(-D') := \{\varphi \in \Sigma : D - D' + \ddiv(\varphi) \geq 0 \}.
    \end{equation*}
\end{definition}
\noindent Note that $\Sigma(-D')$ is a tropical submodule of $\Sigma$.

\medskip

\begin{proposition}
\label{rem:submodule-affine}
Let $\tau \subseteq |\Sigma|$ be a maximal face.  For any effective divisor $D'$, $|\Sigma(-D')|\cap\relint(\tau)$ is the intersection of $\relint(\tau)$ with a finite disjoint union of affine subspaces.  In particular, $\Sigma(-D) \subseteq R(D-D')$ is a polyhedral tropical submodule.
\end{proposition} %

\begin{proof}
     All divisors in the relative interior of $\tau$ have the same combinatorial type, and hence can be determined by the positions of the points of their support in the interiors of edges of $\Gamma$. Fixing an order on these points induces an inclusion $\relint(\tau)$ into some $\RR^k$. Fixing any given point in the support then corresponds to fixing that coordinate. It follows that the set of divisors in $\tau$ whose support contains a given point is a finite union of affine subspaces of $\tau$, and these subspaces do not intersect in the relative interior of $\tau$. This implies that $|\Sigma(-D')|\cap\relint(\tau)$ is a finite disjoint union of affine subspaces of $\relint(\tau)$, and the proposition follows. %
\end{proof}

Figure~\ref{Fig:AffineSubspace} illustrates \Cref{rem:submodule-affine} in the case where $D$ is a divisor of degree 2 on an interval $\Gamma$, $x$ is a point in $\Gamma$, and $\Sigma = R(D)$. The dashed lines indicate the set $\Sigma (-x)$. 

\begin{figure}[ht]
        \centering
        \begin{tikzpicture}
            \draw (0,0) -- (0,2) -- (2,0) -- (0,0);
            \draw[dashed] (1,0) -- (1,1) -- (0,1);
        \end{tikzpicture}
        \caption{$|\Sigma (-x)| \cap\relint(\tau)$ is a union of affine subspaces.}
        \label{Fig:AffineSubspace}
    \end{figure}
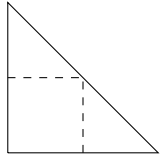

Finitely generated submodules of $R(D)$ are polyhedral by \Cref{prop:fin-gen-definable}.  If $\Sigma \subseteq R(D)$ is finitely generated, we do not know in general whether $\Sigma(-D')$ is finitely generated.  However, \Cref{rem:submodule-affine} shows that if $\Sigma$ is polyhedral then so is $\Sigma(-D')$.  Thus, the class of polyhedral submodules of $R(D)$ is well-suited for inductive arguments, as in the proof of the following theorem.

\begin{theorem} \label{thm:lowerBoundDim}
    If $\Sigma \subseteq R(D)$ is a polyhedral tropical submodule then every maximal face of $|\Sigma|$ has dimension at least $r_\BN(\Sigma)$.
\end{theorem}

\noindent 
\Cref{thm:puredim} is an immediate consequence of \Cref{prop:tropical-rank-properties-2} and \Cref{thm:lowerBoundDim}.

\begin{proof}
We proceed by induction on $r = r_{\BN}(\Sigma)$. The case $r = 0$ is trivial.
Suppose $r \geq 1$, let $\tau \subseteq |\Sigma|$ be a maximal face, and let $D$ be a divisor in the relative interior of $\tau$. Since $|\Sigma|$ is connected, there is a divisor $D' \neq D$ in $\tau$. 
Since every divisor in $\tau$ is of the same combinatorial type, there is a point $x \in \supp(D)$ that is not $\supp(D')$. It is immediate from \Cref{def:submodule-rank} that $r_{\BN}\big(\Sigma(-x)\big) \geq r -1$. Also, $|\Sigma(-x)| \cap \relint(\tau)$ is a disjoint union of affine subspaces, by \Cref{rem:submodule-affine}. Therefore, $D$ is in the relative interior of a maximal face $\tau' \subseteq |\Sigma(-x)|$ whose affine span does not contain $D'$. In particular,  $\dim(\tau) \geq \dim(\tau') + 1$. The theorem follows since, by induction, we have $\dim(\tau') \geq r-1$.  
\end{proof}

\section{Big minimizers and local matroids}\label{sec:local_structures}

In this section, we prove \Cref{Prop:NoIsolatedPoints} and \Cref{thm:local-structure}, showing that a tropical linear series of dimension $r$ locally looks like the Bergman fan of a matroid of rank $r+1$ in a neighborhood of any nondegenerate divisor. We also state and prove \Cref{thm:local_matroid}, which constructs a local matroid $M_{\Sigma}$ associated to a more general class of pairs $(D,\Sigma)$ such that $\Sigma \subseteq R(D)$.

\medskip

Let $D$ be an effective divisor on $\Gamma$, and let $\varphi \in R(D)$.
\begin{lemma}\label{lem:minimizer}
The set $\varphi_{\min}:= \{x\in \Gamma : \varphi (x)=\min \varphi \}$ is a union of points in $\supp(D)$ and closures of connected components of $\Gamma \smallsetminus \supp(D)$.
\end{lemma} 

\begin{proof}
Let $x$ be a point in the boundary of $\varphi_{\min}$.  Then all of the slopes of $\varphi$ along tangent directions at $x$ are nonnegative, and those along tangent directions that leave $\varphi_{\min}$ are strictly positive. Then $x \in \supp(D)$, because $D + \ddiv(\varphi)$ is effective,  and the lemma follows.
\end{proof}

Let $E = \{ C_1, \ldots, C_s \}$ be the set of  connected components of $\Gamma \smallsetminus \supp(D)$. For $\varphi \in R(D)$, define
\[
F_{\varphi} :=\{C_i \in E : C_i \not \subseteq \varphi_{\min} \}.
\]
The following basic properties are immediate from the definition.
\begin{itemize}
    \item The set $F_\varphi$ is empty if and only if $\varphi$ is constant.
    \item If $\min \varphi_1 = \min \varphi_2$ then $F_{\min \{ \varphi_1, \varphi_2 \}} = F_{\varphi_1} \cap F_{\varphi_2}$.
    \item If $\supp(D + \ddiv(\varphi)) \cap C_i \neq \emptyset$ then $C_i \in F_\varphi$.
\end{itemize}

\medskip

Let $\Sigma \subseteq R(D)$ be a tropical submodule, and let $r = r_{\BN}(\Sigma)$.

\begin{lemma} \label{lem:any-r}
Any size $r$ subset $S \subseteq E$ is contained in $F_\varphi$ for some $\varphi \in \Sigma$. 
\end{lemma}

\begin{proof}
After renumbering, we may assume $S = \{ C_1, \ldots, C_r \}$. Choose $x_i \in C_i$. Since $r_{\BN}(\Sigma) = r$, there is some $\varphi \in \Sigma$ such that $\{ x_1, \ldots, x_r \} \subseteq \supp(D + \ddiv(\varphi))$. Then $S \subseteq F_\varphi$.
\end{proof}

Recall from \Cref{def:nondegenerate} that $D \in |\Sigma|$ is nondegenerate if it is locally maximal for $\# \supp$ and locally minimal for the valence 1 degree. The following proposition captures a key technical property of nondegenerate divisors.

\begin{proposition}
\label{Prop:NoIsolatedPoints}
Suppose $\Sigma \subseteq R(D)$ is a tropical linear series and $D \in |\Sigma|$ is nondegenerate. Then, for all $\varphi \in \Sigma$, $\varphi_{\min}$ contains no isolated points. 
\end{proposition}

\noindent In particular, if $D$ is nondegenerate and $\varphi \in \Sigma$, then $\varphi_{\min}$ is the closure of some union of connected components of $\Gamma \smallsetminus \supp(D)$.

\begin{proof}%
Since $D \in |\Sigma|$, the constant functions are in $\Sigma$.  Suppose there is some $\varphi \in \Sigma$ such that $\varphi_{\min}$ contains an isolated point $x \in \supp (D)$.  Without loss of generality, assume that $\min \varphi = 0$.  Choose a small $\epsilon > 0$, and let $\varphi_{\epsilon} := \min \{ \varphi , \epsilon \}$.  For each tangent direction $\zeta$ at $x$, let $s_{\zeta} > 0$ denote the slope of $\varphi$ along $\zeta$.  Then $\supp(D + \ddiv(\varphi_\epsilon))$ contains $s_{\zeta}$ points at distance $\epsilon/s_{\zeta}$ from $x$ along each tangent direction $\zeta$. If there is only one tangent direction, then the valence one degree of $D$ is not locally minimal. And if there is more than one tangent direction, then $\# \supp(D)$ is not locally maximal.  Either way, $D$ is not nondegenerate.
\end{proof}

If $\Sigma \subseteq R(D)$ and $D \in |\Sigma|$ is nondegenerate, then for every $\varphi \in \Sigma$, $\varphi_{\min}$ is a nonempty union of components $C_i \in E$. Our construction of a local matroid $M_{\Sigma}$ works more generally, whenever the following condition is satisfied:

\begin{definition}\label{def:dagger-condition}
A tropical submodule $\Sigma \subseteq R(D)$ has \emph{big minimizers}
if for every $\varphi \in \Sigma$, the set $\varphi_{\min}$ contains some $C_i \in E$.
\end{definition}

\noindent  In other words, we say that a tropical submodule $\Sigma \subseteq R(D)$ has big minimizers if it contains no $\varphi$ with $\varphi_{\min} \subseteq \supp(D)$. In particular, if $D \in |\Sigma|$ is nondegenerate, then $\Sigma$ has big minimizers. However, we also consider cases where $\Sigma \subseteq R(D)$ has big minimizers but $D \in |\Sigma|$ is degenerate (Example~\ref{ex:K-degen}), and cases where $\Sigma  \subseteq R(D)$ but $D \not \in |\Sigma|$ (Examples~\ref{ex:canonical} and \ref{ex:vamos-relaxiation-series}).  See \Cref{rmk:DnotinSigma}.

\begin{theorem}\label{thm:local_matroid}
Let $\Sigma \subseteq R(D)$ be a tropical linear series of dimension $r$ with big minimizers.  Then 
\[
\{ F_\varphi : \varphi \in \Sigma \} \cup \{ E \}
\]
is the lattice of flats of a matroid $M_{\Sigma}$ of rank $r +1$ on $E$, and $M_{\Sigma}$ is loopless if and only $D \in |\Sigma|$.
\end{theorem}

Before proving Theorem~\ref{thm:local_matroid}, we illustrate its conclusion with two examples and three corollaries.

\begin{example} \label{ex:canonical}
Let $\Gamma$ be a metric graph of genus $g \geq 2$ without vertices of valence 1, and let $G$ be the model of $\Gamma$, whose vertex set consists of points of $\Gamma$ of valence at least 3.  Let $K_\Gamma$ be the canonical divisor on $\Gamma$. Haase, Musiker, and Yu observed that a subdivision of $|K_\Gamma|$ contains the Bergman fan of the cographic matroid of $\Gamma$ as a subfan \cite[Theorem~25]{HMY12}. We refine their observation as follows. Recall that the  support of $K_\Gamma$ is the set of vertices, and the multiplicity of a vertex $v$ in $K_\Gamma$ is $n_v - 2$. The set $E$ of connected components of $\Gamma \smallsetminus \supp(K_\Gamma)$ is the set of edges of $G$.  

Let $\Sigma \subseteq R(K_\Gamma)$ be a tropical linear series of dimension $g-1$.  We will show that $\Sigma$ has big minimizers and the local matroid $M_{\Sigma}$ is the cographic matroid of $G$. Note that $K_\Gamma$ may or may not be contained in $|\Sigma|$. For instance, if $\Gamma$ is the complete graph $K_4$, then $K_\Gamma$ is contained in a realizable tropical linear series of dimension 2, by \cite[Proposition~3.27]{DuprazMasters}, cf. \cite[Example~26]{HMY12} and \cite[Example~6.5]{MUW21}. 
On the other hand, when $\Gamma$ is a barbell graph, there is a unique linear series $\Sigma \subseteq R(K_\Gamma)$ of dimension 1, and $|\Sigma|$ does not contain $K_\Gamma$. See \Cref{ex:canonical-divisor-complete-linear-system}. 

Let $\varphi \in R(K_{\Gamma})$.  By \cite[Lemma~3.2]{JensenPayne16}, $\varphi_{\min}$ is a nonempty union of edges in $G$ and has no points of valence 1.  Since $\varphi_{\min}$ is a union of edges, $\Sigma$ has big minimizers.  Furthermore, since the union of edges has no points of valence 1, $\varphi_{\min}$ is a union of circuits in $G$.  Thus, each $F_{\varphi}$ is a flat of the cographic matroid.  Conversely, by \cite[Proposition~3.3]{JensenPayne16}, every corank 1 flat of the cographic matroid is of the form $F_\varphi$ for some $\varphi \in \Sigma$.  Since every flat of a matroid is an intersection of corank 1 flats, it follows that every flat of the cographic matroid is a flat of $M_{\Sigma}$, and hence $M_{\Sigma}$ is the cographic matroid, as claimed.
\end{example}

\begin{example} \label{ex:K-degen}
When $K_\Gamma$ is in $|\Sigma|$, $K_\Gamma$ may be degenerate or nondegenerate. Let $\Gamma_1$ be a chain of two loops with no bridge, and let $\Gamma_2$ be a metric graph modeled on the graph consisting of two vertices connected by three edges, as pictured in Figure~\ref{Fig:KDegenOrNot}. Then both $R(K_{\Gamma_1})$ and $R(K_{\Gamma_2})$ are tropical linear series of dimension one by Riemann-Roch and Proposition~\ref{cor:tropical-rank-properties-3}. It is easy to check that $K_{\Gamma_1}$ is degenerate, while $K_{\Gamma_2}$ is nondegenerate.

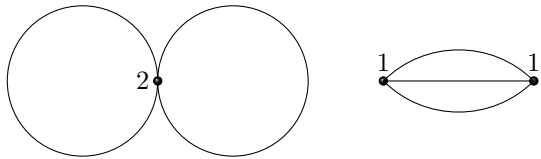
\begin{figure}[ht]
\begin{tikzpicture}

\draw (0,0) circle (1);
\draw (2,0) circle (1);
\draw [ball color=black] (1,0) circle (0.55mm);
\draw (0.8,0) node {2};

\draw [ball color=black] (4,0) circle (0.55mm);
\draw [ball color=black] (6,0) circle (0.55mm);
\draw (4,0) -- (6,0);
\draw (4,0) to [out=45,in=135,looseness=1] (6,0);
\draw (4,0) to [out=315,in=225,looseness=1] (6,0);
\draw (4,0.25) node {1};
\draw (6,0.25) node {1};

\end{tikzpicture}
\caption{The graphs $\Gamma_1$ and $\Gamma_2$ of Example~\ref{ex:K-degen}, with their canonical divisors.}
\label{Fig:KDegenOrNot}
\end{figure}
\end{example}

\noindent The existence of a tropical linear series $\Sigma \subseteq R(K_\Gamma)$ of dimension $g-1$ follows from specialization, by tropicalizing the canonical linear series on any curve with skeleton $\Gamma$.  However, we do not have any direct combinatorial existence proof, let alone a classification of the tropical linear series of dimension $g-1$ that are contained in $R(K_\Gamma)$.

If $M$ and $M'$ are matroids on a set $E$, then $M$ is a \emph{matroid quotient} of $M'$ if every flat of $M$ is a flat of $M'$. The following is an immediate consequence of \Cref{thm:local_matroid}.

\begin{corollary}\label{cor:incidence-local-matroid}
If $\Sigma \subseteq \Sigma' \subseteq R(D)$ are tropical linear series that have big minimizers, then $M_{\Sigma}$ is a matroid quotient of $M_{\Sigma'}$.
\end{corollary}

\begin{proof}
If $\Sigma \subseteq \Sigma' \subseteq R(D)$, then every flat of $M_{\Sigma}$ is a flat of $M_{\Sigma'}$. %
\end{proof}

\begin{corollary}\label{prop:maximality-tls}
    If $\Sigma \subseteq \Sigma' \subseteq R(D)$
    are tropical linear series of the same dimension then $\Sigma = \Sigma'$.
\end{corollary}

\begin{proof}
    Up to replacing $D$, $\Sigma$, and $\Sigma'$ with $D + \ddiv(\varphi)$,  
    $\Sigma - \varphi$ and $\Sigma' - \varphi$, respectively, we may assume that $D$ is non-degenerate in $|\Sigma'|$. Then $\Sigma'$ has big minimizers, and hence so does $\Sigma$.
    By \Cref{cor:incidence-local-matroid},
    $M_{\Sigma}$ is a matroid quotient of $M_{\Sigma'}$.  Since these matroids have the same rank, they are equal. This implies in particular that $M_{\Sigma}$ is loopless, so $D \in |\Sigma|$. Since this is true for every nondegenerate $D \in |\Sigma'|$ and the nondegenerate divisors are dense, it follows that  $|\Sigma| = |\Sigma'|$.
\end{proof}

\begin{corollary}\label{cor:strong-maximality-tls}
    Let $\Sigma \subseteq R(D)$ be a tropical linear series of dimension $r$. If $\Sigma' \subseteq R(D)$ is a tropical submodule such that $\Sigma \subseteq \Sigma'$ and $r_{\ind}(\Sigma') = r + 1$, then  $\Sigma = \Sigma'$.
\end{corollary}

\begin{proof}
Suppose $f \in \Sigma'$. Then $\Sigma'' := \langle \Sigma, f \rangle$ is finitely generated and $r_\BN(\Sigma'') \geq r_\BN(\Sigma) = r$. Also, $r_\ind(\Sigma'') \leq r_{\ind}(\Sigma')$.  Hence, if  $r_\ind(\Sigma') = r + 1$ then $\Sigma''$ is a tropical linear series of dimension $r$ %
and, by \Cref{prop:maximality-tls}, it follows that $f \in \Sigma$.
\end{proof}

We now proceed with a lemma and a proposition leading to the proof of \Cref{thm:local_matroid}. For any set $\cF$ of subsets of a finite set $E$, let $\ell(\cF)$ denote the length of the longest chain of proper inclusions in $\cF$, i.e., if $F_0 \subsetneq F_1 \subsetneq \cdots \subsetneq F_k$ is a chain of maximal length in $\cF$, then $\ell(\cF) = k$.

\begin{lemma}\label{prop:length-upper-bound}
Let $\Sigma \subseteq R(D)$ be a tropical submodule, and let $\cF = \{ F_\varphi : \varphi \in \Sigma \}$. Then \[
\ell(\cF) \leq r_{\ind}(\Sigma) -1.
\]
\end{lemma}

\begin{proof}
Let $F_{\varphi_0}\subsetneq \cdots \subsetneq F_{\varphi_k}$ be a chain in $\cF$. We may assume that $\min \varphi_i = 0$ for all $i$. 
For each $i$, pick $x_i\in \varphi_{i,\min} \smallsetminus \varphi_{i+1,\min}$. Consider the tropical matrix
\[\begin{bmatrix}
    \varphi_i(x_j)
\end{bmatrix}_{i,j=0,...,k}.\]
We have $\varphi_i(x_j)=0$ if $i\leq j$ and $\varphi_i(x_j)>0$ if $i>j$. This matrix is tropically nonsingular, because the identity permutation contributes the unique minimal term in the tropical determinant. Thus, $r_{\ind}(\Sigma)$ is at least $k+1$, by \Cref{cor:tropical-rank-properties-3}.
\end{proof}

We will use the following cryptomorphic characterization of the lattice of flats of a matroid. 

\begin{proposition}\label{lem:matroid_lemma}
Let $\cF$ be a set of proper subsets of a finite set $E$ that is closed under intersection and has length $0 < \ell(\cF) \leq r$.  Suppose that every size $r$ subset $S \subseteq E$ is contained in some $F \in \cF$. Then $\cF \cup \{E\}$ is the lattice of flats of a matroid of rank $r+1$ on $E$. 
\end{proposition}

\begin{proof}
We show that the collection $\cF$ has the partition property: for any $F\in \cF$, the collection
\[
\{G_i \smallsetminus F : G_i\text{ is a minimal subset in $\cF\cup\{E\}$ strictly containing }F\}
\]
partitions $E \smallsetminus F$. Let $G_1 , G_2$ be distinct minimal subsets in $\cF \cup \{ E \}$ that strictly contain $F$.  Since $\cF$ is closed under intersection, by minimality we have $G_1 \cap G_2 = F$.  Hence, it suffices to show that for any $b \in E \smallsetminus F$, there is some $G \in \cF$ that contains $F \cup \{ b \}$ and covers $F$. 

Given a subset $S \subseteq E$, we say that the element $F \in \cF$ is \emph{generated by} $S$ if it is the intersection of all elements of $\cF$ containing $S$.
Let $S$ be a minimal generating set for $F$, and let $k = |S|$.
Choose a sequence $S_1\subseteq S_2\subseteq \cdots \subseteq S_{k}=S$ such that $|S_i|=i$ for each $i$, and let $F_i$ be the element in $\cF$ generated by $S_i$. We must have $F_0\subsetneq F_1\subsetneq\cdots\subsetneq F_k=F$, where $F_0$ is the unique minimal element in $\cF$. Otherwise, if $F_i=F_{i+1}$ for any $i$, then any subset in $\cF$ containing $F_i$ also contains $\{a_{i+1}\}=S_{i+1}\smallsetminus S_i$, meaning $S\smallsetminus\{a_{i+1}\}$ is a smaller generating set for $F$.

In particular, we have $k\leq r$. If $k = r$, then $F\neq E$ and $F$ is a maximal element in $\cF$. The only element of $\cF \cup \{ E \}$ covering $F$ is $E$, and the result follows, so assume that $k < r$. Let $b\in E\smallsetminus F$, and let $G$ be the set generated by $S \cup \{b\}$. Then because $|S \cup \{b\}| \leq r$, it follows from the assumptions that $G \neq E$. By repeating this argument we may choose $b = b_1, \dots, b_{r-k}$, such that if we denote by $G_i$ be the set generated by $S \cup \{b_1, \dots, b_i\}$,
then $b_{i+1} \notin G_i$.
As a result, we get a chain $F_0\subsetneq F_1\subsetneq\cdots\subsetneq F_k=F\subsetneq G_1\subsetneq\cdots\subsetneq G_{r-k}$. Since $\ell(\cF) \leq r$, this chain must be maximal and so $G=G_1$ contains $F\cup\{b\}$ and covers $F$ as required.
\end{proof}

\begin{proof}[\bf{Proof of \Cref{thm:local_matroid}}]
Let $\Sigma \subseteq R(D)$ be a tropical linear series of dimension $r$ that has big minimizers, and let $\cF = \{ F_\varphi : \varphi \in \Sigma \}$. Then $\cF$ is closed under intersection.  By \Cref{prop:length-upper-bound}, it has length at most $r$, and by \Cref{lem:any-r}, any size $r$ subset of $D$ is contained in some element of $\cF$.  Since $\Sigma$ has big minimizers, the set $E$ is not an element of $\cF$, i.e., $\cF$ is a collection of \emph{proper} subsets of $E$.  Hence \Cref{lem:matroid_lemma} says that $\cF \cup \{E \}$ is the lattice of flats of a matroid. Finally, a matroid is loopless if and only if the empty set is a flat. Now $F_\varphi$ is empty if and only if $\varphi$ is constant, and $\Sigma$ contains a constant function if and only if $D \in |\Sigma|$.
\end{proof}

Next, we prove \Cref{thm:local-structure}. The first part of this theorem is a special case of \Cref{thm:local_matroid}. The second part, which is to be proved, says that when $\Sigma \subseteq R(D)$ is a tropical linear series and $D \in |\Sigma|$ is nondegenerate, evaluation of functions at one point in each connected component of $\Gamma \smallsetminus \supp(D)$ embeds a neighborhood of $D \in |\Sigma|$ in $\RR^E/(1, \ldots, 1)$, and extending linearly induces an identification of $\Star(D)$ with the Bergman fan of the local matroid $M_\Sigma$.

\begin{proof}[\bf{Proof of \Cref{thm:local-structure}}] 

Suppose $\Sigma \subseteq R(D)$ is a tropical linear series of dimension $r$ such that $D \in |\Sigma|$ is nondegenerate, and let $E = \{C_1, \ldots, C_s\}$ be the set of connected components of $\Gamma \smallsetminus \supp(D)$.  By \Cref{thm:local_matroid}, $\{ F_\varphi : \varphi \in \Sigma \} \cup \{ E \}$ is the lattice of flats of a matroid $M := M_{\Sigma}$ on $E$.  Let $\{ \varphi_1, \ldots, \varphi_n \} \subseteq \Sigma$ be a generating set.  Then every corank-1 flat $F$ of $M$ is $F_{\varphi_i}$ for some $i$.

Choose $p_i \in C_i$. The evaluation map $\widetilde \Phi \colon \Sigma \to \RR^E$ given by $\varphi \mapsto (\varphi(p_1), \ldots, \varphi(p_n))$ is a homomorphism of tropical modules and descends to a well-defined map $\Phi \colon |\Sigma| \to \RR^E/(1, \ldots, 1)$. By tropical rescaling, we may assume $\min \varphi_i = 0$, for all $i$. We claim that $\Phi$ maps a small neighborhood of $D \in |\Sigma|$ homeomorphically onto a neighborhood of $0 \in \cB(M)$, and the restriction to the preimage of each cone is affine. Extending linearly then gives a homeomorphism $\Star(D) \cong \cB(M)$. It remains to prove the claim. Choose $\epsilon > 0$ sufficiently small, and let $\varphi'_i := \min\{ \varphi_i, \epsilon \}$.  Note that 
\[
\varphi'_i(p_j) = \left\{ \begin{array}{ll}
\epsilon &\text{if } C_j \in F_{\varphi_i} \\
0 &\text{if } C_j \notin F_{\varphi_i}.
\end{array} \right.
\]
Let $\widetilde \cB(M) \subseteq \RR^E$ denote the preimage of $\cB(M)$. It is a tropical submodule in which a neighborhood of $0$ is generated by the indicator vectors $1_F$
of the corank-1 flats $F \subseteq E$. Choose $\varphi'_i$ such that $F = F_{\varphi_i}$.  Then $\widetilde \Phi(\varphi'_i) = \epsilon 1_F$, and the image of $\langle 0, \varphi'_1, \ldots, \varphi'_n \rangle$ is a neighborhood $U$ of $0 \in \cB(M)$.

Next, we show that the restriction of $\Phi$ to a small neighborhood of $D \in |\Sigma|$ is an injection onto $U$.  Choose $\varphi$ within $\epsilon$ of $0$, assume $\min \varphi = 0$, and let $x \in \supp(D)$.  Since $D$ has locally maximal support, the slope of $\varphi$ on any tangent vector $\zeta$ at $x$ is either $0$ or the multiplicity $D(x)$ of $D$.  Moreover, there is at most one tangent vector $\zeta \in T_x (\Gamma)$ such that $s_{\zeta} (\varphi)$ is nonzero.  Because $D$ is a local minimum for the valence 1 degree, there exists at least one tangent vector $\zeta \in T_x (\Gamma)$ such that $s_{\zeta} (\varphi) =0$.
Suppose $\zeta \in C_i$.  If $s_\zeta(\varphi) = 0$, then $\varphi(p_i) = \varphi(x)$.  Otherwise, $\varphi(p_i) - \varphi(x)$ is $D(x)$ times the length of the segment along which $\varphi$ has slope $D(x)$.  In other words, as $C_i$ ranges over all components whose boundary contains $x$, there are at most two values of $\varphi (p_i)$.  The value $\varphi(x)$ is equal to the minimum of these two values $\varphi(p_i)$, and the length of the segment along which $\varphi$ has slope $D(x)$ is determined by the maximum of these two values $\varphi (p_i)$.  In this way, $\varphi$ is determined by its image under $\widetilde \Phi$.  

Descending to projectivizations, we see that $\Phi$ maps a small neighborhood of $D \in |\Sigma|$ homeomorphically onto the neighborhood $U$ of $0 \in \cB(M)$ and, by construction, this map is affine on the preimage of each cone. This proves the claim and the theorem.
\end{proof}

\section{Tropicalizations of canonical linear series}
\label{Sec:Canonical}

Here we apply the results of \Cref{sec:local_structures} to study tropicalizations of canonical linear series, building on the work of M\"oller--Ulirsch--Werner \cite{MUW21}.

Let $\Gamma$ be a metric graph of first Betti number $g$, and let $X$ be a genus $g$ curve over a nonarchimedean field with skeleton $\Gamma$.  Then the canonical linear series $|K_X|$ has dimension $g-1$, and its tropicalization is contained in $|K_\Gamma|$.  A divisor $D \in |K_\Gamma|$ is \emph{realizable} if it is the tropicalization of a divisor in $|K_X|$ for some $X$ with skeleton $\Gamma$.  Little is known about the locus of realizable divisors if one considers curves over nonarchimedean fields with residue characteristic $p$.  See Problem~\ref{prob:canonical}.

A divisor $D \in |K_\Gamma|$ is \emph{realizable in equicharacteristic $0$} if it is the tropicalization of a divisor in $|K_X|$ for some curve $X$ with skeleton $\Gamma$ over a nonarchimedean field of residue characteristic $0$. Let
\[
\Real(|K_\Gamma|) := \{ D \in |K_\Gamma| : D \mbox{ is realizable in equicharacteristic $0$}\}.
\]
Note that $\Real(|K_\Gamma|)$ is a union over all curves with skeleton $\Gamma$ in equicharacteristic zero. Building on the breakthrough work of Bainbridge--Chen--Gendron--Grushevsky--M\"oller on moduli spaces of multiscale holomorphic differentials \cite{BCGGM17}, M\"oller, Ulirsch, and Werner determined $\Real(|K_\Gamma|)$ for all tropical curves \cite{MUW21}. However, they left open the problem of understanding how $\Real(|K_\Gamma|)$ relates to $\Trop(|K_X|)$ for any single curve $X$ with skeleton $\Gamma$.

For the remainder of this section, we focus on the case of equicharacteristic zero and say that a divisor in $|K_\Gamma|$ is realizable if it is realizable in equicharacteristic zero.

The space $\PP\Omega M_g^{\trop}$ of realizable canonical divisors constructed in \cite{MUW21} is a generalized cone complex, in the sense of \cite{acp}, of pure dimension $4g-4$ with a forgetful map to $M_g^{\trop}$, the moduli space of stable tropical curves of genus $g$. The forgetful map is surjective and the fiber over a tropical curve $\Gamma$ is $\Real(|K_\Gamma|)$. By \Cref{thm:lowerBoundDim}, the local dimension of $\Real(|K_\Gamma|)$ at any point is at least $g-1$ at every point.  Since $M_g^{\trop}$ has pure dimension $3g-3$, it follows that there is an open dense subset of $M_g^{\trop}$ over which the fibers have pure dimension $g-1$. However, there are fibers of dimension strictly greater than $g-1$.  See Example~\ref{ex:hyperelliptic-chain} for a tropical curve $\Gamma$ of genus $3$ such that $\Real(|K_\Gamma|)$ has dimension $3$.

The M\"oller--Ulirsch--Werner classification of realizable canonical divisors involves the following notion of inconvenient points.

\begin{definition}
A point $v \in \Gamma$ is \emph{inconvenient} for $\varphi \in R(K_\Gamma)$ if
$s_{\zeta} (\varphi) \neq 0$ for all $\zeta \in T_v (\Gamma)$, and
\[
\max_{\zeta \in T_v (\Gamma)} \{s_{\zeta} (\varphi) \} > -\sum_{\substack{\zeta \in T_v (\Gamma)\\ s_{\zeta} (\varphi) < 0}} s_{\zeta} (\varphi) . 
\]
\end{definition}

Note that an inconvenient point $v$ necessarily has valence $n_v \geq 3$.

\begin{theorem*} \cite[Theorem~1]{MUW21}
\label{thm:MUW}
A divisor $D = K_\Gamma + \ddiv(\varphi)$ in $|K_\Gamma|$ is realizable if and only if 
\begin{enumerate}[label=(\roman*)]
\item \label{it:inconvenient} any $v \in \Gamma$ that is inconvenient for $\varphi$ is contained in a simple cycle $\gamma \subseteq \Gamma$ with $\varphi (x) \leq \varphi(v)$ for all $x \in \gamma$, and 
\item \label{it:horizontal} any tangent vector $\zeta \in T_v(\Gamma)$ such that $s_{\zeta} (\varphi) = 0$ is contained in a simple cycle $\gamma \subseteq \Gamma$ with $\varphi (x) \leq \varphi(v)$ for all $x \in \gamma$.
\end{enumerate}
\end{theorem*}

To prove \Cref{thm:realizable}, we will use the following proposition, which first appeared in Dupraz's Master's thesis \cite[Section~3.4]{DuprazMasters}.

\begin{proposition}
\label{prop:realizable-module}
The subset $\{ \varphi \in R(K_\Gamma) : K_\Gamma + \ddiv(\varphi) \in \Real(|K_\Gamma|) \}$ 
is a polyhedral submodule. %
\end{proposition}

\begin{proof}
First, we claim that $\Real(|K_\Gamma|)$ is polyhedral. We give a short proof using Berkovich theory and elimination of quantifiers. For a combinatorial (but longer) proof, see \cite[Section~3.4]{DuprazMasters}. We begin by noting that the locus of curves with skeleton $\Gamma$ is a definable semialgebraic subset of the moduli space of curves $M_g$. Therefore, the locus $S_\Gamma$ of pairs $(X,D_X)$ such that $X$ has skeleton $\Gamma$ and $D_X \in |K_X|$ is a definable semialgebraic subset of $\PP\Omega M_g$. By elimination of quantifiers in the first order theory of algebraically closed valued fields \cite{Robinson56}, the image $\Real(|K_\Gamma|)$ of this definable set under the definable tropicalization map to $\Sym^{2g-2}(\Gamma)$ is definable, i.e., a finite Boolean combination of polyhedra. Finally, we note that the Berkovich analytifcation $S_\Gamma^\an$ is compact, because it is the preimage of a compact set (the space of curves with skeleton $\Gamma$) under a proper morphism. Thus its image $\Real(|K_\Gamma|)$ is compact and hence closed, i.e. $\Real(|K_\Gamma|)$ is a closed polyhedral subset of $\Sym^{2g-2}(\Gamma)$, as required.

It remains to show that $R = \{ \varphi \in R(K_\Gamma) : K_\Gamma + \ddiv(\varphi) \in \Real(|K_\Gamma|) \}$ is a tropical submodule.  It is clearly closed under scalar addition.  We must show that if $\varphi_1$ and $\varphi_2$ are in $R$, then $\varphi := \min\{ \varphi_1, \varphi_2 \}$ is in $R$.  In other words, we assume that $\varphi_1$ and $\varphi_2$ satisfy conditions \ref{it:inconvenient} and \ref{it:horizontal} and show that $\varphi$ satisfies these conditions as well.

First, suppose $v \in \Gamma$ is inconvenient for $\varphi$.  Without loss of generality we may assume that $\varphi_1 (v) \leq \varphi_2 (v)$.
Note that $s_{\zeta} (\varphi) \leq s_{\zeta} (\varphi_1)$ for all $\zeta \in T_v (\Gamma)$, hence
\[
\max_{\zeta \in T_v(\Gamma)}\{s_\zeta(\varphi_1)\} \geq
\max_{\zeta \in T_v(\Gamma)}\{s_\zeta(\varphi)\} > - \sum_{\substack{\zeta \in T_v (\Gamma) \\ s_{\zeta} (\varphi) < 0}} s_{\zeta} (\varphi) \geq - \sum_{\substack{\zeta \in T_v (\Gamma) \\ s_{\zeta} (\varphi_1) < 0}} s_{\zeta} (\varphi_1) . 
\]
It follows that either there is a tangent vector $\eta \in T_v (\Gamma)$ such that $s_{\eta}(\varphi_1)=0$, or $v$ is inconvenient for $\varphi_1$.  In each case, there is a simple cycle $\gamma \subseteq \Gamma$ containing $v$ such that $\varphi(x) \leq \varphi_1 (x) \leq \varphi_1 (v) = \varphi(v)$ for all $x \in \gamma$. Thus $\varphi$ satsifies condition \ref{it:inconvenient}.

Second, suppose $s_{\zeta} (\varphi) = 0$ for some tangent vector $\zeta \in T_v(\Gamma)$.  Without loss of generality, we may assume that $\varphi_1 (v) \leq \varphi_2 (v)$ and $s_{\zeta} (\varphi_1) = 0$.  Then there is a simple cycle $\gamma \subseteq \Gamma$ containing $\zeta$ such that $$\varphi(x) \leq \varphi_1 (x) \leq \varphi_1 (v) = \varphi(v),$$ for all $x \in \gamma$.  Thus $\varphi$ satisfies condition \ref{it:horizontal}, as required.
\end{proof}

Next, we prove \Cref{thm:realizable}, which says that if $X$ is a curve of genus $g$ with skeleton $\Gamma$ in equicharacteristic zero, then $\Trop(|K_X|) = \mathrm{Real}(|K_\Gamma|)$ if and only if $\mathrm{Real}(|K_\Gamma|)$ has dimension $g-1$.

\begin{remark}
By \Cref{thm:realizable}, if $\Real(|K_\Gamma|)$ has dimension $g -1$, then it is finitely generated as a tropical module.  However, we do not know whether $\Real(|K_\Gamma|)$ is finitely generated in general.
\end{remark}  %

\begin{proof}[{\bf Proof of \Cref{thm:realizable}}]
If $\Real(|K_\Gamma|) = \Trop(|K_X|)$ for some curve $X$ then it has dimension $g-1$, because $\Trop(|K_X|)$ is a tropical linear series of dimension $g-1$. We now prove the converse.

Suppose the polyhedral submodule $\Real(|K_\Gamma|) \subseteq R(K_\Gamma)$ has dimension $g-1$, and let $X$ be any curve with skeleton $\Gamma$ over a nonarchimedean field of  equicharacteristic zero. Then $\Real(|K_\Gamma|)$ has independence rank $g$, by \Cref{cor:tropical-rank-properties-3}, and $\Trop(|K_X|) \subseteq\Real(|K_\Gamma|)$ is a tropical linear series of dimension $g -1$. By \Cref{cor:strong-maximality-tls},  $\Trop(|K_X|) = \Real(|K_\Gamma|)$, as required. 
\end{proof}

We now give an example showing that $\Real(|K_\Gamma|)$ can have dimension strictly greater than $g-1$.

\begin{example}
\label{ex:hyperelliptic-chain}
Let $\Gamma$ be a hyperelliptic chain of 3 loops, as pictured in Figure~\ref{fig:hyperelliptic-chain}. The condition for $\Gamma$ to be hyperelliptic is that the top and bottom edge of the middle loop have equal length.  We show that $|\mathrm{Real}(K_{\Gamma})|$ has dimension at least 3. 
Let $\sigma$ be the face of $|K_{\Gamma}|$ consisting of those divisors that are supported on the third loop $\gamma_3$. Note that this face is of dimension 3 by \cite[Proposition~13]{HMY12}. %
We claim that all the divisors in $\sigma$ are realizable.

Let $D = K_\Gamma + \ddiv(\varphi) \in \sigma$.  All such $\varphi$ have the same restriction to the first two loops and bridges, as shown in Figure~\ref{fig:hyperelliptic-chain-functions}.  Since no point of valence 2 is inconvenient, we see that the only inconvenient point is $v_2$.  But the restriction of $\varphi$ to the simple cycle $\gamma_2$ achieves its maximum at $v_2$, so $D$ satisfies \ref{it:inconvenient}.  The function $\varphi$ is constant on the first loop $\gamma_1$.  If $\zeta$ is a tangent vector outside of $\gamma_1$ with $s_\zeta(\varphi) = 0$, then $\zeta$ is contained in $\gamma_3$.  In this case $\varphi$ achieves it maximum at $\zeta$, so $D$ satisfies \ref{it:horizontal}.

\begin{figure}[ht]
    
        \begin{tikzpicture}[scale=1.5]
            \draw (-0.5, 0) circle (0.5);
            \draw (0,0) -- (1,0);
            \draw (1.5, 0) circle (0.5);
            \draw (2,0) -- (3,0);
            \draw (3.5, 0) circle (0.5);
            \node[vertex, label={60:$v_1$}] at (0, 0) {};
            \node[vertex, label={120:$w_2$}] at (1, 0) {};
            \node[vertex, label={60:$v_2$}] at (2, 0) {};
            \node[vertex, label={120:$w_3$}] at (3, 0) {};
        \end{tikzpicture}
    \caption{A chain of 3 loops, such that the top and bottom edges of the middle loop have the same length.}
    \label{fig:hyperelliptic-chain}
\end{figure}
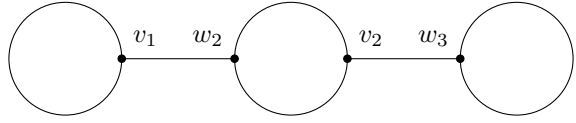

\begin{figure}[ht]
    
        \begin{tikzpicture}[scale=1.5]
            \draw (-0.5, 0) circle (0.5);
            \draw (0,0) -- (1,0);
            \draw (1.5, 0) circle (0.5);
            \draw (2,0) -- (3,0);
            \draw[dashed] (3.5, 0) circle (0.5);
            \node[label={90:$1$}] at (0.5, 0) {};
            \node[label={90:$3$}] at (2.5, 0) {};
            \node[label={90:$1$}] at (1.5, 0.5) {};
            \node[label={270:$1$}] at (1.5, -0.5) {};
            \node[label={$0$}] at (-0.5, 0) {};

        \end{tikzpicture}
    \caption{Slopes of the function $\varphi$, for any $D = K_\Gamma + \ddiv(\varphi) \in \sigma$, oriented left to right, on the first two loops and bridges.}
    \label{fig:hyperelliptic-chain-functions}
\end{figure}
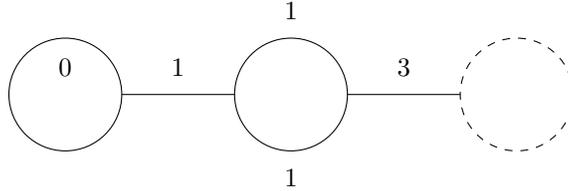

\end{example}

\section{Local matroids of matroidal linear series} \label{sec:local-matroidal}

In this section, we relate the local matroids of matroidal linear series to the initial matroids of a parametrizing valuated matroid, up to simplification and passage to submatroids.  Recall that the simplification of a matroid is the matroid obtained by deleting loops and identifying parallel elements. Let $\Sigma$ be a matroidal linear series with a parametrization $\cV \twoheadrightarrow \Sigma$.   We will show that the simplification of each local matroid of $\Sigma$ is isomorphic to a submatroid of an initial matroid of $\cV$.  This gives a local parametrization of $|\Sigma|$ near a nondegenerate divisor $D$ by the Bergman fan of an initial matroid of $\cV$. It follows, in particular, that if one of the local matroids of a tropical linear series at a nondegenerate divisor is not realizable, then the tropical linear series is not realizable. 

In addition, we show that every matroid occurs as the local matroid of a matroidal linear series at a nondegenerate divisor, and we use recent results in tropical linear incidence geometry, from \cite{Wang24}, to produce examples of matroidal linear series that are not strongly recursive.

\subsection{Local matroids and initial matroids} 
Let $\cV$ be a valuated matroid of rank $r$ on a finite set $E$.  For $w \in \cV \cap \RR^E$, the \emph{initial matroid} of $\cV$ at $w$, denoted $\cV^w$, is a loopless matroid on $E$ with flats
\[
\{ \supp(w'- w) : w' \in \cV \mbox{ and } (w' - w) \in \RR_{\geq 0}^E \}.
\]
Here, $\supp$ means the set of coordinates that are nonzero. Let $\overline w$ be the image of $w$ in the projectivization $|\cV|$. Then $\Star(\overline w)$ is isomorphic to the Bergman fan $\cB(\cV^w)$. See \cite[Definition~4.2.7]{MaclaganSturmfels}. %

To state our main result about local matroids and initial matroids, we also need the notion of a submatroid. Let $M$ be a matroid on $E$. For each subset $E' \subseteq E$ that contains a basis, the corresponding \emph{submatroid} $M'$ is the  matroid on $E'$ defined in any of the following equivalent ways.
\begin{itemize}
    \item The independent sets of $M'$ are the independent sets of $M$ contained in $E'$.
        \item The flats of $M'$ are the intersections of $E'$ with flats of $M$. 
    \item The covectors of $M'$ are the covectors of $M$ with the coordinates labeled by $E\smallsetminus E'$ deleted. 
\end{itemize}
See \cite[Proposition 7.3.1]{white1986theory} for further details on submatroids. By the characterization of submatroids in terms of covectors, the projection $\RR^E \to \RR^{E'}$ induces a surjection $\cB(M) \twoheadrightarrow \cB(M')$.  

Let $\overline M$ denote the simplification of a matroid $M$, obtained by deleting loops and identifying parallel elements. Bergman fans are defined for any loopless matroid. A loopless matroid $M$ and its simplification have isomorphic lattices of flats, so their Bergman fans are isomorphic $\cB(\overline M) \cong \cB(M)$.

\begin{theorem} \label{thm:local-is-initial}
Let $\Sigma \subseteq R(D)$ be a matroidal linear series with parametrization $\pi \colon \cV \to \Sigma$. 
\begin{enumerate}
    \item There is a unique continuous section $\sigma$ of $\pi$ over $\{ \varphi : D + \ddiv(\varphi) \mbox{ is nondegenerate} \}$.
    \item 
    Setting $\Sigma_\varphi := \{ \varphi' : \varphi + \varphi' \in \Sigma \}$, 
    the simplified local matroid  $\overline M_{\Sigma_\varphi}$ is isomorphic to a submatroid of  $\cV^{\sigma(\varphi)}$. 
\end{enumerate}
    
\end{theorem}

\noindent In particular, if $D \in |\Sigma|$ is nondegenerate, then there is a natural surjective coordinate projection  $\cB(\cV^{\sigma(0)}) \twoheadrightarrow \cB(M_{\Sigma})$ inducing a local parametrization $\Star(\sigma(0)) \twoheadrightarrow \Star(D)$.

Before proving \Cref{thm:local-is-initial}, we give an application, showing that the local matroids of realizable tropical linear series are realizable.

\begin{proof}[\bf Proof of \Cref{thm:local-realizable}]
    Suppose $\Sigma$ is the tropicalization of a linear series.  By \Cref{thm:realizablematroidal}, there is a parametrization by a realizable matroid $\cV \twoheadrightarrow \Sigma$. Then every initial matroid $\cV^w$ is realizable. Moreover, any submatroid of a realizable matroid is realizable. Thus, by \Cref{thm:local-is-initial}, for each nondegenerate $D \in |\Sigma|$ the simplified local matroid $\overline M_{\Sigma}$ is realizable. The theorem follows, since a matroid is realizable if and only if its simplification is realizable.
\end{proof}

To prove the existence part of \Cref{thm:local-is-initial}, we use the following lemma about lattices of flats. %
Given two flats $F$ and $G$ of a matroid $M$, their join $F\vee G$ is the smallest flat $\cL(M)$ containing $F$ and $G$, and their meet $F\wedge G$ is the intersection $F\cap G$. A join-subsemilattice of $\cL(M)$ is a subset $\cL'\subseteq \cL(M)$ such that if $F,G\in \cL'$, then $F\vee G\in \cL'$.

\begin{lemma}\label{lem:surjection_of_lattices}
	Suppose $M_1$ and $M_2$ are matroids of the same rank. If $\tau\colon \cL(M_1) \to \cL(M_2)$ is meet-preserving and surjective, then $\overline M_2$ is isomorphic to a submatroid of $M_1$. 
\end{lemma}
\begin{proof}%
Given a meet-preserving map $\tau\colon \cL(M_1) \to \cL(M_2)$, the map
\[\tau^*\colon \cL(M_2)\to \cL(M_1),\quad x\mapsto \wedge\{y\in \cL(M_1) : \tau(y)\geq x\}\]
is join-preserving, and it is injective if $\tau$ is surjective. See, e.g., the introduction to \cite{higgs1968}. 
Since $M_1$ and $M_2$ have the same rank, $\tau^*$ sends a maximal chain in $\cL(M_2)$ to a maximal chain in $\cL(M_1)$. In particular, rank 1 flats of $M_2$ are sent to rank 1 flats of $M_1$. This means $\cL(M_2)$ embeds in $\cL(M_1)$ as a join-subsemilattice. The conclusion follows from \cite[Proposition 7.3.1]{white1986theory}.
\end{proof}

\begin{proof}[\bf{Proof of \Cref{thm:local-is-initial}}] 
Suppose $\varphi \in \Sigma$. Then $\Sigma_\varphi := \{ \varphi' : \varphi' + \varphi \in \Sigma \}$ is a matroidal linear series contained in $R(D + \ddiv(\varphi))$. We claim that there is some $w \in \cV$ such that  $\overline M_{\Sigma_\varphi}$ is isomorphic to a submatroid of the initial matroid  $\cV^{w}$. It suffices to consider the case where $\varphi = 0$ and $\Sigma_\varphi = \Sigma$.

  Let $H_1,\ldots,H_k$ be the corank 1 flats of $M_{\Sigma}$. For each $H_i$, choose $\varphi_i\in \Sigma$ such that $H_i=F_{\varphi_i}$ and  $\min \varphi_i=0$. Choose $ w_i\in \cV$ such that $\pi(w_i)=\varphi_i$. Recall that $\cV$ is the space of covectors of a valuated matroid and hence is a tropical submodule of $\Rbar^{N}$ for some $N$. Therefore, applying coordinatewise minimum (the addition operation on the tropical submodule), $\cV$ contains
\begin{equation} \label{eq:defw}
w:=\min\{ w_1,\ldots,  w_k\}.
\end{equation}
Note that $\pi(w)=\min\{\pi(w_1),\ldots, \pi(w_k)\}=\min\{\varphi_1,\ldots,\varphi_k\}=0$. The last equality holds because $M_{\Sigma}$ is loopless, and the intersection of all corank 1 flats of a loopless matroid is empty.
We will now show that $\overline M_{\Sigma}$ is isomorphic to a submatroid of the initial matroid $\cV^w$. By \Cref{lem:surjection_of_lattices}, it suffices to produce a meet-preserving surjective map of lattices $\cL(\cV^w) \to \cL(M_{\Sigma})$.

By translating $\cV$, we may assume that $w=0$. Note that $\{w_1, \ldots, w_k\} \subseteq \Rbar_{\geq 0}^N$, by \eqref{eq:defw}, and  $\supp(w_i)$ is a flat of $\cV^w$.  We construct a map $\td{f}\colon \cL(\cV^w)\to \cL(M_{\Sigma})$ as follows. Pick $\epsilon>0$ sufficiently small. Then, for each flat $G \in \cL(\cV^w)$, we have $\epsilon 1_{G}$ in $\cV$. Let $\varphi_G := \pi(\epsilon 1_G)$.  Because $\pi(0) = 0$ and $\pi$ is a homomorphism of tropical modules, we have $\varphi_G \geq 0$.  We define
\[
\td{\pi}(G) := \left\{ \begin{array}{ll} F_{\varphi_G}  & \mbox{ if } \min \varphi_G = 0, \\
E & \mbox{ otherwise.} \end{array} \right.
\]
By construction, $\td{\pi}$ is meet-preserving %
and $\td{\pi}(\supp(w_i)) = H_i$.  Since the corank 1 flats generate $\cL(M_{\Sigma})$ with respect to meet, it follows that $\td{\pi}$ is surjective.  This proves the claim.

By the claim proved above, there is a set-theoretic section $\sigma$ of $\pi$ over the open set $$ U := \{ \varphi : D + \ddiv(\varphi) \mbox{ is nondegenerate} \}$$ such that the local matroid associated to $\Sigma_\varphi$ is isomorphic to a submatroid of $\cV^{\sigma(\varphi)}$ for all $\varphi \in U$. 
To prove the theorem, it remains to show that this section is unique and continuous.

Note that there is an open dense subset $U' \subseteq U$ over which $\pi$ is a homeomorphism. This is because $\pi$ is a piecewise linear map between polyhedral spaces of the same dimension with connected fibers. The section $\sigma$ necessarily agrees with $\pi^{-1}$ on $U'$. Hence, if $\sigma$ is continuous then it is unique.

It remains to show that $\sigma$ is continuous. %
Let $w = \sigma(\varphi)$. By construction, $\pi$ maps a small neighborhood $W$ of $w$ onto a small neighborhood $U''$ of $\varphi$. Since $U'$ is dense, the intersection $U' \cap U''$ is nonempty.  Thus, $\sigma(U') \cap W$ is nonempty. Since this holds for arbitrarily small neighborhoods $W$, it follows that $w$ is in the closure of $\sigma(U')$ and is the limit of $\sigma(\varphi_i)$ for any sequence $\{ \varphi_i \}$ in $U'$ converging to $\varphi$. This shows that $\sigma$ is continuous, as required.%
\end{proof}

\subsection{Every matroid is a local matroid}

Here, we study the local matroids of tropical and matroidal linear series on an interval. In particular, we prove \Cref{thm:every-matroid}, showing that every loopless matroid appears as the local matroid of a tropical linear series at a nondegenerate divisor on an interval and also on a loop. Moreover, these tropical linear series can be chosen to be matroidal.

Let $[v,w]$ be a metric graph consisting of a single edge with endpoints $v$ and $w$.  Let $D = dv$, let $x$ denote the linear function with slope $1$ that takes the value $0$ at $v$, and let $$\Phi \colon \Rbar^{d+1} \to R(D)$$ be the surjective tropical linear map given by 
$
(a_0, \ldots, a_d) \stackrel{\Phi}{\mapsto} \min\{ a_0, x + a_1, 2x + a_2, \ldots, dx + a_d \}.
$ 
Since $\Rbar^{d+1}$ is a valuated matroid of rank $d+1$ and $R(D)$ has Baker--Norine rank $d$, $R(D)$ is a matroidal linear series of dimension $d$.

\begin{theorem}\label{thm:intervalProof}
Let $\cV \subseteq \Rbar^{d+1}$ be a valuated matroid of rank $r+1$.  Then $\Phi(\cV) \subseteq R(D)$ is a matroidal linear series of dimension $r$. 
\end{theorem}

\begin{proof}
It suffices to show that $\Phi (\cV)$ has Baker-Norine rank at least $r$. The set $|\Phi(\cV)|$ is compact and hence the set 
\[
\{D'\in\operatorname{Sym}^r[v,w]: D -D' + \ddiv(\varphi) \geq 0\text{ for some } \varphi \in \Phi(\cV) \}
\]
is closed. Therefore, it suffices to check the Baker--Norine rank condition for effective divisors of the form $D'=p_1+\cdots+p_r$ where $p_i$ are mutually distinct and $p_i\in(v,w)$ for each $i$. Let $x_i$ be the distance from $v$ to $p_i$ and let $a = (a_0, \ldots, a_d)$ be a point in $\Rbar^{d+1}$. Then  $D + \ddiv(\Phi(a)) -p_i \geq 0$ if and only if $\min\{a_0, x_i + a_1, 2x_i + a_2, \ldots, dx_i + a_d \}$
is attained at least twice. We rephrase this geometrically as follows. Let $H_i \subseteq\Rbar^{d+1}$ be the tropical hyperplane given as the bend locus of the tropical linear form
\[\min\{b_0, x_i + b_1,2x_i + b_2,..., dx_i + b_d\}.
\]
Then $D + \ddiv(\Phi(a)) -p_i \geq 0$ if and only if $a \in H_{i}$.
 
 Since $\cV$ has rank $r+1$, the stable intersection $\cV \cap_{\textnormal{st}} H_{1} \cap_{\textnormal{st}} \cdots \cap_{\textnormal{st}} H_{r}$ is nonempty. A point $a$ in this stable intersection has  $D + \ddiv(\Phi(a)) - p_1 - \cdots - p_r \geq 0$. It follows that $r_{\BN}(\Phi(\cV)) \geq r$, and hence $\Phi(\cV)$ is a matroidal linear series of dimension $r$. 
\end{proof}

The map $\Phi \colon \cV \to \Phi(\cV)$ is a parametrization. Note that $|D|$ is a $d$-simplex whose interior parametrizes nondegenerate divisors. \Cref{fig:interval} illustrates the parametrizing map $\Phi\colon \Rbar^{d+1}\to R(D)$ when $d = 2$. In this case, the complete linear series $|D|$ is a 2-simplex with vertices $2v,v+w$ and $2w$. The domain is partitioned into seven regions. The regions $R_1$, $R_2$, and $R_3$ map to the vertices of $|D|$, and the regions $R_{12}$, $R_{13}$, and $R_{23}$ map to the edges. The projectivization of $R_0$ maps isomorphically to the interior of $|D|$. %
See \cite{applet-tls} for an applet visualizing tropical linear series of degree 2 on an interval.

\begin{figure}
    \centering
    \includegraphics[width=0.5\linewidth]{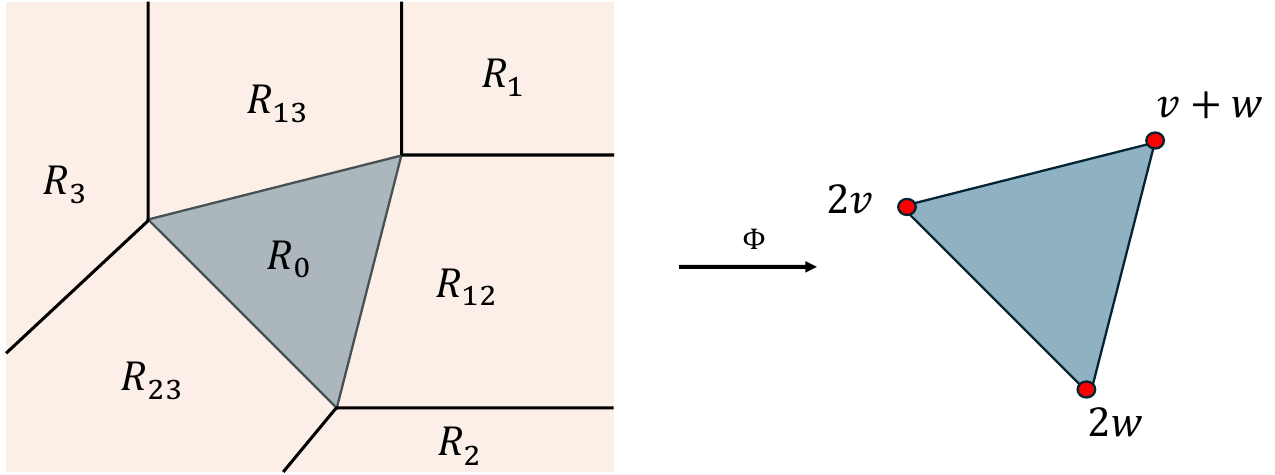}
    \caption{The parametrizing map $\Phi\colon \Rbar^3\to R(2v)$. %
    }
    \label{fig:interval}
\end{figure}

\begin{example}\label{ex:local-matroid-interval} \Cref{fig:local-matroids} illustrates two parametrizations of a matroidal linear series of dimension~1. The first parametrization $\Phi'$ surjects the uniform matroid $U_{2,3}$ onto $\Sigma$. The point $(0,0,0)$ is sent to the interior of $|D|$, and the local matroid of $\Sigma$ at $f(0,0,0)$ is  $U_{2,3}$, which agrees with the initial matroid of $U_{2,3}$ at $(0,0,0)$. Now consider the map $f$ which surjects the uniform matroid $U_{2,4}$ onto $U_{2,3}$ by forgetting the first coordinate. Precomposing $\Phi'$ with the map $f$, we get another parametrization by the matroid $U_{2,4}$, where the local matroid at $\Phi'(f(0,0,0,0))$ is a submatroid of the initial matroid of $U_{2,4}$ at $(0,0,0,0)$.
\end{example}

\begin{figure}
    \centering
\includegraphics[width=0.9\linewidth]{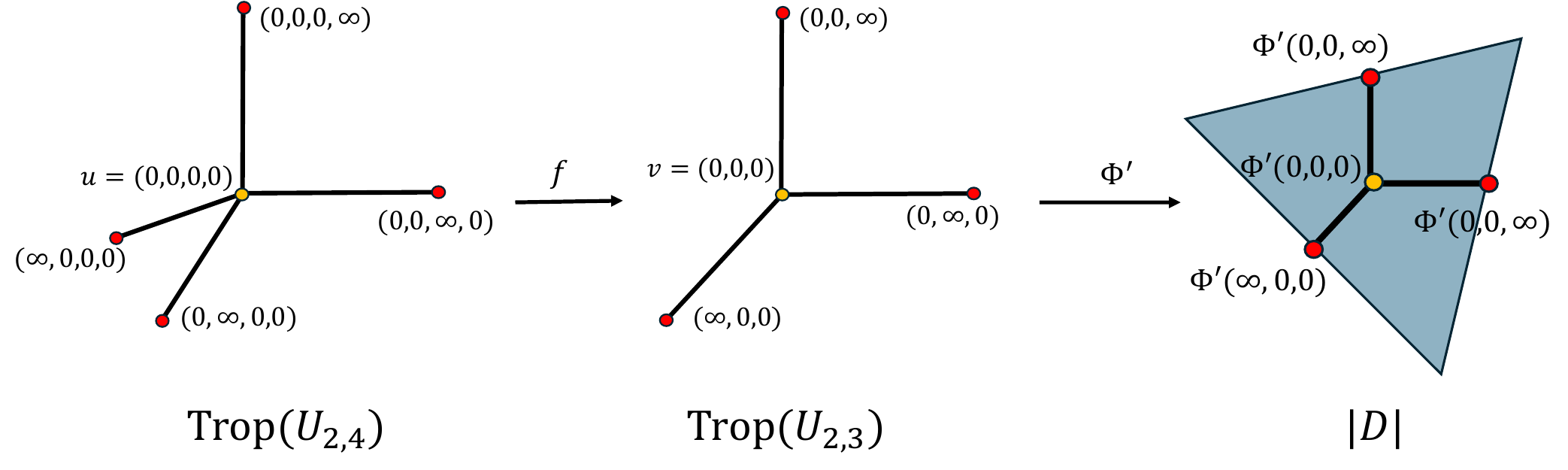}
    \caption{Matroidal linear series of degree 2 on an interval.}
    \label{fig:local-matroids}
\end{figure}

We now proceed to show \Cref{thm:every-matroid}, which states that every loopless matroid is isomorphic to the local matroid of a tropical linear series at a non-degenerate divisor.

\begin{proof}[\bf Proof of \Cref{thm:every-matroid}] 
Let $M$ be a loopless matroid on $d+1$ elements, and let $\Phi'$ be a translation of $\Phi$ such that $D'=D + \ddiv(\Phi'(0))$ is nondegenerate.  It suffices to show that the local matroid of $\Phi'\big(\Trop(M)\big)$ at $D'$ is $M$.  A small neighborhood of 0 in $\Trop(M)$ is mapped isomorphically to a small neighborhood of $\Phi'(0)$. By \Cref{thm:local-structure}, $M$ is the local matroid of $\Phi'\big(\Trop(M)\big)$ at $\Phi'(0)$.
\end{proof}

The arguments above also apply to the case where the graph is a loop.  Let $\Gamma$ be a loop.  Any divisor of degree $d$ on $\Gamma$ is equivalent to $dv$ for some $v \in \Gamma$. Let $D := dv$. Then the extremals of $|D|$ are of the form $dv'$ where $d$ times the distance from $v$ to $v'$ is an integer multiple of the length of the loop. This gives rise to a surjection $\Phi \colon \Rbar^d \twoheadrightarrow R(D)$.

\begin{theorem}\label{Thm:loop}
Let $\Gamma$ be a loop and let $D$ be an divisor of degree $d\geq 0$ on $\Gamma$.  Let $\Phi \colon \Rbar^d \to R(D)$ be the tropical linear map induced by the minimal generating set of $R(D)$. Then for any valuated matroid $\cV \subseteq \Rbar^d$ of rank $r+1$, $\Phi(\cV)$ is a matroidal linear series of dimension $r$.  
\end{theorem}

\begin{proof}
The proof is almost identical to that of \Cref{thm:intervalProof}. In this case, $r_{\BN}(R(D)) = d - 1$. Moreover, for any point $p$ such that $d$ times the distance from $p$ to $v$ is not an integer multiple of the length of the loop, the set of points $a \in \Rbar^d$ such that $D - p + \ddiv(\Phi(a)) \geq 0$ is a hyperplane $H_p$. Then, a stable intersection argument shows that $r_{\BN}(\Phi(\cV)) = r$, and the theorem follows.
\end{proof}

\subsection{Incidence geometry in tropical linear series} Here, we illustrate surprising phenomena for incidence geometry of tropical linear subseries of a tropical linear series, drawing upon recent results in tropical linear incidence geometry from \cite{Wang24}. Recall that the tropical linear subspaces of $\RR^{E}/(1,\ldots,1)$ are the projectivizations of valuated matroids $\cV$ on $E$. In particular, the tropical linear spaces that are fans are the Bergman fans of matroids.

\begin{itemize}
    \item If a matroid $M$ of rank $r+1$ does not have the Levi intersection property, then there is a set of $r$ points in the tropical linear space $\Trop(M)$ that is not contained in any codimension 1 tropical linear subspace. The V\'amos matroid $V_8$ is of this type. See \Cref{fig:vamos}.
    \item There are matroids $M$ such that $\Trop(M)$ contains two codimension 1 tropical linear subspaces $H_1$ and $H_2$ such that $H_1 \cap H_2$ does not contain any codimension 2 tropical linear subspace. The matroid $V_8^-$, a relaxation of $V_8$, is of this type. See \Cref{fig:vamos-relaxation}.
\end{itemize}

Our first example is a tropical linear series that is not strongly recursive. We follow the common practice of specifying a matroid of rank $r$ on $E$ by listing its non-spanning circuits, i.e., the minimal dependent sets that are contained in a proper flat $F \subsetneq E$. The remaining circuits are the sets of size $r +1$ that do not contain a non-spanning circuit.

\begin{example}\label{ex:vamos-series}
    The V\'amos matroid $V_8$,  whose affine diagram is given in \Cref{fig:vamos}, is the matroid on $E=\{0,\ldots,7\}$ whose non-spanning circuits are:
    $$\{0,1,2,3\},\{0,3,4,5\},\{1,2,4,5\},\{0,3,6,7\},\{1,2,6,7\}.$$
    \begin{figure}
        \centering
        \includegraphics[width=0.25\linewidth]{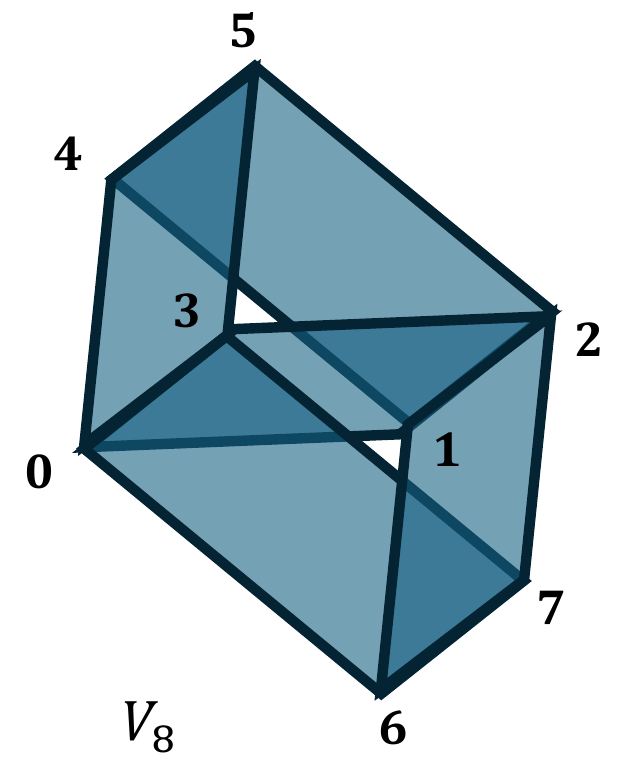}
        \caption{The affine diagram depicting the non-spanning circuits of the Vámos matroid $V_8$.}
        \label{fig:vamos}
    \end{figure}
    Choose 7 distinct points $p_1<p_2<\cdots <p_7$ in $(v,w)$. %
    Let $D = p_1+p_2+\cdots +p_7$. Then $D$ is nondegenerate. Suppose $D=7v + \ddiv \varphi$ where $\varphi=\min\{a_0,x+a_1,\ldots,7x+a_7\} \in R(7v)$.
    Define the tropical linear map $\Phi' \colon \Rbar^{E}\to R(7v)$ by 
    \[
    \Phi' (b_0 , \ldots , b_7 ) = \min\{ a_0+b_0,x+ a_1 + b_1,\ldots,7x + a_7 + b_7\} .
    \]
    In other words, $\Phi'$ is a translation of $\Phi$ by $(a_0 , \ldots , a_7)$. By \Cref{thm:intervalProof}, $\Phi'(\Trop (V_8))$ is a matroidal linear series of dimension three, and the local matroid $M_{\Phi'(Trop(V_8))} = V_8$.
    
    Equivalently, using the isomorphism $R(7v) \cong R(D)$, we regard $\Phi'$ as a map $\Rbar^{E}\to R(D)$, sending $0$ to the constant function $0$. For the three flats $F_1=\{0,1,2,3\}$, $F_2=\{1,2,4,5\}$, and $F_3=\{0,3,4,5\}$ of $V_8$, let $f_i=\delta_{F_i}$ where 
     \[\delta_{F_i}(j)=\begin{cases}
        \infty, & j\in F_i \\
        0,& \text{ else}.
    \end{cases}\]
    Put 
    $D_i = 7v + \Phi'(f_i)$ for $1\leq i\leq 3$. We show that there cannot be a tropical linear series $\Sigma \subseteq \Phi' (\Trop (V_8))$ of dimension 2 such that $D_1 , D_2 , D_3 \in |\Sigma|$. Let $\Sigma$ by any tropical submodule containing $f_1,f_2$ and $f_3$. We have $\min\{f_1,f_2,f_3\}=0\in \Sigma$, so $D\in |\Sigma|$. Note that $F_{\Phi'(f_i)}=F_i$, meaning that $F_1,F_2$ and $F_3$ are all flats of $M_{\Sigma}$. By \cite[Theorem~6.4]{Wang24}, the matroid $V_8$ does not have a quotient of rank 3 that contains these three flats. We conclude by \Cref{cor:incidence-local-matroid}, that $\Sigma$ cannot be a tropical linear series of dimension 2.
\end{example}

Our second example is a tropical linear series with two codimension 1 tropical linear subseries whose intersection does not contain any tropical linear series of codimension 2.  We use the following criterion for a tropical submodule $\Sigma\subseteq R(D)$ to have big minimizers. %

\begin{lemma} \label{prop:dagger-criterion}
   If $D$ is multiplicity free and $\supp (D)$ does not contain any vertex of valence 1, then any tropical submodule $\Sigma\subseteq R(D)$ has big minimizers.
\end{lemma}

\begin{proof}
   Let $\varphi\in R(D)$. We will show that $\varphi_{\min}$ contains no isolated points. By \Cref{lem:minimizer}, $\varphi_{\min}$ is a union of points in $\supp (D)$ and closures of connected components of $\Gamma \smallsetminus \supp (D)$. If $x\in \varphi_{\min}$, then all the outgoing slopes of $\varphi$ at $x$ are nonnegative. Since $D$ is multiplicity-free and $\supp(D)$ does not have vertex of valence 1, this means $\varphi_{\min}$ does not contain isolated points, and the lemma follows.
\end{proof}

\begin{example}\label{ex:vamos-relaxiation-series}
   Consider the interval $[v,w]$, with $D = p_1 + \cdots + p_7$ and $\Phi' \colon\Rbar^{8}\to R(D)$, as in \Cref{ex:vamos-series}. Let $V_8^-$ be the relaxation of $V_8$ obtained by omitting the non-spanning circuit $\{1,2,6,7\}$. Let $Q_1$ and $Q_2$ be the two elementary quotients of $V_8^-$ whose affine diagrams are given in \Cref{fig:vamos-relaxation}.  Let $\Sigma = \Phi(\Trop(V_8^-))$ and $\Sigma_i=\Phi'(\Trop(Q_i))$ for $i=1,2$.  Then $\Sigma_1$ and $\Sigma_2$ are tropical linear subseries of codimension 1 in $\Sigma$. We claim that $\Sigma_1 \cap \Sigma_2$ does not contain  any tropical linear subseries of codimension 2. %
   
 The divisor $D$ is multiplicity-free and its support does not contain any vertex of valence 1. Let $\Sigma'$ be a tropical linear series contained in $\Sigma_1\cap \Sigma_2$. By \Cref{prop:dagger-criterion}, the local matroid $M_{ \Sigma'}$ is well-defined. 
 The lattice of flats of $M_{\Sigma'}$ is contained in the intersection of the lattices of flats of $M_{\Sigma_1}$ and $M_{\Sigma_2}$.
 As in the proof of \Cref{thm:every-matroid}, the local matroids are $M_{\Sigma_1}=Q_1$ and $M_{\Sigma_2}= Q_2$. \Cref{cor:incidence-local-matroid} says that the local matroid $M_{\Sigma'}$ is a common quotient of $Q_1$ and $Q_2$. However, by \cite[Example~6.14]{Wang24}, there is no corank 2 quotient of $V_8^-$ which is a common quotient of $Q_1$ and $Q_2$. Therefore, $\Sigma'$ cannot have codimension 2. 
\end{example}

\begin{remark} \label{rmk:DnotinSigma}
In Example~\ref{ex:vamos-relaxiation-series}, it is crucial that we can define the local matroid $M_{\Sigma'}$ and prove its basic properties, such as being a quotient of $M_{\Sigma}$, when $D$ is not necessarily contained in $|\Sigma'|$.
\end{remark}

\begin{figure}
    \centering
    \includegraphics[width=0.5\linewidth]{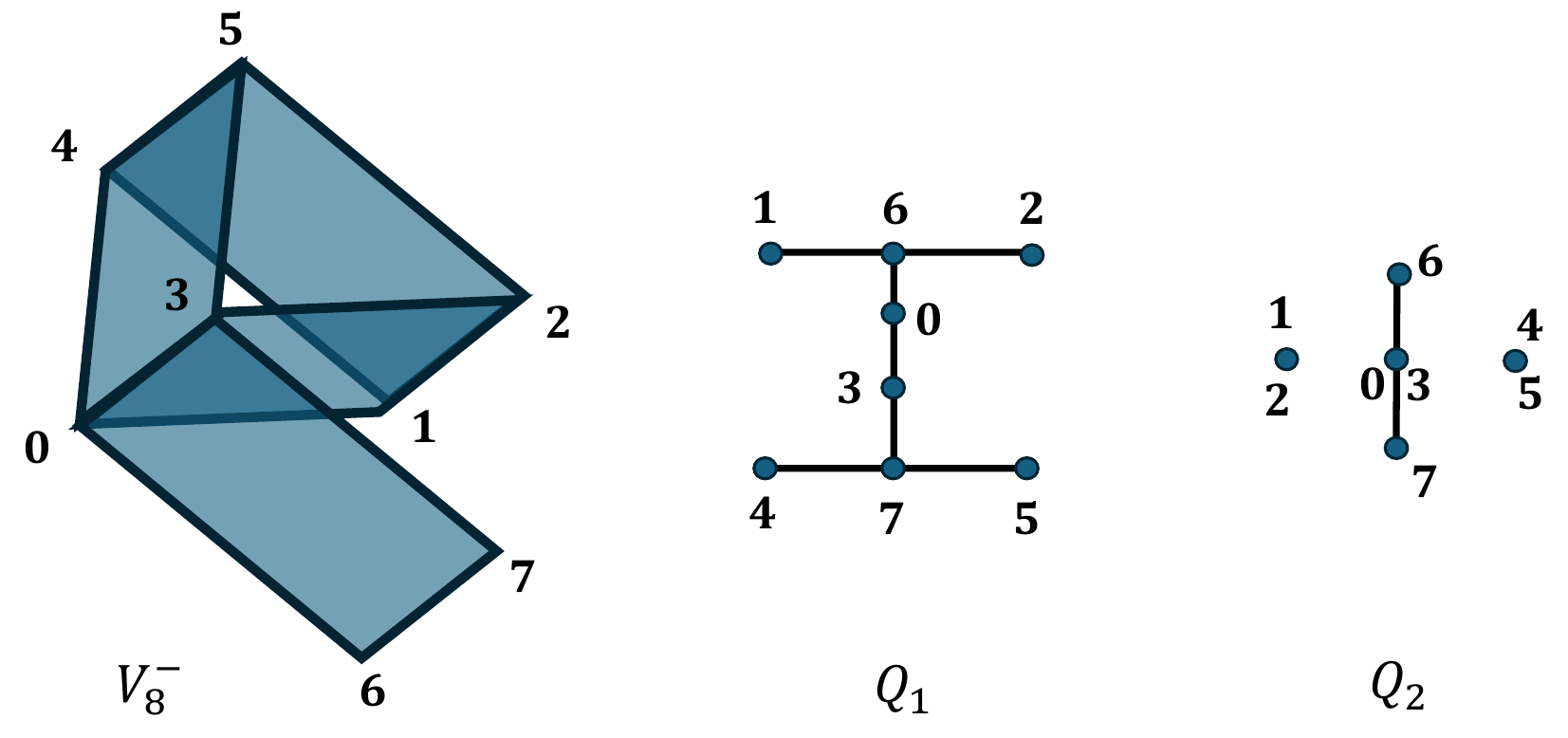}
    \caption{The affine diagrams depicting the non-spanning circuits of $V_8^-$ and its elementary quotients $Q_1$ and $Q_2$.} %
    \label{fig:vamos-relaxation}
\end{figure}

\section{Cartwright divisors and matroid adjoints}\label{sec:Cartwright}

In \cite{Cartwright15}, Cartwright studied a rank 2 divisor on the Levi graph of a rank 3 matroid, and showed that it is the tropicalization of a rank 2 divisor on an algebraic curve if and only if the matroid is realizable. In this subsection, we revisit these Cartwright divisors and relate the associated $2$-dimensional tropical linear series to matroid adjoints.

\subsection{Cartwright divisors on Levi graphs} We begin by briefly reviewing Cartwright's construction. Let $M$ be a simple rank 3 matroid on a finite set $E$. The \emph{Levi graph} $\Gamma_M$ is the bipartite incidence graph between the ground set $E$ and the set of rank 2 flats $H$. We write $[e,F]$ for the edge of $\Gamma_M$ connecting $e \in E$ to $F \in H$, when $e \in F$.  Cartwright studied the case where all edge lengths are 1. Here, we consider $\Gamma_M$ as a metric graph with arbitrary positive edge lengths. %

The \emph{Cartwright divisor} of the simple rank $3$ matroid $M$ is $$D_M := \sum_{e \in E} e.$$ It has rank $r_{\BN}(D_M) = 2$.  This was proved by Cartwright when all edge lengths are equal to 1 \cite[Proposition~2.2]{Cartwright15}. To extend this to arbitrary edge lengths, note that the Levi graph of a rank 3 matroid has girth 6 and apply \cite[Corollary~2.2]{Jensen16}.
In this section, we study the tropical linear series $\Sigma \subseteq R(D_M)$ of dimension 2 and relate their local matroids at $D_M$ to adjoints of $M$.

\subsection{Matroid adjoints}

Let $H = \{H_1, \ldots, H_s \}$ be the set of corank $1$ flats of a simple rank $r$ matroid $M$. An \emph{adjoint} of $M$ is a rank $r$ matroid $W$ on $H$ such that, for each rank $k$ flat $F$ of $M$, 
\begin{equation}\label{eq:adjoint-flats}
\{ H_i : F \subseteq H_i \}
\end{equation}
is a corank $k$ flat of $W$. %
Any realization of $M$ gives rise to an adjoint (and a realization of this adjoint), by viewing the realization as a point configuration spanning $\PP^r$ and interpreting the corank $1$ flats of the realization as a point configuration in the dual projective space. Conversely, if $M$ has a realizable adjoint, then $M$ is realizable  \cite[Corollary 3.2]{bachem1986extension}.

In general, a nonrealizable matroid may or may not have an adjoint. Moreover, when an adjoint exists, it is typically not unique. See \cite{fu2024adjoints} for further details on matroid adjoints.

We recall that every rank $3$ matroid $M$ has an adjoint. One of these, called the \emph{free adjoint}, is maximal in the sense that any basis of any adjoint is a basis of the free adjoint. The rank 2 flats of the free adjoint are those of the form \eqref{eq:adjoint-flats}, together with the disjoint pairs $\{H_i, H_j \}$.

\subsection{Local matroids at Cartwright divisors}

Let $M$ be a simple rank $3$ matroid, let $D_M$ be the Cartwright divisor on the Levi graph $\Gamma_M$, and let $\Sigma \subseteq R(D_M)$ be a rank 2 tropical linear series. By \Cref{prop:dagger-criterion}, any tropical linear series $\Sigma \subset R(D_M)$ has big minimizers, so the local matroid $M_{\Sigma}$ is well-defined. The connected components of $\Gamma_M \smallsetminus \supp(D_M)$ are the open subgraphs $\gamma_F$ consisting of a vertex $F \in H$ together with the interiors of its adjacent edges. %
When no confusion seems possible, we identify $\{\gamma_F:F\in H\}$, i.e., the ground set of the local matroid $M_{ \Sigma}$, with $H$. Thus, it makes sense to ask whether or not $M_{\Sigma}$ is an adjoint of $M$.

\begin{theorem}\label{thm:local-matroid-adjoint}
    Let $M$ be a simple rank 3 matroid and let $\Sigma\subseteq R(D_M)$ be a tropical linear series of dimension 2. Then $D_M\in |\Sigma|$ and $M_{\Sigma}$ is an adjoint of $M$.
\end{theorem}

\noindent Since any adjoint of a nonrealizable matroid is nonrealizable, combining Theorems~\ref{thm:local-realizable} and \ref{thm:local-matroid-adjoint} recovers Cartwright's non-realizability result in the case where all edge lengths are 1 and extends this to the case of arbitrary edge lengths.

\begin{corollary} \label{cor:Cartwright-not-realizable}
    If $M$ is not realizable, then $R(D_M)$ does not contain any realizable tropical linear series of dimension 2.
\end{corollary}

We also prove the following partial converse to \Cref{thm:local-matroid-adjoint}. %

\begin{theorem}\label{thm:adjoint-and-linear-series}
    Suppose that all edges of $\Gamma_M$ have length 1.  If $W$ is an adjoint of $M$, then there is a matroidal linear series $\Sigma \subseteq R(D_M)$ such that $M_{\Sigma}=W$.
\end{theorem}

\begin{remark}
When the edges of $\Gamma_M$ do not all have equal length, we do not know whether there exists a tropical linear series $\Sigma \subseteq R(D_M)$ of dimension 2.
\end{remark}

\medskip

In the arguments below, we use the notion of reduced divisors and the burning algorithm for metric graphs, following \cite{luo2011rank,baker2013chip}. Let $v\in \Gamma$ be any point. An effective divisor $D$ is \emph{$v$-reduced} if for every closed subgraph $A\subseteq\Gamma \smallsetminus \{v\}$, there is a point $x \in A$ such that 
\[D(x) < \outdeg_A(x).\]
Here, $\outdeg_A(x)$ is the degree of $x$ in $\Gamma$ minus the degree of $x$ in $A$. The following lemma relates reduced divisors to the minimizers of functions in $R(D)$.

\begin{lemma}\label{lem:reduced-divisor-minimizer}
   Let $D$ be an effective divisor on $\Gamma$. Then  
   \[\{v\in \Gamma:D\text{ is $v$-reduced}\}=\bigcap_{\varphi\in R(D)}\varphi_{\min}.\]
\end{lemma}

\begin{proof}%
    Suppose $D$ is $v$-reduced. Let $\varphi\in R(D)$. Since $D+\ddiv(\varphi)$ is effective, $\varphi_{\min}$ is a closed subgraph satisfying $D(x)\geq \outdeg_{\varphi_{\min}}(x)$ for all $x \in \varphi_{\min}$. Therefore, $v\in \varphi_{\min}$. 
    
    Conversely, suppose $D$ is not $v$-reduced. Then there is some closed subgraph $A$ such that $D(x)\geq \outdeg_{A}(x)$ for all $x \in A$. Let $\varphi$ be the function that takes the value 0 on $A$, has slope 1 for a sufficiently small distance $\epsilon$ on every tangent direction leaving $A$, and takes constant value $\epsilon$ elsewhere. By construction, $\varphi \in R(D)$ and $v \not \in \varphi_{\min}$.
\end{proof}

\begin{proof}[{\bf Proof of \Cref{thm:local-matroid-adjoint}}]
    We first show that $D_M\in |\Sigma|$. By \Cref{thm:local_matroid}, it is enough to show that the local matroid $M_{\Sigma}$ is loopless. Recall that the ground set of $M_{\Sigma}$ is the set of open subgraphs $\{ \gamma_F : F \mbox{ is a rank 2 flat of $M$} \}$. %

    Let $F$ be a rank 2 flat of $M$. We will show that $\gamma_F$ is not a loop of $M_{\Sigma}$.  Let $e_1 \neq e_2\in F$. We claim that $D' = D_M - e_1 - e_2$ is $v$-reduced for all $v \in \gamma_F$.  To see this, apply Dhar's burning algorithm.  A fire started at $v$ will burn through $F,e_1$ and $e_2$. Then it burns through all flats containing $e_1$ or $e_2$. If $e \in E$ is not contained in $F$, then there is a rank 2 flat $G_1$ containing $\{e,e_1\}$ and a distinct rank 2 flat $G_2$ containing $\{e,e_2\}$, which the fire burns through.  The fire thus burns through $e$, and then all rank 2 flats containing $e$.  Thus, the fire burns every element of $E$ that is not contained in $F$, and every element of $H$. Finally, since every element of $E$ is contained in at least two rank 2 flats, the fire burns through all remaining elements of $E$. Since the whole graph burns, $D'$ is $v$-reduced. This proves the claim.

    Now, since $r_{\BN}(\Sigma) = 2$, there is some $\varphi \in \Sigma$ such that $D' + \ddiv(\varphi)$ is effective.  Applying \Cref{lem:reduced-divisor-minimizer} for $D'$, it follows that $\gamma_F \subseteq \varphi_{\min}$. Hence $\gamma_F$ is not contained in the flat $F_{\varphi}$ and is not a loop of the local matroid. It follows that $M_{\Sigma}$ is loopless and $D_M \in |\Sigma|$.

    It remains to show that $M_{\Sigma}$ is an adjoint of $M$, i.e., we must show that $\{ \gamma_F : e \in F \}$ is a flat of $M_{\Sigma}$, for each $e \in E$.  To see this, fix an element $e \in E$.  For each $e' \in E \smallsetminus \{ e \}$, let $F' \in H$ be the unique rank 2 flat containing $\{ e,e' \}$.  Fix a positive real number $c$ that is smaller than the length of $[e',F']$ for all $e'$, and let $\varphi_{e,c}$ be the function that takes the value 0 on $\cup_{e \notin F} \gamma_F$, has slope 1 for distance $c$ on each edge $[e',F']$, and takes the value $c$ outside.  The divisor
\begin{equation}\label{eqn:distinguished-divisor}
    D_{e,c}:=D_M+\ddiv(\varphi_{e,c})
\end{equation}
is multiplicity free. Its support consists of $e$, together with the point in $[e',F']$ at distance $c$ from $e'$, whenever $F'$ is a rank $2$ flat containing the pair $\{e, e'\}$.
    It is enough to show that the function $\varphi_{e,c}$ is in $\Sigma$, because then $F_{\varphi_{e,c}} = \{ \gamma_F : e \in F \}$.
    
   Let $e'\in F\smallsetminus\{e\}$. Let $p$ be in the edge $[e',F]$ at distance $c$ from $e'$, with $0 < c< 1$. Let $D'' := D_{e,c} - e - p$, with $D_{e,c}$ as in \eqref{eqn:distinguished-divisor}. Note that $D''$ is effective and multiplicity free. We then observe that $ \Gamma_M \smallsetminus \supp(D'')$ is connected; the proof is straightforward and similar to the burning algorithm argument above.  
Since $D''$ is multiplicity free and $\Gamma_M \smallsetminus \supp(D'')$ is connected, the divisor $D''$ is $v$-reduced for all $v$. Therefore, by  \Cref{lem:reduced-divisor-minimizer}, $D''$ is rigid. In other words, $D_{e,c}$ is the unique divisor in $|D_M|$ whose support contains $e$ and $p$. Since $r_{\BN}(\Sigma) = 2$, it follows that $D_{e,c} \in |\Sigma|$ and  $\varphi_{e,c} \in \Sigma$, as required.  %
\end{proof}

\begin{proof}[{\bf Proof of Theorem~\ref{thm:adjoint-and-linear-series}}]
Let $W$ be an adjoint of $M$. We construct a homomorphism of tropical modules $\Phi \colon \Trop(W) \to R(D_M)$ whose image is a matroidal linear series $\Sigma$ with $M_{\Sigma} = W$.

For $e \in E$, let $\varphi_e \in \PL(\Gamma_M)$ be the function that is linear on each edge and takes the following values at vertices:
\begin{equation} \label{eq:phiedef}
\varphi_e(e) = 2, \quad \varphi_e(F) = 1 \mbox{ if $e \in F$}, \quad \mbox{ and } \varphi(v) = 0 \mbox{ otherwise.}
\end{equation}
Here, $F$ denotes a rank $2$ flat of $M$.  If $\varphi_e \in \Sigma$, then the corresponding flat of the local matroid is
\[
H_e := \{ F : e \in F \}.
\]

For each rank 2 flat $G$ of $W$ that is not of the form $H_e$, we define a function $\varphi_G$ as follows. The flat $G$ is a collection of pairwise disjoint rank 2 flats of $M$. The function $\varphi_G$ takes the following values at vertices
\[
\varphi_G(F) = \begin{cases}
    1,& \text{ if }F\in G \\
    0, & \text{ otherwise,}
\end{cases}
\]
and is linear on each edge. 

Let $\Sigma \subseteq R(D_M)$ be the submodule generated by
\[\{ \varphi_e : e \in E \} \cup \{ \varphi_G : G \text{ is a rank 2 flat of $W$ not of the form }H_e\}.\] 
Note that $\Sigma_M$ contains the constant functions, because $\min_{e} \varphi_e = 0$.  

We claim that $\Sigma$ is a matroidal linear series of dimension $2$. 
To see that $r_{\BN}(\Sigma) = 2$, let $[e_1,F_1] , [e_2,F_2]$ be edges in $\Gamma_M$.  It suffices to show that, for two arbitrary points $v_1 \in [e_1,F_1] , v_2 \in [e_2,F_2]$, there is a function $\varphi \in \Sigma$ such that $\ddiv (\varphi) + D_M - v_1 - v_2 \geq 0$.  We break this into cases:

\begin{itemize}
\item  Suppose $e_1 = e_2$ and $F_1=F_2$.  Then $\varphi$ can be chosen to be a tropical linear combination of $0$, $\varphi_e$, and $\varphi_{e'}$, where $e' \neq e_1$ is contained in the flat $F_1$.
\item  Suppose $e_1=e_2$ and $F_1\neq F_2$.  Let $e_1'$ be an element contained in $F_1$ but not in $F_2$, and let $e_2'$ be an element contained in $F_2$ but not in $F_1$.  Then $\varphi$ can be chosen to be a tropical linear combination of $0$, $\varphi_{e_1}$, $\varphi_{e_1'}$, and $\varphi_{e_2'}$.
\item  Suppose $F_1=F_2$ and $e_1\neq e_2$.  Then $\varphi$ can be chosen to be a tropical linear combination of $0$, $\varphi_{e_1}$, and \nolinebreak $\varphi_{e_2}$.
\item  Finally, suppose $e_1\neq e_2$ and $F_1\neq F_2$, i.e., $[e_1,F_1]$ and $[e_2,F_2]$ are disjoint. If $F_1$ and $F_2$ share a common element $e$, then $\varphi$ can be chosen to be a tropical linear combination of $0$, $\varphi_{e_1}$, $\varphi_{e_2}$, and $\varphi_{e}$.  Otherwise, if $F_1$ and $F_2$ are disjoint, then there is some rank 2 flat $G$ of $W$ that contains $F_1$ and $F_2$, and $\varphi$ is a tropical linear combination of $0$, $\varphi_{e_1}$, $\varphi_{e_2}$, and $\varphi_G$.
\end{itemize}
We now show that $\Sigma$ is parametrized by $\Trop(W) \subseteq \Rbar^{H}$. For each $F\in H$, let $\delta_F$ be the indicator vector of $\{F\}$, i.e., 
\[
    \delta_F(G) = \begin{cases}
            0,& \text{ if }G=F\\
            \infty,& \text{ else}.
        \end{cases}
\]
Recall that $\gamma_F$ is the subgraph whose edges are those containing $F$. Consider the function $\psi_F \in \PL(\Gamma_M)$ given by the distance function
\[
    \psi_F(p):=\operatorname{d}(p,\gamma_F).
\]
Define a tropical linear map
$\Rbar^{H}\to \operatorname{PL}(\Gamma_M),\ \delta_F\mapsto \psi_F
$
which restricts to a tropical linear map $$\Phi\colon\Trop(W)\to\operatorname{PL}(\Gamma_M).$$ Under the map $\Phi$, the covector $H \smallsetminus H_e$ of $W$ is sent to the function $\varphi_e$ and the covector $H \smallsetminus G$ is sent to the function $\varphi_G$. Hence, $\Phi$ maps $\Trop(W)$ onto $\Sigma \subseteq R(D_M)$.

Finally, we show that the local matroid at $D_M$ is the adjoint $W$.  Recall that the ground set of $M_{\Sigma}$ consists of the subgraphs $\gamma_F$ for $F \in H$.  For every element $e \in E$, the function $\varphi_e$ obtains its minimum on $\cup_{F \notin H_e} \gamma_F$, and for a rank 2 flat $G$ of $W$ not of the form $H_e$, the function $\varphi_G$ obtains its minimum on $\cup_{F \notin G} \gamma_F$.  Thus, every rank 2 flat of $W$ is a flat of the local matroid.  These are all the rank 2 flats of $M_{\Sigma}$, because they are the minimizers of the generators of $\Sigma$. The rank 1 flats of $W$ and the rank 1 flats of $M_{\Sigma}$ agree by definition. Hence, $M_{\Sigma}=W$.
\end{proof}

\section{Further examples} \label{sec:examples}

In this section we revisit examples from \cite{LPP12, ABBR2, MUW21} that illustrate ways in which the properties of tropical complete linear series $|D|$ differ from well-known properties of linear series on algebraic curves. In each of these examples, we explain how to understand the relevant tropical linear series contained in $R(D)$ and observe that their behavior more closely reflects the properties of linear series on algebraic curves. 

\medskip

Recall that the  \emph{canonical divisor} $K_\Gamma$ on a metric graph $\Gamma$ is 
\[K_\Gamma := \sum_{v\in \Gamma}(\deg v-2)\cdot v.\]
Its rank is $r(K_\Gamma) = g-1$, where $g$ is the genus, or loop order, of $\Gamma$. By Riemann--Roch and specialization, $R(K_\Gamma)$ always contains the tropicalization of a linear series of rank $g -1$.

\begin{example}\label{ex:canonical-divisor-complete-linear-system}
Let $\Gamma$ be the genus $2$ barbell graph shown in \Cref{fig:canonical-barbell}, with canonical divisor $K_\Gamma = v + w$.  Note that the complete linear system in this example is not pure dimensional.  Divisors in $|K_\Gamma|$ that are supported on the bridge have two degrees of freedom and so lie in a 2-dimensional maximal face of $|K_\Gamma|$, whereas divisors that are supported on either of the loops lie in a 1-dimensional maximal face. In particular, $R(K_\Gamma)$ is not a tropical linear series.

\begin{figure}[ht]
    \begin{subfigure}[h]{0.40\linewidth}
        \centering
        \begin{tikzpicture}[scale=1.5]
            \draw (-0.5, 0) circle (0.5);
            \draw (0,0) -- (1,0);
            \draw (1.5, 0) circle (0.5);
            \node[vertex, label={60:$v$}] at (0, 0) {};
            \node[vertex, label={120:$w$}] at (1, 0) {};
        \end{tikzpicture}
    \end{subfigure}
    \hfill
    \begin{subfigure}[h]{0.55\linewidth}
        \centering
        \label{fig:dumbbell-rank-1}
        \begin{tikzpicture}[scale=1.2]
            \begin{scope}[shift={(0, 1.5)}, scale=0.7]
                \draw (-1, 0) circle (0.5);
                \draw (-0.5,0) -- (0.5,0);
                \draw (1, 0) circle (0.5);
                \node[divisor, label={180:1}] at (-0.5, 0) {};
                \node[divisor, label={0:1}] at (0.5, 0) {};
            \end{scope}
            \begin{scope}[shift={(-2.2, 0.2)}, scale=0.7]
                \draw (-1, 0) circle (0.5);
                \draw (-0.5,0) -- (0.5,0);
                \draw (1, 0) circle (0.5);
                \node[divisor, label={60:2}] at (-0.5, 0) {};
            \end{scope}
            \begin{scope}[shift={(2.2, 0.2)}, scale=0.7]
                \draw (-1, 0) circle (0.5);
                \draw (-0.5,0) -- (0.5,0);
                \draw (1, 0) circle (0.5);
                \node[divisor, label={120:2}] at (0.5, 0) {};
            \end{scope}
            \begin{scope}[shift={(-1.73, -1.5)}, scale=0.7]
                \draw (-1, 0) circle (0.5);
                \draw (-0.5,0) -- (0.5,0);
                \draw (1, 0) circle (0.5);
                \node[divisor, label={180:2}] at (-1.5, 0) {};
            \end{scope}
            \begin{scope}[shift={(1.73, -1.5)}, scale=0.7]
                \draw (-1, 0) circle (0.5);
                \draw (-0.5,0) -- (0.5,0);
                \draw (1, 0) circle (0.5);
                \node[divisor, label={0:2}] at (1.5, 0) {};
            \end{scope}
            \draw[thick, pattern=north west lines] (90:1) -- (210:1) -- (330:1) -- cycle;
            \draw[thick] (210:1) -- (210:2);
            \draw[thick] (330:1) -- (330:2);

            \draw[dotted] (90:1)--(90:1.5);
            \draw[dotted] (210:1)--(-1.25,-0.05);
            \draw[dotted] (330:1)--(1.25, -0.05);
            \draw[dotted] (210:2)--(-1.73, -1.5);
            \draw[dotted] (330:2)--(1.73,-1.5);
            
            \node[vertex] at (90:1) {};
            \node[vertex] at (210:1) {};
            \node[vertex] at (330:1) {};
            \node[vertex] at (210:2) {};
            \node[vertex] at (330:2) {};

        \end{tikzpicture}
    \end{subfigure}

    \caption{The barbell graph and the complete linear system of its canonical divisor $K_\Gamma = v + w$. The divisor corresponding to each vertex of the complete linear system is depicted adjacent to that vertex.}
    \label{fig:canonical-barbell}
\end{figure}
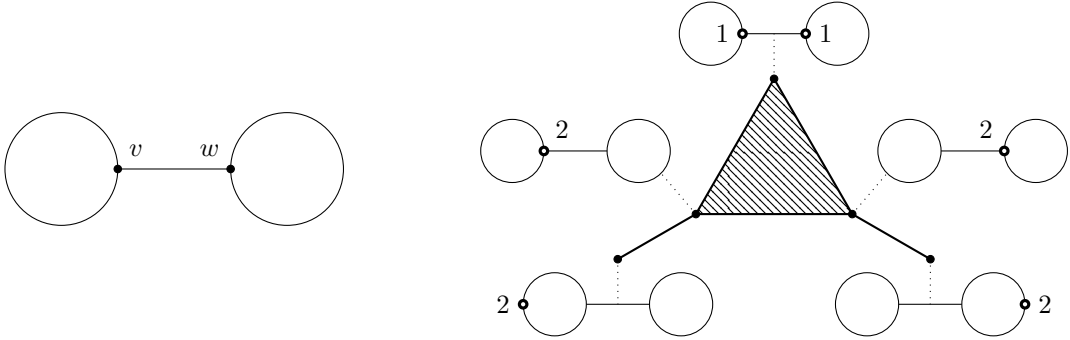

We claim that $R(K_{\Gamma})$ contains a unique tropical linear series $\Sigma$ of dimension 1. %
Let $x$ be the point of $\Gamma$ opposite $v$ on the left loop. There is a unique divisor $D_x \in |K_\Gamma|$ such that $\supp(D_x)$ contains $x$.  Choose $\varphi_x$ such that $D_x = K_\Gamma + \ddiv \varphi_x$.  Similarly, choose $\varphi_y$ so that $D_y := K_\Gamma + \ddiv(\varphi_y)$ is the unique divisor in $|K_\Gamma|$ whose support contains the point $y$ on the right loop opposite $w$. 
Since $\Sigma$ has Baker-Norine rank 1, $\varphi_x , \varphi_y \in \Sigma$.  Now, any tropical linear series of dimension $1$ contained in $R(K_\Gamma)$ must contain $\varphi_x$ and $\varphi_y$.  However, one readily checks that 
$\langle \varphi_x , \varphi_y \rangle$ is itself a tropical linear series of dimension 1. By \Cref{prop:maximality-tls}, it follows that $\langle \varphi_x, \varphi_y\rangle$ is the unique tropical linear series of dimension 1 in $R(K_\Gamma)$. Note that $|\langle \varphi_x, \varphi_y \rangle|$ is the realizability locus for canonical divisors determined in \cite[Example~6.4]{MUW21}, and does not contain $K_\Gamma$ (cf. \cite[Example~4.4]{ABBR2}).

\end{example}

We now consider another example of a divisor $D$ with $r(D) = 1$ that is not the tropicalization of any divisor of rank $1$. We revisit an example, attributed to Ye Luo in \cite[Example~5.13]{ABBR2}. Here, we give a different proof for the non-realizability of $D$, by showing that $R(D)$ does not contain \emph{any} tropical linear series of positive dimension.  %

\begin{example}
\label{Ex:Harmonic}
Consider the divisor $D = p+q+s$ on the graph $\Gamma$ pictured in \Cref{Fig:NonRealize}, where all edge lengths are equal. Note that this divisor has rank $r(D) = 1$. In \cite[Example~5.13]{ABBR2} it was proven that $D$ cannot be the tropicalization of a divisor of positive rank on an algebraic curve, by showing that there is no degree 3 harmonic morphism from a modification of $\Gamma$ to a tree.

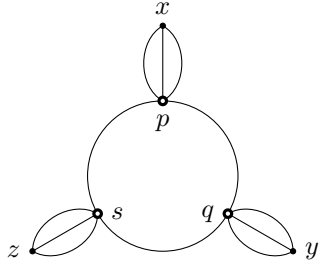
\begin{figure}[ht]
\begin{tikzpicture}

\draw (0,0) circle (1);

\draw (90:1)--(90:2);
\draw (90:1)to[bend right=60](90:2);
\draw (90:1)to[bend left=60](90:2);

\draw (210:1)--(210:2);
\draw (210:1)to[bend right=60](210:2);
\draw (210:1)to[bend left=60](210:2);

\draw (330:1)--(330:2);
\draw (330:1)to[bend right=60](330:2);
\draw (330:1)to[bend left=60](330:2);

\node[divisor, label={270:$p$}] at (90:1) {};
\node[vertex,scale=0.7,label={90:$x$}] at (90:2) {};
\node[divisor, label={0:$s$}] at (210:1) {};
\node[vertex,scale=0.7, label={180:$z$}] at (210:2) {};
\node[divisor, label={180:$q$}] at (330:1) {};
\node[vertex,scale=0.7, label={0:$y$}] at (330:2) {};

\end{tikzpicture}
\caption{Luo's example of a non-realizable divisor of positive rank.}
\label{Fig:NonRealize}
\end{figure}

We claim that $R(D)$ does not contain a tropical linear series of dimension 1.  %
Note that $|D|$ contains unique divisors 
\[
D_x = D + \ddiv(\varphi_x), \quad D_y = D + \ddiv(\varphi_y), \quad \mbox{and} \quad D_z = D = \ddiv(\varphi_z)
\]
whose support contains $x$, $y$, and $z$, respectively. 
Thus, any tropical submodule of $R(D)$ with rank $r_\BN(\Sigma) = 1$ must contain $\{ \varphi_x, \varphi_y, \varphi_z \}$, which is tropically independent because these three functions have distinct slopes along $[s,q]$. 
\end{example}

The argument in \Cref{Ex:Harmonic} illustrates a general fact that we record here for future reference. %

\begin{proposition}
\label{Lem:Slopes}
Let $\Sigma \subseteq R(D)$ be a tropical linear series of dimension $r$.  For each tangent vector $\zeta$, the set of slopes $s_\zeta(\Sigma)$ has size exactly $r+1$.
\end{proposition}

\begin{proof}
Suppose $\zeta$ is a tangent vector based at $v \in \Gamma$, and let $I \subseteq \Gamma$ be a closed interval with one endpoint at $v$ that contains $\zeta$.  By choosing $I$ sufficiently small, we may assume that $I \smallsetminus \{ v \}$ does not intersect the support of $D$, and each of the functions in a minimal generating set for $\Sigma$ is linear on $I$. Since $\Sigma|_I$ is a tropical linear series of dimension $r$, it follows that these generators have exactly $r+1$ distinct slopes on $I$, as required. 
\end{proof}

We now revisit an example of a graph of genus 4 that is not hyperelliptic but has infinitely many divisor classes of degree 3 and rank 1.  This contrasts with the classical fact that a non-hyperelliptic curve of genus 4 has either 1 or 2 divisor classes of degree 3 and rank 1.

\begin{example}
\label{Ex:LoopOfLoops}
Consider the genus 4 loop of loops $\Gamma$, pictured in \Cref{Fig:LOL}. Let $\ell_i$ be the length of the edge $[v_i, w_i]$. If $\ell_1 > \ell_2 > \ell_3$ and $\ell_2 + \ell_3 > \ell_1$, then $\Gamma$ has infinitely many divisor classes of degree 3 and rank 1, and no divisors of degree 2 and rank 1 \cite[Theorem~1.2]{LPP12}.  %
However, we claim that there is a unique divisor class $[D]$ on $\Gamma$ of rank 1 and degree 3 with the property that $R(D)$ contains a rank 1 tropical linear series and, moreover, this tropical linear series is unique. %

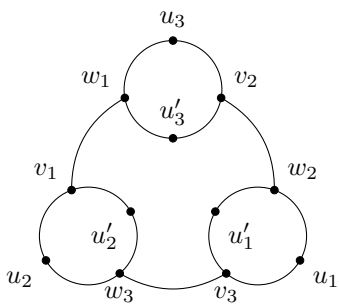
\begin{figure}[ht]
\begin{tikzpicture}[scale=1.3]
    \coordinate (w1) at ($(90:1) + (190:0.5)$);
    \coordinate (v1) at ($(210:1) + (110:0.5)$);
    \coordinate (w3) at ($(210:1) + (310:0.5)$);
    \coordinate (v3) at ($(330:1) + (230:0.5)$);
    \coordinate (w2) at ($(330:1) + (70:0.5)$);
    \coordinate (v2) at ($(90:1) + (350:0.5)$);
    
    \draw[fill=white] (90:1) circle (0.5);
    \draw[fill=white] (210:1) circle (0.5);
    \draw[fill=white] (330:1) circle (0.5);
    \draw (w1) to[bend right] (v1);
    \draw (w2) to[bend right] (v2);
    \draw (w3) to[bend right] (v3);

    \node[vertex, label={160:$v_1$}] at (v1) {};
    \node[vertex, label={160:$w_1$}] at (w1) {};
    \node[vertex, label={20:$v_2$}] at (v2) {};
    \node[vertex, label={20:$w_2$}] at (w2) {};
    \node[vertex, label={270:$v_3$}] at (v3) {};
    \node[vertex, label={270:$w_3$}] at (w3) {};
    
    \node[vertex, label={90:$u_3$}] at (90:1.5) {};
    \node[vertex, label={210:$u_2$}] at (210:1.5) {};
    \node[vertex, label={330:$u_1$}] at (330:1.5) {};
    \node[vertex, label={90:$u'_3$}] at (90:0.5) {};
    \node[vertex, label={210:$u'_2$}] at (210:0.5) {};
    \node[vertex, label={330:$u'_1$}] at (330:0.5) {};
    
\end{tikzpicture}
    \caption{The loop of loops of genus 4.}
    \label{Fig:LOL}
\end{figure}

Let $D$ be a divisor of the form $v_1 + w_3 + w$, where $w$ is a point on the edge $[v_2, w_2]$.  Let $x$ be the distance from $w$ to $v_2$, and assume $x \geq \ell_1 - \ell_2$.  One can show using Dhar's burning algorithm that every divisor of degree 3 and rank 1 on $\Gamma$ is equivalent to a divisor of this form. We omit the cumbersome case analysis. 
Observe that $|D|$ contains unique divisors
\[
D_1 = D + \ddiv(\varphi_1), \quad D_2 = D + \ddiv(\varphi_2), \quad \mbox{and} \quad D_3 = D + \ddiv(\varphi_3)
\]
such that $\supp(D_i)$ contains $u_i$, as shown in \Cref{Fig:FunctionsOnLOL}.
Suppose $\Sigma \subseteq R(D)$ is a tropical linear series of rank $1$.  Then $\Sigma$ must contain $\{ \varphi_1 , \varphi_2 , \varphi_3 \}$, and there must be a tropical dependence
\begin{equation} \label{eq:phi-depend}
\theta = \min\{ \varphi_1 + a_1, \varphi_2 + a_2, \varphi_3 + a_3 \}.
\end{equation}  
We claim that the existence of such a  tropical dependence determines the point $x$ in $[v_2, w_2]$.   

Consider the point on $[v_3, w_3]$ in the support of $D_3$, which has distance $\ell_1 - x$ from $w_3$.  No two of the functions $\varphi_i$ differ by a constant in a neighborhood of this point, so in the tropical dependence, all three functions must achieve the minimum in \eqref{eq:phi-depend} at this point.  Similarly, the point on $[v_1, w_1]$ in the support of $D_1$ has distance $\ell_3 - \ell_2 + x$ from $v_1$, and all three functions must achieve the minimum in \eqref{eq:phi-depend} at this point. The tropical dependence \eqref{eq:phi-depend} is illustrated in \Cref{Fig:LOLDependence}, with the loops labeled by the functions that achieve the minimum on them.  Since the function $\varphi_2$ takes the same value at these two points, it follows that $\ell_1 - x = \ell_3 - \ell_2 +x$.  Equivalently, $x = \frac{1}{2}(\ell_1 + \ell_2 - \ell_3)$.

Now, suppose $x = \frac{1}{2}(\ell_1 + \ell_2 - \ell_3)$. We show that $R(D)$ contains a unique tropical linear series of rank $1$. Let $\varphi_1$, $\varphi_2$, and $\varphi_3$ be as above. Note that there are unique divisors
\[
D'_1 = D + \ddiv(\varphi'_1), \quad D'_2 = D + \ddiv(\varphi'_2), \quad \mbox{and} \quad D'_3 = D + \ddiv(\varphi'_3)
\]
such that the support of $D'_i$ contains $u'_i$. The tropical submodule $\Sigma = \langle \varphi_1, \varphi'_1, \varphi_2, \varphi'_2, \varphi_3, \varphi'_3\rangle$ has Baker--Norine rank $r_\BN(\Sigma) = 1$. Moreover, any set of three functions from this generating set is tropically dependent. By \Cref{Lem:GeneratorDependence}, it follows that $r_\ind(\Sigma) = 2$, and hence $\Sigma$ is a tropical linear series of dimension $1$.  Any other tropical linear series of rank 1 in $R(D)$ must contain $\Sigma$. By     \Cref{prop:maximality-tls}, $\Sigma$ is the unique tropical linear series inside of $R(D)$.

\begin{figure}[ht]
\begin{tikzpicture}

\begin{scope}
    \coordinate (w1) at ($(90:1) + (190:0.5)$);
    \coordinate (v1) at ($(210:1) + (110:0.5)$);
    \coordinate (w3) at ($(210:1) + (310:0.5)$);
    \coordinate (v3) at ($(330:1) + (230:0.5)$);
    \coordinate (w2) at ($(330:1) + (70:0.5)$);
    \coordinate (v2) at ($(90:1) + (350:0.5)$);
    \draw[fill=white] (90:1) circle (0.5);
    \draw[fill=white] (210:1) circle (0.5);
    \draw[fill=white] (330:1) circle (0.5);
    
    \draw (w1) to[bend right]
        node[divisor, pos=0.7, label={150:1}] {} (v1);
    \draw (w2) to[bend right] (v2);
    \draw (w3) to[bend right] (v3);

    \node[divisor, label={330:$2$}] at (330:1.5) {};
    \node[label={$D_1$}] at (0, -2) {};
\end{scope}

\begin{scope}[shift={(4, 0)}]
    \coordinate (w1) at ($(90:1) + (190:0.5)$);
    \coordinate (v1) at ($(210:1) + (110:0.5)$);
    \coordinate (w3) at ($(210:1) + (310:0.5)$);
    \coordinate (v3) at ($(330:1) + (230:0.5)$);
    \coordinate (w2) at ($(330:1) + (70:0.5)$);
    \coordinate (v2) at ($(90:1) + (350:0.5)$);
    \draw[fill=white] (90:1) circle (0.5);
    \draw[fill=white] (210:1) circle (0.5);
    \draw[fill=white] (330:1) circle (0.5);
    
    \draw (w1) to[bend right] (v1);
    \draw (w2) to[bend right]
        node[divisor, pos=0.7, label={30:1}] {} 
        (v2);
    \draw (w3) to[bend right] (v3);

    \node[divisor, label={210:$2$}] at (210:1.5) {};

    \node[label={$D_2$}] at (0, -2) {};
\end{scope}

\begin{scope}[shift={(8, 0)}]
    \coordinate (w1) at ($(90:1) + (190:0.5)$);
    \coordinate (v1) at ($(210:1) + (110:0.5)$);
    \coordinate (w3) at ($(210:1) + (310:0.5)$);
    \coordinate (v3) at ($(330:1) + (230:0.5)$);
    \coordinate (w2) at ($(330:1) + (70:0.5)$);
    \coordinate (v2) at ($(90:1) + (350:0.5)$);
    \draw[fill=white] (90:1) circle (0.5);
    \draw[fill=white] (210:1) circle (0.5);
    \draw[fill=white] (330:1) circle (0.5);
    
    \draw (w1) to[bend right] (v1);
    \draw (w2) to[bend right] (v2);
    \draw (w3) to[bend right] 
        node[divisor, pos=0.3, label={60:1}] {} 
        (v3);

    \node[divisor, label={90:$2$}] at (90:1.5) {};
    \node[label={$D_3$}] at (0, -2) {};
\end{scope}

\end{tikzpicture}
\caption{The divisors corresponding to the functions $\varphi_1 , \varphi_2,$ and $\varphi_3$ of \Cref{Ex:LoopOfLoops}.}
\label{Fig:FunctionsOnLOL}
\end{figure}
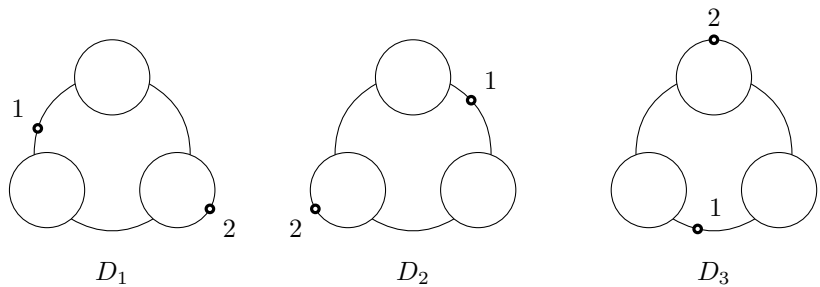

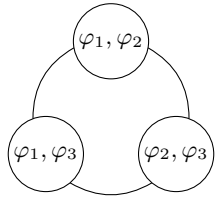
\begin{figure}[ht]

\begin{tikzpicture}
    \coordinate (w1) at ($(90:1) + (190:0.5)$);
    \coordinate (v1) at ($(210:1) + (110:0.5)$);
    \coordinate (w3) at ($(210:1) + (310:0.5)$);
    \coordinate (v3) at ($(330:1) + (230:0.5)$);
    \coordinate (w2) at ($(330:1) + (70:0.5)$);
    \coordinate (v2) at ($(90:1) + (350:0.5)$);
    
    \draw[fill=white] (90:1) circle (0.5);
    \draw[fill=white] (210:1) circle (0.5);
    \draw[fill=white] (330:1) circle (0.5);

    \draw (90:1) node {{\small $\varphi_1, \varphi_2$}};
    \draw (210:1) node {{\small $\varphi_1, \varphi_3$}};
    \draw (330:1) node {{\small $\varphi_2, \varphi_3$}};
    
    \draw (w1) to[bend right] (v1);
    \draw (w2) to[bend right] (v2);
    \draw (w3) to[bend right] (v3);
\end{tikzpicture}

\caption{The dependence between $\varphi_1, \varphi_2,$ and $\varphi_3$.}
\label{Fig:LOLDependence}
\end{figure}

\end{example}

\section{Open problems and questions}
\label{Sec:Questions}

In the introduction we posed \Cref{q:tropical-implies-matroidal}, asking whether every tropical linear series is a matroidal linear series. The answer is affirmative if $\Sigma\subseteq R(D)$ is a tropical linear series of dimension one where $D$ is a divisor of degree two on an interval $\Gamma=[v,w]$. Even in this case, some nontrivial calculation is already needed \cite[Example~6.10]{M23}, and we can show that $\Sigma$ arises in the manner described in \Cref{thm:intervalProof}. 
We pose the following variant for tropical linear series on an interval.
\begin{question}\label{qst:TLS-interval}
Does every tropical linear series on an interval or loop $\Gamma$ arise from the constructions in \Cref{thm:intervalProof} and \Cref{Thm:loop}, respectively?%
\end{question}

\medskip

In \Cref{rem:BN-bound}, we explained that any metric graph $\Gamma$ of first Betti number $g$ has a divisor $D$ of degree $d$ with a tropical linear series $\Sigma \subseteq R(D)$ of dimension $r$, whenever $(r+1)(g-d+r) \leq g$, and how this follows from the Brill--Noether theorem. Similarly, it follows from the Riemann--Roch inequality that there is a tropical linear series $\Sigma \subseteq R(D)$ of dimension $r$, whenever $r \leq d - g$, for \emph{any} divisor $D$ of degree $d$. It would be interesting to find a proof of any such statement about the existence of positive dimensional tropical linear series that does not depend on algebraic geometry.

\begin{problem}
Give a direct combinatorial proof of existence theorems for tropical linear series that follow from the Riemann--Roch and Brill--Noether theorems.
\end{problem}

\noindent An interesting special case would be to give a combinatorial proof of the existence of a tropical linear series of dimension $g-1$ contained in $R(K_\Gamma)$.

The Baker--Norine rank $r(D)$ for a divisor on a graph famously satisfes the precise analog of the Riemann--Roch theorem \cite{BakerNorine07} as well as the specialization inequality \cite{Baker08}. Subsequently, other rank functions for divisors on graphs were discovered that satisfy their own Riemann--Roch theorems and strengthen the specialization inequality \cite{CaporasoLenMelo15, BarbosaChristMelo25}.  

\begin{question}
Is there a rank function defined in terms of tropical or matroidal linear series that satisfies an analog of the Riemann--Roch theorem?
\end{question}

\noindent One might consider, for instance, the function taking a divisor $D$ to the largest dimension of a tropical or matroidal linear series contained in $R(D)$, which is smaller than $r(D)$ yet still satisfies the specialization inequality.

\bigskip

As discussed in the introduction, many advances in this research area were spurred by relations to the Brill--Noether theory of algebraic curves, following the program initiated by Baker in \cite{Baker08}. In Brill--Noether theory, one studies the moduli space $G^r_d(X)$ of linear series of degree $d$ and dimension $r$ on a general curve $X$, along with its image $W^r_d(X)$ in $\Pic_d(X)$. For metric graphs, a tropical analog $W^r_d(\Gamma) \subseteq \Pic_d(\Gamma)$ is studied in \cite{LPP12}. It is, in particular, a polyhedral complex with a well-defined dimension and a \emph{rank}, analogous to the Baker--Norine rank. %
However, there is not yet any tropical analog of $G^r_d(X)$, because we did not previously have a suitable analog of incomplete linear series, i.e., proper subspaces of $H^0(X, \cO(D_X))$.

\begin{problem}
Construct a moduli space of tropical or matroidal linear series on a metric graph $\Gamma$ and study its geometric properties, including dimension and analogs of the Baker--Norine rank.
\end{problem}

\noindent The following special case is related to the Riemann--Roch questions, above.

\begin{problem} \label{prob:canonical}
Let $\Gamma$ be a metric graph of genus at least 2 with no 1-valent vertices.  Classify the tropical linear series of rank $g-1$ that are contained in $R(K_\Gamma)$.
\end{problem}

\noindent It follows from \Cref{cor:strong-maximality-tls} that $R(K_\Gamma)$ is the unique such tropical linear series whenever $|K_\Gamma|$ has dimension $g-1$. Example~\ref{ex:canonical} is a  first step towards such a classification in the remaining cases, by determining the local matroid at $K_\Gamma$. The ideas and results of \Cref{Sec:Canonical} are also relevant, but many subtleties remain. There may be tropical linear series that are realizable in residue characteristic $p > 0$, but not in equicharacteristic zero, and these may or may not be contained in the equicharacteristic zero realizable locus determined by M\"oller, Ulirsch, and Werner. Also, aside from the local structure at $K_\Gamma$, the possibilities are wide open when the realizable locus has dimension greater than $g-1$, as in \Cref{ex:hyperelliptic-chain}.

Our interest in tropical linear series was partly motivated by the study of multiplication maps via tropical independence in \cite{tropicalGP, JensenPayne16, MRC2, M13, M23}. Here, the multiplication for linear series $(D_X, V)$ and $(D'_X, V')$ is the linear map
\[
V \otimes V' \to H^0(X, \cO(D_X + D'_X))
\]
induced by tensor product of global sections or, equivalently, multiplication of rational functions.

Many fundamental questions about the geometry of algebraic curves can be expressed in terms of ranks of maps between linear series, and one can give lower bounds on these ranks by proving tropical independence of collections of functions in the image. %
Thus, there is significant interest in understanding how multipliciation maps behave for tropical and matroidal linear series.

\begin{question} \label{q:multiplication-map}
Given tropical linear series $\Sigma_1 \subseteq R(D_1)$ and $\Sigma_2 \subseteq R(D_2)$ on a metric graph $\Gamma$, define the multiplication map $\mu \colon \Sigma_1 \times \Sigma_2 \to R(D_1 + D_2)$ by $\mu (\varphi_1 , \varphi_2) = \varphi_1 + \varphi_2$.
 Is there a tropical linear series $\Sigma \subseteq R(D_1 + D_2)$  that contains the image of $\mu$?
\end{question}

When $\Sigma_1$ and $\Sigma_2$ are tropicalizations of linear series, the answer is affirmative, and the tropical linear series can be chosen with dimension at most $r_\ind(\Sigma_1)r_\ind(\Sigma_2) -1$, because the image of the tropical multiplication map is contained in the tropicalization of the image of the classical multiplication map. However, the answer is unclear when $\Sigma_1$ or $\Sigma_2$ is not realizable.

\Cref{q:multiplication-map} is an important special case of the more general problem of finding a tropical linear series that contains a given subset or tropical submodule of $R(D)$.

\begin{question}
Is there a procedure that determines whether a tropical module $\Sigma \subseteq R(D)$ is contained in a tropical linear series of a given rank?
\end{question}

The main combinatorial results of \cite{M23} include a classification of strongly recursive tropical linear series on a chain of $g$ loops with bridges $\Gamma$ with edge lengths that satisfy a specified admissibility condition and a proof that if  
$g = 22$ or $23$ and $\Sigma$ is a strongly recursive tropical linear series of dimension 6 and degree 25 or 26, respectively, then the image of the multiplication map
\[
\mu \colon \Sigma \times \Sigma \to \PL(\Gamma) 
\]
contains a tropically independent subset of size $28$.  In particular, the image cannot be contained in any tropical linear series of dimension less than 27. The corresponding cases of the strong maximal rank conjecture of Aprodu and Farkas follow immediately, and this was an essential step in the proof that the moduli spaces $\overline M_{22}$ and $\overline M_{23}$ are of general type.

\begin{question}
Let $\Gamma$ be a chain of $22$ or $23$ loops with bridges with admissible edge lengths, and let $\Sigma \subseteq \PL(\Gamma)$ be a tropical linear series of dimension $6$ and degree $25$ or $26$ respectively. Does every tropical linear series that contains the image of $\mu \colon \Sigma \times \Sigma \to \PL(\Gamma)$ have dimension at least 27?
\end{question}

We have seen that $\Star(D)$ is the Bergman fan of the local matroid $M_{\Sigma}$ when $D \in |\Sigma|$ and $\Sigma \subseteq R(D)$ is nondegenerate.

\begin{question}
    What are the possibilities for the local structure of $|\Sigma|$ at a degenerate divisor?
\end{question}

For any divisor $D\in |\Sigma|$, the set $\cL = \{F_\varphi:\varphi\in \Sigma\} \cup \{E\}$ is a lattice. If $D$ is nondegenerate, this is the lattice of flats of a matroid that controls the local structure of $|\Sigma|$ at $D$. In general, there is a surjective map from a small neighborhood of $D$ in $|\Sigma|$ to $\cL$, but essential information may be lost if $D$ is degenerate. For instance, let $p$ be the midpoint of the interval $\Gamma=[v,w]$ and $D=2p$. The tropical module $R(D)$ contains the functions $\varphi_1$ and $\varphi_2$ where $\ddiv\varphi_1+D=p+w$ and $\ddiv\varphi_2+D=2w$. Then $\varphi_1$ and $\varphi_2$ attain their minimum at the same place, so they are mapped to the same element in $\cL$. However, for any $\epsilon>0$, $\min\{\varphi_1,\epsilon\}\neq \min\{\varphi_2,\epsilon\}$. They are not determined by their image in $\cL$, in contrast to \Cref{thm:local-structure}. If there is a nice way to describe the local structure of $\Sigma$ at a degenerate divisor, then one might hope to `glue' all the local pictures together and get a global parametrization and answer \Cref{qst:TLS-interval}.

\begin{question}
Let $\Sigma \subseteq R(D)$ be a matroidal linear series with parametrization $\cV \twoheadrightarrow \Sigma$, and let $\pi \colon \cV \to |D|$ denote the composition with projectivization. What are the properties of the submodules $\pi^{-1}(D)$ for $D \in |\Sigma|$, and how do they differ depending on whether or not $D$ is degenerate?
\end{question}

\appendix
\section{Tropicalizations of linear series are matroidal} \label{sec:realizablematroidal}

In this Appendix, we give a brief account of tropicalizations of linear series on algebraic curves. In particular, we give a self-contained proof that tropicalizations of linear series are matroidal.  This argument first appeared in the proof that tropicalizations of linear series are finitely generated as tropical modules \cite[Proposition~6.4]{M23}. The material in this appendix is not logically necessary for the main results of this paper, but it is an essential motivation for investigating the relations between tropical linear series and matroids. 

\subsection{Skeletons of semi-stable models} Let $X$ be a smooth projective algebraic curve over a field $K$ with a nontrivial valuation $\val \colon K \to \Rbar$.  For simplicity, we assume $K$ is algebraically closed and spherically complete, and $\val$ is surjective. With these assumptions, the Berkovich analytification $X^\an$ is set-theoretically a disjoint union of two subsets:
\begin{itemize}
\item the type I points $X(K)$, and
\item the type II points, which are the extensions of $\val$ to valuations on the function field $K(X)$.
\end{itemize}
Let $R \subseteq K$ be the valuation ring. By the semistable reduction theorem, there are models $\fX \to \Spec R$ with generic fiber $\fX_\eta \cong X$ and special fiber $\overline \fX$ a nodal curve. %
Formally locally near each node, the total space of $\fX$ looks like $\VV(xy-t) \subseteq \AA^2_R$ for some $t$ in the maximal ideal of $R$, and $\val(f)$ is independent of the choice of local coordinates.  The corresponding \emph{skeleton} $\Gamma_{\fX}$ is a metric graph with vertices and edges corresponding to the irreducible components and nodes of $\overline \fX$, respectively. The length of the edge corresponding to a node with local defining equation $xy-t$ is $\val(t)$.

There is a natural inclusion $\Gamma_\fX \subseteq X^\an$ with image contained in the subset of type II points, and a natural retraction $X^\an \to \Gamma_\fX$.  A formal blowup $\fX' \to \fX$ induces a continuous, surjective map $\Gamma_{\fX'} \to \Gamma_{\fX}$ compatible with the retraction maps from $X^\an$, and the induced map
\[
X^\an \xrightarrow{\sim} \varprojlim_{\fX} \Gamma_{\fX}
\]
is a homeomorphism.  

Let $\Gamma = \Gamma_{\fX}$ be such a skeleton. When the skeleton is fixed, the retraction $X^\an \to \Gamma$ is called the \emph{tropicalization} map. Since each point in $\Gamma$ represents a valuation on the function field $K(X)$, for any nonzero rational function $f\in K(X)$ we obtain a function denoted $\trop(f)$ on $\Gamma$ by evaluation. The function $\trop(f)$ is continuous and piecewise-linear, with integer slope on each edge. 

The bend-locus of such a piecewise-linear function is the analog of the zeros and poles of a rational function, with the analog of order of vanishing given by the sum of the incoming slopes at a point. These notions are compatible with tropicalization, i.e., the sum of the incoming slopes of $\trop(f)$ at a point is the sum of the multiplicities of the zeros and poles of $f$ that tropicalize to that point. See, e.g., \cite[Remark~5.4]{BakerRabinoff15}. 

\subsection{Tropicalizations of linear series} \label{subsec:tropicalizations-linser}

With $X$ and $\Gamma = \Gamma_\fX$ as above, consider a linear series $V \subseteq H^0(X, \cO(D_X))$. We identify nonzero sections of $\cO(D_X)$ with rational functions $f \in K(X)$ such that $\ddiv(f) + D_X \geq 0$. Then $D_X$ tropicalizes to a divisor $D$ on $\Gamma$, and $V$ tropicalizes to 
\[
\Sigma = \{ \trop(f) : f \in V \smallsetminus \{0\}\}.
\]

\begin{proposition} \label{prop:matroidparam}
Let $(D,\Sigma)$ be the tropicalization of a linear series $(D_X, V)$ of dimension $r$. Then $\Sigma$ is a tropical linear series and is the image of a surjective homomorphism of tropical modules $\cV \twoheadrightarrow \Sigma$, where $\cV$ is a realizable valuated matroid of rank $r + 1$.
\end{proposition}

\noindent This fact first appeared as a step in the proof that the tropicalization of a linear series is finitely generated as a tropical module \cite[Proposition~6.4]{M23}. Since it was not stated as a separate proposition in loc. cit., and since it is essential motivation for studying relations between tropical linear series and valuated matroids, we present it here in this appendix with a self-contained proof. Our proof will use the decomposition of $X^\an$ into nonarchimedean analytic balls and annuli induced by the semistable model $\fX$.

\subsection{Semistable decompositions}

We now recall the decomposition of $X^\an$ into nonarchimedean analytic discs and annuli induced by the semistable model $\fX$, following \cite[\S3]{BPR13}.

The skeleton $\Gamma = \Gamma_{\fX}$ comes with a distinguished vertex set $W \subseteq \Gamma$ corresponding to the irreducible components of the special fiber $\overline \fX$. The complement $X^\an \smallsetminus W_\fX$ is the  disjoint union of its connected components. These connected components consist of infinitely many nonarchimedean analytic discs, one for each smooth closed point of $\overline \fX$, and finitely many nonarchimedean analytic annuli, one for each node of $\overline \fX$.  This decomposition is compatible with tropicalization, i.e., each smooth closed point of $\overline \fX$ is contained in a unique irreducible component of $\fX$, and the corresponding disc tropicalizes to that vertex in $W_{\fX}$.  Similarly, the nodes of $\overline \fX$ correspond to the edges of $\Gamma$, and the annulus corresponding to a node of $\fX$ is exactly the preimage under tropicalization of the interior of the corresponding edge. We note also that the length of an edge is the logarithmic modulus of the corresponding annulus.

Furthermore, any finite set of type II points that contains the vertex set $W_\fX$ is the vertex set of some semistable model obtained as a formal blowup $\fX' \to \fX$. In particular, any subdivision of $\Gamma$ is the skeleton associated to a semistable model of $\fX$.

\subsection{Proof of \Cref{prop:matroidparam}}

In the introduction, we have discussed the fact that the tropicalization of a linear series is a tropical linear series of the same dimension.  It remains to show that this tropical linear series is matroidal, i.e., there is a valuated matroid $\cV$ of rank $r + 1$ and a surjective homomorphism of tropical modules $\cV \twoheadrightarrow \Sigma$.  We do so as follows.

First, since any subdivision of $\Gamma$ is the skeleton of some formal model $\fX' \to \fX$, we may assume that the vertex set $W$ contains $\trop(\supp(D_X))$.  Next, by continuity, the tropicalization of any nonzero rational function $f \in K(X)^\times$ is determined by its restriction to $\Gamma \smallsetminus W$, i.e., the union of the open edges of $\Gamma$. 

Let $e \subseteq \Gamma$ be an open edge of length $\ell$.  Then $\trop^{-1}(e)$ is analytically isomorphic to the standard open annulus $U_\ell$ of logarithmic modulus $\ell$, i.e., the preimage of $(0,\ell)$ under the standard tropicalization map $\GG_m^\an \to \RR$.  

Since the vertex set $W$ contains the tropicalization of $\supp(D_X)$, any $f \in H^0(X, \cO(D_X))$ is regular on $\trop^{-1}(e) \cong U_\ell$.  Let $t$ be the coordinate on $U_\ell \subseteq \GG_m^\an$.  Then any regular function on $U_\ell$ has a Laurent series expansion
\[
f = \sum_{n=-\infty}^\infty \alpha_n t^n,
\]
with $\alpha_i \in K$.  The convergence condition on $U_\ell$ says that $\lim \val (\alpha_n) = \lim (\ell n + \val(\alpha_n)) = \infty$, i.e., for any $N > 0$, we have $\val(\alpha_n) > N$ and $\ell n + \val(\alpha_n) > N$ for all but finitely many integers $n$.

In terms of this power series expansion, and the identification $e \cong (0,\ell)$, we have
\[
\trop(f)|_{e}(x) = \min_{n \in \ZZ} \{n x + \val(\alpha_n) \}, \mbox{ for } x \in (0,\ell).
\]
This expresses $\trop(f)|_e$ as the minimum of a countable collection of affine linear functions with distinct integer slopes.  However, the slope of any function in the complete tropical linear system $R(D)$ is bounded in absolute value by $d = \deg(D)$ \cite[Lemma~7]{HMY12}. Thus, we have 
\[
\trop(f)|_{e}(x) = \min_{|n| \leq d} \{n x + \val(\alpha_n) \}.
\]

We have shown that $\trop(f)|_e$ is determined by the valuations of the $2d+1$ coefficients $\alpha_{-d}, \ldots, \alpha_d$ in the corresponding power series expansion of $f$ on $U_\ell$, for $\ell = \ell(e)$. Say $\Gamma$ has $s$ edges.  Applying this argument to each edge, we see that $\trop(f)$ is globally determined by the valuations of $(2d+1)s$ coefficients in the $s$ corresponding Laurent series expansions. Altogether, these coefficients give a linear embedding $V \subseteq K^{(2d+1)s}$. Let $\cV$ be the image of $V$ under the tropicalization map to $\Rbar^{(2d+1)s}$.  By construction, $\cV$ is a realizable valuated matroid of rank $r + 1$ and the tropicalization map on $V$ factors as  
\begin{center}
\begin{tikzcd}
V \arrow[rd, rightarrow, "\trop"'] \arrow[r] & \cV \arrow[d] \\
{}  & \Sigma.
\end{tikzcd}
\end{center}
The map $\cV \to \Sigma$ is a homomorphism of tropical modules, by construction, and the commutativity of the diagram shows that it is surjective, as required.
\hfill $\Box$

\bibliography{TLS}

\end{document}